\documentclass[11pt,notitlepage]{amsart}
\usepackage{latexsym,amsfonts,amssymb,amsmath,amsthm} %{pstricks,pst-node,epsf}
\newcommand{\D}{\ensuremath{\mathbb{D}_{\mathbf{K}}}}

\newcommand{\K}{\ensuremath{\mathbf{K}}}
\newcommand{\N}{\ensuremath{\mathbb N}}
\newcommand{\T}{\ensuremath{Tr_{\mathbf{K}/\mathbf{Q}}}}
\newcommand{\Z}{\ensuremath{\mathbb Z}}
\newcommand{\R}{\ensuremath{\mathbb R}}

\newcommand{\Q}{\ensuremath{\mathbb Q}}
\newcommand{\C}{\ensuremath{\mathbb C}}

\newcommand\GG{{\mathbf G}}

\newcommand\Okr{{\mathcal O}}

\newcommand{\ec}{\textrm}

\newcommand{\F}{\ensuremath{\mathbf F}}

\newcommand\QQ{\mathbb Q}
\newcommand\RR{\mathbb R}

\newcommand{\nn}{\mathtt{n}}

\newcommand\G{\Gamma}

\newcommand\ld{\lambda}

\newcommand\n{\nu}
\newcommand\x{\xi}

\newcommand\irr{\Pi}

\newcommand{\mm}{\mathit{j}}

\newcommand{\ri}{\ensuremath{O_{\mathbf{K}}}} % ring of integers!

\newcommand{\xbar}{\overline{x}}
\newcommand{\ybar}{\overline{y}}
\newcommand{\pa}{\boldsymbol{p}_1}
\newcommand{\pb}{\boldsymbol{p}_2}
\newcommand{\isdef}{:=}
\theoremstyle{plain}		
	\newtheorem{theorem}{Theorem}[section]
	\newtheorem{prop}[theorem]{Proposition}
	\newtheorem{heur}[theorem]{Heuristic Theorem}

	\newtheorem{cor}[theorem]{Corollary}
     \newtheorem{lemma}[theorem]{Lemma}
	\newtheorem{definition}[theorem]{Definition}
	\newtheorem{hypoth}[theorem]{Hypothesis}
\theoremstyle{remark}		
	
	\newtheorem*{remark}{Remark}

\numberwithin{equation}{section}
\begin{document}

\title[Analytic Continuation of the Asai L-function]{Quadratic Base Change and the Analytic Continuation of the Asai L-function: A new trace formula approach}
\author{P. Edward Herman} 
\email{peherman@math.uchicago.edu}
\address{University of Chicago,Dept. of Mathematics,
5734 S. University Avenue,
Chicago, Illinois 60637}

%\classification{11F72}

 \keywords{trace formula, beyond endoscopy}

\thanks{This research was supported by an NSF Mathematical Sciences Research Institutes Post-Doc and by the American Institute of Mathematics.}
\begin{abstract}
Using Langlands's  beyond endoscopy idea, we study the Asai L-function associated to a real quadratic field $\mathbf{K}/\Q.$  We prove that the Asai L-function associated to a cuspidal automorphic representation over $\mathbf{K}$ has analytic continuation to the complex plane with at most a simple pole at $s=1.$ We then show if the L-function has a pole then the representation is a base change from $\Q$.

While this result is known using integral representations from the work of Asai and Flicker, the approach here uses novel analytic number techniques and gives a deeper understanding of the geometric side of the relative trace formula. Moreover, the approach in this paper will make it easier to grasp more complicated functoriality via beyond endoscopy. 
%Langlands ultimately expects studying functoriality via beyond endoscopy to be the correct one.  

We also describe a connection of Langlands's technique to the distinguishing of representations via periods.

%Rather than dealing with conjugacy classes and their associated orbital integrals over the quadratic field $\mathbf{K}$ and $\Q$ respectively (as in the traditional trace formula approach), we compare sums of exponential sums. We obtain an equality between two geometric sides of the trace formula via an exponential sum identity of Zagier. 

%We actually go further than to prove th
%Along with proving the Asai L-function has only a simple pole, we also prove its analytic continuation to the entire complex plane via an extra averaging. 

%We hope could be extended so as to allow one to apply {\it beyond endoscopy} methods to the study of other L-functions.

% If the Asai L-function associated to an automorphic form over $\mathbf{K}$ has a pole, then the form is a base change from $\Q$. We prove this and further prove the analytic continuation of the L-function. This is one of the first examples of using a trace formula to get such information. A hope of Langlands is that general L-functions can be studied via this method.

\end{abstract}
\maketitle

\section{Introduction}

%Understanding the analytic continuation of L-functions is a central object of interest in number theory. If the L-function is associated to an automorphic form, then the analytic properties of the function are key in understanding the form itself. This is clearly demonstrated with the Asai L-function in this paper. We will now define it. 

Steps have been made towards a better understanding of Langlands's idea of beyond endoscopy.   Frenkel, Langlands, and Ng\^o in the paper \cite{FLN} applied beyond endoscopy to the Arthur-Selberg trace formula and made a first step of removing the trivial representation so as to understand the key poles of the L-functions in question. In a slightly different direction, there has been a study of approaching beyond endoscopy questions via the relative trace formula. Earliest in this approach, are Sarnak \cite{S} and Venkatesh \cite{V} using the Kuznetsov trace formula for the standard L-function and symmetric square L-function respectively. The Rankin-Selberg L-function is studied by the author in \cite{H}. In this paper, we study the Asai L-function or the twisted tensor L-function. 

Following \cite{AR} there are three ways to define the Asai L-function: via the Langlands correspondence and Langlands parameters; via a Rankin-Selberg integral; and via the Langlands-Shahidi method. Using a certain kind of Rankin-Selberg integral classically in \cite{A} and adelically in \cite{F}, a certain Dirichlet series expression can be formulated for the Asai L-function. We now explain this series.

%Understanding the analytic continuation of $L$-functions is a primary goal in number theory. For $L$-functions associated to automorphic forms the analytic properties of the $L$-functions are key in understanding the forms themselves. This is clearly demonstrated by the case of the Asai $L$-function which is the object of interest in this paper.

%Let $\mathbf{K}=\Q(\sqrt{D})$ be a real quadratic extension of $\Q,$ with $D$ a prime. We assume for computational clarity in this paper that $\mathbf{K}$ has class number one. Let $\Pi$ be an automorphic form with level one and trivial nebentypus over the field $\K.$ The details we need of a automorphic form we save to next section, and for more elaborate details we refer to \cite{BMP1}, \cite{BMP2}, and \cite{V}. If it is more comfortable, one can think of a Hilbert modular form instead of an automorphic form over $\K.$ 
%The form has associated to it Fourier coefficients $\{c_{\mu}(\Pi)\}$ parametrized by integral ideals $\mu.$

Let $D$ be positive square-free rational integer and let $\mathbf{K}$ be the real quadratic extension $\Q(\sqrt{D})$. 
%In order to make the computations as clear as possible we assume that $\mathbf{K}$ has class number one, but this restriction is not essential for the results. 
Let $\Pi$ be an automorphic representation of $GL_2$ over $\K.$ (The necessary background information about $\Pi$ and the associated Hecke eigenform $\phi_{\Pi}$ is given in Section \ref{sec:prelim}. For a more extensive discussion see \cite{BMP1}, \cite{BMP2}, and \cite{V}.)
The eigenform $\phi_{\Pi}$ has associated Fourier coefficients $\{c_{\mu}(\Pi)\}$ parametrized by integral ideals $\mu$ of $\mathbf{K}$. The standard $L$-function associated with the form $\phi_{\Pi}$ is given by $$L(s,\Pi)=\sum_{\mu} \frac{c_{\mu}(\Pi)}{\mathbb{N}(\mu)^s}.$$ It is well known $L(s,\Pi)$ has analytic continuation to the complex plane and satisfies a functional equation. 

The Dirichlet series of the Asai $L$-function for $\Pi$ is a sort of sub-series of $L(s,\Pi)$ and is defined by \begin{equation}\label{eq:asailf}L(s,\Pi,\mathrm{Asai})=\zeta(2s)\sum_{n=1}^\infty  \frac{c_{n}(\Pi)}{n^s}.\end{equation} Asai \cite{A} proved that a certain Rankin-Selberg integral integrating a parallel weight Hilbert modular form restricted to $\Q$ against an Eisenstein series (over $\Q$) can be expressed as such a Dirichlet series and can be analytically continued to the entire complex plane except possibly for a simple pole at $s=1.$ He also proved that it has an Euler product and satisfies a functional equation by using the analytic properties of the Eisenstein series intrinsic to the integral representation. Further, if it does have a pole then the Hilbert modular form is a base change or a lift of a classical modular form over $\Q.$ The notation of such a lift--if the Hilbert modular form is associated to an automorphic representation $\Pi$ and the modular form associated to a representation $\pi$--is standardly denoted by $\Pi=BC_{\K/\Q}(\pi).$ In terms of the standard L-function of $\Pi=BC_{\K/\Q}(\pi),$  base change can be expressed as the decomposition $$L(s,\Pi)=L(s,\pi)L(s,\pi \otimes \chi_D),$$ where $\chi_D$ is the character associated to the field $\K=\Q(\sqrt{D}).$

Following Asai, Flicker in \cite{F} took an adelic treatment of Asai's integral representation and proved the same results as Asai but for a general automorphic representation on a real quadratic extension of a number field. Again, the key ingredient in obtaining the analytic continuation of the L-function is the analytic continuation of the Eisenstein series on $GL(n)$ from \cite{JS}. He also proved that an automorphic representation over the quadratic field is distinguished with respect to the ground field precisely when its corresponding Asai L-function has a pole. Recall that a representation $\Pi$ over $\K$ is distinguished if there exists a vector $\phi_{\Pi} \in \Pi$ such that $$\int_{GL_2(\K)Z( \mathbb{A}_{\K}) \backslash GL_2(\mathbb{A}_{\K})} \phi_{\Pi}(g) dg \neq 0.$$ 
%The Asai $L$-function having a pole is equivalent to the standard L-function decomposing as a product of two $L$-functions over $\Q$ of degree 2. 

In this paper, using a beyond endoscopy approach and completely avoiding integral representations we prove the analytic continuation of the Asai L-function for an arbitrary automorphic representation over the field $\mathbf{K}=\Q(\sqrt{D})$ having class number one. The assumption on the class number is for simplicity in a calculation that is already technically difficult. Further we show that if the Asai L-function has a simple pole at $s=1$ then the associated representation is a base change. 

One can consider the method of the integral representation of the Asai L-function in \cite{A}, \cite{F} as an easier way to get the analytic continuation, but in consideration of Langlands's beyond endoscopy idea, a trace formula approach seems the most systematic way to get analytic continuation for all L-functions $L(s,\pi,\rho)$ associated to a dual group representation $\rho$ of an automorphic representation $\pi$ of a group $G.$ For example, currently there is no general procedure of using integral representations to get the analytic continuation for the symmetric power L-functions. This paper goes beyond the goals of beyond endoscopy in showing analytic continuation to the complex plane which is more difficult than whether the L-function has a pole at $s=1$ or not. The question requires a deeper understanding of the geometric side of the trace formula, and this paper is a step in the direction of understanding all the analytic structure of the L-function at the same time. Recently the author made another step in this direction in reclaiming the functional equation of the standard L-function associated to a $GL_2$ automorphic form from just the trace formula \cite{H1}.

\section{Setup of the problem}

%To study these representations and their associated Asai $L$-function we use the Kuznetsov trace formula which is a special case of the relative trace formula. 
We use a technique originally formulated in \cite{L}, and modified in \cite{S}, to average over the spectrum of cuspidal automorphic representations over the quadratic field $\K,$ along with an averaging over the associated Fourier coefficients. In other words, we study the sum (ignoring test functions and convergence issues of the $\Pi$-sum) \begin{equation}\label{eq:BE}\sum_{\Pi}L(s,\Pi,\mathrm{Asai}).\end{equation}  
However, in order to make the analysis easier on the geometric side of the trace formula, we take a smoothed average of Asai L-functions $$\frac{1}{X} \sum_{\Pi} \sum_{n }g(n/X) c_n(\Pi),$$ for $g \in C_0^\infty(\R^{+}).$ Here  $C_0^\infty(\R^{+})$ denotes smooth compactly supported functions on the postive real numbers. By Mellin inversion this equals \begin{equation}\label{eq:mellga}\frac{1}{2\pi i}\int_{\Re(s)=\sigma} G(s)\left[\sum_{\Pi} L(s,\Pi,\mathrm{Asai})\right]\frac{X^{s-1}}{\zeta(2s)} ds,\end{equation} for $\sigma$ sufficiently large and $G(s)=\int_0^\infty g(x)x^{s-1}dx.$ So how far left we can shift the contour of the integral directly depends on the analytic properties of the Asai L-function. Expecting there to be a simple pole at $s=1,$ the main terms after the contour shift past $\Re(s)=1$ should not depend on $X,$ and the remainder term should decay in $X$ (i.e.\ it should be of size $O(X^{-\delta}),$ for some $\delta>0.$).

 More specifically the calculation should yield  \begin{equation}\label{eq:BE1}\frac{1}{X} \sum_{n}g(\frac{n}{X})\sum_{\Pi}c_{n}(\Pi)\quad = \quad \frac{G(1)}{\zeta(2)}\sum_{\Pi} Res_{s=1}L(s,\Pi,\text{Asai}) +O(X^{-\delta}), \quad \delta >0.\end{equation}   As we said, if the Asai $L$-function has a simple pole at $s=1,$ then $\Pi$ is a base change from a form over $\Q,$ specifically a representation $\pi_D$ of level $D$ (see \cite{A} for the case of Hilbert modular forms).    So our heuristic theorem in \eqref{eq:BE1} becomes  \begin{equation}\label{eq:BE2} \sum_{\Pi}  Res_{s=1}L(s,\Pi,\text{Asai})\quad = \sum_{\pi_D} A(\pi_D)L(1,sym^2(\pi_D)) +O(X^{-\delta}),\quad \delta >0,\end{equation} where $A(\pi)$ is a certain constant associated to $\pi_D.$ This heuristic theorem is still not quite accurate as \eqref{eq:BE2} says every $\pi_D$ has a cuspidal base change. This is not true as Maass's cuspidal theta forms which have naturally associated representations $\pi_{D,\mu}$ constructed from Hecke characters $\mu$ over a quadratic field have base changes that are Eisenstein series. 

Taking this phenomenon
into account, our almost complete heuristic theorem for the calculation is  \begin{equation}\label{eq:BE3} \sum_{\Pi}  Res_{s=1}L(s,\Pi,\text{Asai})\quad =\quad  \sum_{\substack{\pi_D\neq \pi_{D,\mu} }} A(\pi_D)L(1,sym^2(\pi_D)) +O(X^{-\delta}), \quad \delta >0.\end{equation}

This asymptotic is still lacking in some respects, as it only tells us information at or near $s=1.$ From the definition of the Asai $L$-function in \eqref{eq:asailf}, we can only shift the contour in \eqref{eq:mellga} (ignoring whether there is a pole or not at $s=1$) to the left at most to the zeroes of $\zeta(2s).$ The point is that the Dirichlet series $\sum_{n=1}^\infty \frac{c_{n}(\Pi)}{n^s}$ does not tell us the whole story of analytic continuation of the Asai L-funciton. This phenomenon can also be seen in the case of the Rankin-Selberg L-function \cite{H}. To remedy this we do an extra averaging in $\mm \in \Z,$ which should, and does, remove the poles created by the zeroes of the $\zeta(2s).$
To be more precise, we have \begin{multline} \frac{1}{X} \sum_{\mm} \sum_{n}g(\frac{\mm^2 n}{X})\sum_{\Pi}c_{n}
(\Pi)=\\ \int_{\Re(s)=\sigma} G(s)\left[\sum_{\Pi} \frac{L(s,\Pi,\mathrm{Asai})}{\zeta(2s)}\left(\sum_
{\mm=1}\frac{1}{\mm^{2s}}\right)\right]X^{s-1} ds=\int_{\Re(s)=\sigma} G(s)\sum_{\Pi} L(s,\Pi,\mathrm
{Asai})X^{s-1} ds.\end{multline}

Thus the heuristic main theorem is,
\begin{heur}\label{hth}
 \begin{equation}\label{eq:BE4} 
 \frac{1}{X} \sum_{\mm} \sum_{n}g(\frac{\mm^2 n}{X})\sum_{\Pi}c_{n}(\Pi)\quad= \sum_{\substack{\pi_D\neq \pi_{D,\mu} }} A(\pi_D)L(1,sym^2(\pi_D)) +O(X^{-M}), \end{equation} for any positive integer $M >0.$  \end{heur} We explain very explicitly the main Theorem \ref{theo} associated to \eqref{eq:BE4} in Section \ref{sec:haq}.

To pass from Theorem \ref{hth} to a statement about individual L-functions, we require a hypothesis which allows us to
match representations.
\begin{hypoth}\label{stu}
Assume for any $ \Pi,$ there exists $M_1,M_2 >0,$ such that \begin{equation}\label{eq:assump}\frac{1}{X}\sum_{n,m}g(\mm^2 n/X) c_n(\Pi) \ll (1+|t_{\Pi_1}|)^{M_1} (1+|t_{\Pi_2}|)^{M_2} ,\end{equation} where $t_{\Pi_1}, i=1,2$ are the archimedean parameters(eigenvalue, weight) associated to $\Pi.$\end{hypoth}

\begin{remark} This hypothesis allows us to reduce the infinite sum over $\Pi$ to a finite one. Also, the hypothesis is equivalent to knowing $L(s,\Pi, \mathrm{Asai})$ has a functional equation and is analytic except for at most poles at $s=0,1.$ In other words, to understand the complete analytic continuation of the Asai L-function at the possible poles, we assume knowledge of the functional equation and the L-function being analytic outside of the critical strip. 

It is important to say that knowing the functional equation and the locations of the poles are two different problems. The functional equation of the Asai L-function associated to any automorphic representation has been known for awhile by using the Langlands-Shahidi method, see \cite{F}. However, this method only gives meromorphic continuation and tells no information inside the ``critical strip."
We see this knowledge of meromorphic continuation of an L-function associated to automorphic representation, but there is lack of knowledge of its analytic continuation in Langlands's paper \cite{L1} for the symmetric cube L-function. It took many years later to realize the analytic continuation of the symmetric cube L-function by the work of Kim and Shahidi \cite{KS}.
\end{remark}

 Ideally, we would like to obtain all information about the Asai $L$-function using only the trace formula, and it is certainly conceivable
that it is possible to derive Hypothesis \ref{stu} (or something similar) directly from the trace formula. %Being able to reduce an equality of two different trace formulas from an infinite dimensional space to a finite dimensional has certainly been done, but requires an adelic setting. For example, the ``matching" for the Arthur-Selberg trace formula (and its twisted variant) are done for cyclic base change in \cite{La}. 
In fact a similar question of the trace formula implying the functional equation for the standard L-function is answered in \cite{H1}. However, we do not address this issue for the Asai L-function in the present paper, preferring simply to take the uniformity
assumption of Hypothesis \ref{stu} as given.
 
 Reduced to a finite sum of $\Pi$, we can use Hecke operators to isolate a single representation $\Pi$ over $\K$ and match it to a single representation $\pi_D$ over $\Q,$ getting the following corollary:
\begin{cor}\label{bc}
Assuming Hypothesis \ref{stu}, the Asai $L$-function associated to a representation $\Pi$ has analytic continuation to the entire complex plane with at most a simple pole at $s=1.$ If the Asai L-function has a  simple pole, then there exists a representation $\pi_D$ over $\Q$ such that $\Pi=BC_{\K/\Q}(\pi_D).$

\end{cor}

\section{Relation to other papers using the trace formula for base change}

There have been several papers using the trace formula to prove quadratic base change. These include \cite{La}, \cite{Sa}, and \cite{Y}. The first two references \cite{La}, \cite{Sa} proved base change by comparing a trace formula over the ground field (in our case \Q) with a certain ``twisted" trace formula over the quadratic field. The comparison is made through a matching of certain test functions used for each trace formula. The last reference \cite{Y} is the most similar to our approach as it uses the relative trace formula. The Kuznetsov trace formula, which we use, is a special case of the relative trace formula (see \cite{KL} for a derivation). Here also there is still a comparison of trace formulas involved with Ye's thesis \cite{Y}. Our work sheds some light on this paper. Naively, the comparison and main theorem made in \cite{Y} (up to some normalizations and written for the quadratic field $\K/\Q$) is for matching $f' $ and $f$ in the Hecke algebras of $\K$ and $\Q$ respectively,  \begin{multline}\label{eq:yeye}
\sum_{\Pi} \sum_{\phi_{\Pi}} \int_{\K\setminus \mathbb{A}_{\K}} R(f')\phi_{\Pi}\left(  \begin{array}{cc}
1& x  \\
0 & 1 \\ \end{array} \right)\psi'(-x)dx \int_{Z(\mathbb{A}_{\Q})GL_2(\Q)\setminus GL_2(\mathbb{A}_{\Q})}\overline{\phi_{\Pi}}(g)dg =\\ \sum_{\pi_D} \sum_{\phi_{\pi_{D}}}  \int_{\Q\setminus \mathbb{A}_{\Q}} R(f)\phi_{\pi_D}\left( \begin{array}{cc}
1& y  \\
0 & 1 \\ \end{array} \right)\psi(y)dy \int_{\Q\setminus \mathbb{A}_{\Q}} \overline{\phi_{\pi_D}} \left(  \begin{array}{cc}
1& z  \\
0 & 1 \\ \end{array} \right)\psi(z)dz,
\end{multline}
where the operator $R(f')$ (respectively $R(f)$) is the standard convolution operator for the trace formula over $\K$ (respectively $\Q$). Also, $\phi_{\Pi}$ is the spherical vector for each $\Pi$ ($\phi_{\pi_D}$ is the spherical vector for $\pi_D$). %(respectively $\phi_{\pi}$ for $\pi$) is an orthogonal basis 
and $\psi$ is an additive character for $\Q \setminus \mathbb{A}$ while $\psi'$ is $\psi$ composed with the trace function associated to $\K.$ Following \cite{F}, $$Res_{s=1} L(s,\Pi, Asai)=\int_{Z(\mathbb{A}_{\Q})GL_2(\Q)\setminus GL_2(\mathbb{A}_{\Q})}\phi_{\Pi}(g)dg$$ using the adelic integral representation of the Asai L-function. Expressing the Whittaker function as a Fourier coefficient (again up to some normalization factors) and choosing $f'$ such that at nonarchimedean places it is the unit of the Hecke algebra and denoting $h_{f'}$ as the Selberg transform of $f'$ we have $$R(f')\phi_{\Pi}=\otimes_{\nu} R(f'_{\nu})\phi_{\Pi}=h_{f'}(\nu_{\Pi}) \phi_{\Pi}.$$ These calculations can be made very explicit by using \cite{KL1}. We can then write the left hand side of \eqref{eq:yeye} as \begin{equation}
\label{eq:yeye1} \sum_{\Pi} h_{f'}(\nu_{\Pi}) \overline{c_l(\Pi)} Res_{s=1} L(s,\Pi, \text{Asai}).\end{equation} This can be considered the left hand side of Theorem \ref{theo}. Therefore, the starting point of Ye's paper is looking exactly at the residual term at $s=1$ of the current paper. The right hand side of \eqref{eq:yeye} is a standard Kuznetsov trace formula and equals via matching of functions $f'$ and $f$ the left hand side of \eqref{eq:yeye}. This is exactly what the main term of the right hand side of Theorem \ref{theo} is. So indirectly we are proving some kind of matching of Hecke algebras in our beyond endoscopy approach as well as the distinguishing of representations which are base change, and the analytic continuation of the Asai L-function. We hope that beyond endoscopy can be used in other cases to get such a comprehensive treatment of L-functions.

One can see from the connection with \cite{Y} an aspect of beyond endoscopy is to start with a trace formula with test function $f'$ and an extra averaging over spectral data (Fourier coefficients, Whittaker functions,..)\ weighted by a complex parameter $s,$ and rearrange the geometric side of the trace formula to clearly see the analytic continuation of the spectral side of the trace formula. Then we can realize what the expected matching of test functions is with the other trace formula and its test function $f.$ 

It is not clear what is the systematic way of rearranging the geometric side of the trace so that we see the analytic continuation, or, for that matter, just the analytic structure even near the pole. The analysis near the pole will give us our main term for \eqref{eq:mellga}. The hardest part of this paper on the Asai L-function is understanding how to get such a main term. 

The beyond endoscopy idea is in its infancy and trying to find a general way of getting a main term for a general L-function $L(s, \pi,\rho)$ is a difficult and important goal; so for now we are happy with understanding explicit examples.

The way we find the main term is also not used in previous trace formula comparison papers. The paper depends heavily on analytic number theory. With these analytic techniques, one literally ``builds" the geometric sum of a trace formula over $\Q,$ and then connects it to its associated spectral sum.  %In the sense of getting analytic information on $L$-functions using a trace formula, the author is reminded of the beautiful paper of Jacquet and Zagier \cite{JZ}, which gets analytic continuation of the symmetric square $L$-function associated to a representation of $GL_2.$

\section{Details of Paper}\label{det}

To the specifics of the paper, in Section \ref{sec:prelim}, we introduce the Kuznetsov trace formula stated in \cite{BMP1}. They state it for a general real number field, we only use it for a quadratic extension of $\Q.$ In Section \ref{sec:haq}, we state very explicitly our main theorem. In Section \ref{sec:exq}, to check that our calculations are correct we describe what is the residue of the Asai $L$-function for a Hilbert modular form.
In Section \ref{sec:ntl}, we prove a crucial bijection between solutions of two different sets of equations. This bijection gets rid of the difficult to handle $e(\frac{\overline{x}}{c})$ when one ``opens" the Kloosterman sum on the geometric side of the trace formula.  In Section \ref{sec:howto} and Section \ref{sec:qann} we implement the bijection into the trace formula and apply analytic number theoretic to get a main term plus a negligible remainder term. Let us explicate these sections more as it is key to the paper. 

Let $V_1, V_2 \in C^{\infty}_0(\R^{+}).$ Using the Kuznetsov formula  on the left hand side of Theorem \ref{hth} (now including the archimedean test functions)  gives
\begin{multline} \frac{1}{X} \sum_{\mm,n\in \Z}
g(\mm^2 n/X) \bigg(\sum_{\Pi \neq \mathbf{1}}
h(V,\nu_{\Pi})c_{n}(\Pi)\overline{c_{l}(\Pi)}+
\{CSC^{\K}_{n,l}\}\bigg) =\\ \frac{1}{X} \sum_{\mm,n\in \Z}
g(\frac{\mm^2 n}{X})  \sum_{c \in \ri, c \neq 0}
\frac{1}{\mathbb{N}(c)} S(n , l,c)V_1(\frac{4\pi
\sqrt{n l}}{c})V_2(\frac{4\pi
\sqrt{n l'}}{c'}).\end{multline}

We now break up the Kloosterman sums and gather all the $n$-terms to get

\begin{equation} \frac{1}{X}  \sum_{\mm}  \sum_{c \in \ri, c \neq 0}
\frac{1}{\mathbb{N}(c)}\sum_{x  (c)^{*}}
  e(\frac{\overline{ x}l}{\delta c}+\frac{\overline{ x'}l'}{\delta' c'})\end{equation} $$\left\{
  \sum_{n\in \Z}
 e(n(\frac{ x}{\delta c}+\frac{ x'}{\delta' c'}))g(\frac{\mm^2 n}{X})
 V_1(\frac{4\pi
\sqrt{nl}}{c})  V_2(\frac{4\pi
\sqrt{nl'}}{c'})\right\},
$$
where $\overline{x}$ is the multiplicative inverse of $x(c).$

Since the term in brackets is smooth, we can and do apply Poisson to the $n$-sum to get 
\begin{equation}\label{eq:LLL00}  \frac{1}{X} \sum_{\mm}  \sum_{c \in \ri, c \neq 0} \frac{1}{\mathbb{N}(c)}\sum_{x  (c)^{*}}
  e(\frac{\overline{ x}l}{\delta c}+\frac{\overline{ x'}l'}{\delta' c'})\end{equation}
  $$\left\{
 \sum_m
 \int_{-\infty}^{\infty} e(t(\frac{x \delta' c' + x'  c \delta-\mathbf{N}(\delta c)m}{\mathbf{N}(\delta c )}) )g(\frac{\mm^2t}{X})
 V_1(\frac{4\pi
\sqrt{tl}}{c})  V_2(\frac{4\pi
\sqrt{tl'}}{c'})dt\right\}.
$$

We now make a change of variables $t \to Xt$ to get \begin{equation}\label{eq:LLL}  \sum_{\mm}  \sum_{c \in \ri, c \neq 0}
\frac{1}{\mathbb{N}(c)}\sum_{x  (c)^{*}}
  e(\frac{\overline{ x}l}{\delta c}+\frac{\overline{ x'}l'}{\delta' c'})\end{equation}
  $$\left\{
 \sum_m
 \int_{-\infty}^{\infty} e(Xt(\frac{x \delta' c' + x'  c \delta-\mathbf{N}(\delta c)m}{\mathbf{N}(\delta c )}) )g(\mm^2t)
 V_1(\frac{4\pi
\sqrt{Xtl}}{c})  V_2(\frac{4\pi
\sqrt{Xtl'}}{c'})dt\right\}.
$$

As we have fixed $l,$ let 
$$   I_{\mm}(n,c,X): = \int_{-\infty}^{\infty}
e(\frac{Xtn}{\mathbf{N}(\delta c)})g(\mm^2 t)
 V_1(\frac{4\pi
\sqrt{Xtl}}{c})  V_2(\frac{4\pi \sqrt{Xtl'}}{c'})dt.
$$

Then our starting sum is equal to 
\[  \sum_{\mm}
 \sum_{c \in \ri, c \neq 0}
\frac{1}{\mathbb{N}(c)}\sum_{x  (c)^{*}}
e(\frac{\overline{x}l}{\delta c}+\frac{\overline{x'}l'}{\delta' c'}) \sum_{m \in
\mathbf{Z}}
 I_{\mm}(x \delta' c' + x'  c \delta-\mathbf{N}(c \delta)m ,c,X).
\]
 
The goal here is to remove terms $e(\frac{\overline{x}l}{\delta c}+\frac{\overline{x'}l'}{\delta' c'}).$ Normally, in encountering Kloosterman sums the best resort is to use Weil's bound, however that would not suffice here.

 Let $X'(c,n)$ be the set of solutions $(x,m)$ of
the equation
$$
 \delta' c'x +   \delta c x' -\mathbf{N}(\delta c)m = n \in \Z,
$$
where $x$ range over a fixed set of representatives of
$(O_{\mathbf{K}}/cO_{\mathbf{K}})^*$ and $m \in \Z.$ 

Then the starting geometric sum equals \[
  \sum_{\mm} \sum_{n\in\mathbf{Z}}   \sum_{c \in \ri, c \neq 0}
\frac{1}{\mathbb{N}(c)}\sum_{x\in X'(c,n)}
 e(\frac{\overline{x}l}{\delta c}+\frac{\overline{x'}l'}{\delta' c'}) I_{\mm}(n,c,X).
\]
Finally, let
\begin{equation} A_{n,X}:=  \sum_{\mm}  \sum_{c \in \ri, c \neq 0}
\frac{1}{\mathbb{N}(c)}\sum_{x\in X'(c,n)}
 e(\frac{\overline{x}l}{\delta c}+\frac{\overline{x'}l'}{\delta' c'}) I_{\mm}(n,c,X);
\end{equation}
then by an interchange of sums we can write the starting sum as
 \begin{equation} \label{eq:an}  \sum_{n \in \mathbf{Z}} A_{n,X}.\end{equation}

 Now for $n \neq 0,$ we can use a bijection of Section \ref{sec:ntl} to rewrite $\overline{x} \mod(c)$ with $x \in X'(c,n)$ in terms another parameter $r \mod (n), rr' \equiv 1 \mod (n)$ namely $$\overline{x} \equiv \frac{cr+c'}{n} \mod (c).$$  This gives $$e(\frac{\overline{x}l}{\delta c}+\frac{\overline{x'}l'}{\delta' c'})=e(\frac{rl}{n}+\frac{rl'}{n})e(\frac{c'l}{\delta c}+ \frac{cl'}{\delta' c'}).$$  We are avoiding some details in that the bijection puts arithmetic restrictions on the $c$-sum and $r.$ Let us label these arithmetic restrictions on the $c$-sum as $c \in Z(r,n).$ Then it remains to investigate the sum

\[
 \sum_{n \in \mathbf{Z}} A_{n,X}= \sum_{\mm} \sum_{\substack{r\in O_{\mathbf{K}}/(\frac{n}{\delta})\\rr'\equiv1(n)}} e(\frac{ r l+ r' l'}{n})
 \sum_{\substack{ c \in \ri, c \neq 0 \\ r\in
Z(r,n)}} \frac{1}{\mathbb{N}(c)}
 e(\frac{-1}{n}(\frac{l c'}{ c }+\frac{l' c}{ c' })) I_{\mm}(n,c,X)
 \]

So the geometric side of the trace formula started with a hard to analyze ``opened" Kloosterman sum and was traded in for a $c$-sum with certain arithmetic conditions. Using standard analytic number techniques (Poisson summation, Mobius inversion) the $c$-sum can be replaced with a main term of an integral times some arithmetic volume factor plus a remainder term that is negligible. Negligible in the sense of the error terms seen in\eqref{eq:BE3}, but on the geometric side of the trace formula. The $\mm$-sum comes into play on the geometric side of the trace formula in the Poisson summation on the $c$-sum. Geometrically, this allows us to go from the error term $O(X^{-\delta}),$ for some $\delta>0$ in \eqref{eq:BE3} to $O(X^{-M}),$ for any positive integer $M>0$ as in \eqref{eq:BE4}.

 In Section \ref{sec:zag} --assuming the Poisson dual of the above $c$-sum is independent of $r$-- we use a formula of Zagier \cite{Z} for the $r$-sum. This formula, in its simplest form is for $n \equiv 0(D),$
$$ 
 \sum_{\substack{r\in O_{\mathbf{K}}/(\frac{n}{\sqrt{D}})\\rr'\equiv1(n)}} e(\frac{ r + r' }{n})=  S_{D}(1,1,n).$$ Here $S_{D}(1,1,n)=\sum_{a(n)^{*}}\chi_D(a)e(\frac{x+\overline{x}}{n}).$
This is our ``bridge" to the trace formula over $\Q.$ This ``bridge"  is analogous to a local matching of orbital integrals  as in \cite{La} in the case of comparison of a twisted trace formula over $\mathbf{K}$ and standard trace formula $\Q.$ Not surprisingly, this identity also shows up in \cite{Y} as again the calculations in that paper start with the spectral sum $$\sum_{\Pi} h_{f'}(t_{\Pi})Res_{s=1} L(s,\Pi, Asai)$$ following the analysis after the equation \eqref{eq:yeye}. In Section \ref{sec:cssc}, we deal with the continuous spectrum over $\mathbf{K}.$ In Section \ref{sec:putti} we realize that the computation done in Section 9 can be written in terms of the geometric side of the trace formula over $\Q$. This requires an important theorem on the convolution of Bessel transforms in \cite{H}.  In order to get an equality or comparison of just cusp forms from $\mathbf{K}$ to $\Q,$ we compare Fourier coefficients for the continuous spectrum over the 2 different fields. Then the continuous spectrum can be removed from both sides. Lastly, in Section \ref{sec:hal}, we exploit the Hecke algebra associated with the problem, and match associated forms assuming Hypothesis \ref{stu}. As well in this section, we show the analytic continuation of the Asai $L$-function. 

\vspace{6mm}

{\bf Acknowledgements.} I would like to thank my advisor Jonathan Rogawski for proposing the idea of the studying the Asai $L$-function. I would like to also acknowledge the very useful conversations with Brian Conrey and Akshay Venkatesh. I like to also thank Jayce Getz for pointing out an error in an initial draft of this paper.

\section{Preliminaries}\label{sec:prelim}

	Let
$\mathbf{K}=\Q(\sqrt{D}),$ with rational prime discriminant $D.$ We assume $\mathbf{K}$ is of class number one. The ring of integers
will be denoted $\mathcal{O_{\mathbf{K}}}.$ Here the discriminant
$D_{\mathbf{K}}=D,$ and the different is generated by $\delta =
\sqrt{D}.$ Likewise, define the absolute norm of an ideal $c$ as
$\mathbb{N}(c),$ and the norm of an element $z\in
\mathcal{O_{\mathbf{K}}}$ as $\mathbf{N}(z).$ We denote the
non-trivial automorphism in this field by $x \rightarrow x'.$ We
use the standard notation for the exponential $e(\T(x)) := \exp(2\pi
i(\frac{x}{\delta} +\frac{x'}{\delta'})).$

   \vspace{.1in}
    A bit of terminology is needed before we can define the Kuznetsov trace formula. We closely follow \cite{BMP1} . Consider the algebraic group $\GG=R_{\mathbf{K}/\QQ}$(SL$_2)$ over~$\QQ$ obtained by restriction of scalars applied to SL$_2$ over~$\mathbf{K}$. We have
\begin{equation} \label{Gdef}
G:= \GG_\RR \cong SL_2(\RR)^2,  \qquad \GG_\QQ \cong \{(x,x') : x\in SL_2(\mathbf{K}) \},
\end{equation}
$G$ contains $K:=$ SO$_2(\RR)^2$ as a maximal compact
subgroup.

 \vspace{.1in}
The image of $SL_2(\Okr)\subset SL_2(\mathbf{K})$ corresponds to $\GG_\Z$.
This is a discrete subgroup of $\GG_\R$ with finite covolume. It is
called the {\sl Hilbert modular group}. We label it $\G.$

 \vspace{.1in}

 \subsubsection{Functions of product type}

The test functions on $G$ that we use are of product type: $f(g)=f_1(g_1)f(g_2)$ for $g=(g_1,g_2) \in G$ with $f_j$ a complex valued function on $SL_2(\R).$

\subsubsection{Subgroups of $G$} For $y \in \R^{+,2}$ we put $$a[y]:=\left( \begin{pmatrix} \sqrt{y_1} & 0 \\ 0 & \frac{1}{\sqrt{y_1}}\end{pmatrix},\begin{pmatrix} \sqrt{y_2} & 0 \\ 0 & \frac{1}{\sqrt{y_2}}\end{pmatrix}\right).$$ This is the identity component of a maximal $\R$-split torus in $G$ which we label $A.$ We normalize the Haar measure of $A$ by $da=\frac{dy_1}{y_1}\frac{dy_2}{y_2}.$

For $x \in \R^2,$ we let $$n[x]:=\left(\begin{pmatrix}1 & x_1 \\ 0& 1 \end{pmatrix},\begin{pmatrix}1 & x_2 \\ 0& 1 \end{pmatrix}\right) \in G.$$ The normalization of the Haar measure for $N:=\{n[x]:x \in \R^2\}$ is $dx_1,dx_2.$

For $u \in \R^{*,2}$ we define $$b[u]:=\left( \begin{pmatrix}u_1 & 0 \\ 0& \frac{1}{u_1}\end{pmatrix},\begin{pmatrix}u_2 & 0 \\ 0& \frac{1}{u_2}\end{pmatrix}\right).$$ 

For $\theta \in \R^{*,2}$ we define $$k[\theta]:=\left( \begin{pmatrix}\cos \theta_1 & \sin \theta_1 \\ -\sin \theta_1 & \cos \theta_1 \end{pmatrix}, \begin{pmatrix}\cos \theta_2 & \sin \theta_2 \\ -\sin \theta_2 & \cos \theta_2 \end{pmatrix}\right) \in K=SO(2) \times SO(2).$$ We normalize the Haar measure for $K$ by $dk:=\frac{d \theta_1}{2\pi}\frac{d \theta_2}{2\pi}$ for $k=k[\theta] \in K.$

Let $M$ be the subgroup $\{b[u]:|u_j|=1\}$ of $K.$ We then have the Iwasawa decomposition $G=NAK$ with standard Parabolic subgroup $P:=NAM.$

 \subsubsection{Characters} \label{sec:chcc}

We put $t'=\{r \in \F: \T(rx) \in \Z \mbox{ for all } x \in \ri \}.$
 All characters of $N$ are $\chi_r:n[x]^{-1} \to e^{2\pi i \T(rx)}$ for $r \in \R.$ All characters of $N $ are obtained by taking $r \in t'.$  
 
All characters of $K$ are of the form $\delta_q: k[\theta]  \to e(\T(q \theta))$ with $q \in (2\Z)^2.$  A function on $G$ has weight $q$ if it transforms on the right according to this character.  
\subsection{Automorphic forms over \ri}\label{sec:asec}

\begin{definition}
Let $q \in (2\Z)^2$ and $\lambda \in \C^2.$ An automorphic form for $SL_2(\ri)$ is a function on $f \in C^{\infty}(SL_2(\R)^2)$ satisfying

\begin{itemize}
\item $$f(\gamma gk) =  f(g) \delta_q(k) \qquad \ec{for all } \gamma \in SL_2(\ri), k \in K.$$
\item $$C_j f = \lambda_j f \mbox{ for } j=1,2.$$
\end{itemize}

We say $q=(q_1,q_2)$ is the weight of $f$ and $\lambda=(\lambda_1,\lambda_2)$ the eigenvalue of $f.$
\end{definition}

 We write $\lambda_j=\frac{1}{4}-\nu_j^2$ with $\Re \nu_j \geq 0$ and call $\nu=(\nu_1,\nu_2)\in \C^2$ the spectral parameter of $f.$

\subsection{Fourier Expansion of the discrete spectrum}
Let $q \in (2\Z)^2,\nu \in \C^2,$ and $r \in t'.$ Define $$W_q^{r,\nu}(nga[y]k):=\chi_r(n)\delta_q(k) \prod_{j=1}^2 W_{\mbox{sign}(r_j)q_j/2,\nu_j}(4\pi y_j|
r_j|),$$ with $W_{k,m}$ the Whittaker function. For example, the Whittaker function $$W^{r,\nu}_0(n[x]ga[y]k)=e( \T(rx)) \prod_{j=1}^2 W_{0,\nu_j}(4 \pi |y_j|)=e( \T(rx)) \prod_{j=1}^2  \sqrt{y_j} K_{\nu_j-1/2}(4 \pi |y_j|)$$  with $K_{s}(x)$ the standard K-Bessel function.

The Fourier expansion of the 
automorphic form $f$ of weight $q$ and spectral parameter $\nu_f$ is of the form \begin{equation}\label{eq:fouo}f(g)=F_{0,\sigma}f(a[g])\delta_q(k)+\sum_{r \in t'} a_{r}(f)d_{r,\sigma}(\nu_f)W^{r,\nu_f}_q(g).\end{equation}

\subsubsection{Eisenstein Series}
We now describe $L_c^2\left( \Gamma \backslash G\right).$ For each $q \in (2\Z)^2,$ there is an Eisenstein series $$E(P,\nu, i\mu,g):=\sum_{\gamma \in \Gamma_{P}\slash \Gamma} a_[\gamma g]^{\rho+2\nu\rho+i\mu}\delta_q(k(\gamma g)).$$ Here $\nu \in \C,$ and $\mu$ is an element of a lattice in the hyperplane $\Re(x)=0, x \in \R^2.$ In particular, $\mu$ is defined by $a[\gamma]^{\mu}=1$ for $\gamma \in \Gamma_{P},$ again using similar notation to \cite{BMP1}. The series converges for $\nu > \frac{1}{2},$ and has meromorphic continuation in $\nu$ with Laplaican eigenvalue $\bigg(\frac{1}{4}- (\nu+\mu_1)^2,\frac{1}{4}- (\nu+\mu_2)^2\bigg).$

 Let $L^2(\G\backslash G,q)$ denote the Hilbert space of (classes of) functions that are left invariant under $\G$ and transform according to $\delta_q:k[\theta]=\begin{pmatrix} \cos \theta & \sin \theta \\ -\sin \theta & \cos \theta \end{pmatrix} \to \exp(i \T(q \theta))$ with $q \in (2\Z)^2$ , and square
integrable on $\G\backslash G$ for the measure induced by the Haar
measure. This Hilbert space is a direct sum of orthogonal subspaces $L^2_c(\G\backslash G,q)$ and $L^2_d(\G\backslash G,q).$  The subspace $L^2_d(\G\backslash G,q)$ has a countable orthonormal basis consisting of square integrable automorphic forms. The orthogonal subspace
$L^2_c(\G\backslash G,q)$ generated by integrals of unitary Eisenstein
series. 

\subsection{Representations associated to Automorphic forms over $\ri$}\label{sec:asecdd}

Let $L^2(\G \backslash G)^+$ denote the closure of $\sum_{q \in 2\Z} L^2(\G\backslash G,q)$ and similarly for $L^2_c(\G\backslash G)^+.$ This latter space is invariant under the action by $G$ by right translation. The orthogonal complement of $L^2_c(\G\backslash G)^+$ in   $L^2(\G\backslash G)^+$ denoted $L^2_d(\G \backslash G)^+$ is the closure of $\sum_{\irr} V_\irr$,
where $\Pi$ runs through an orthogonal family of closed
irreducible subspaces for the $G$-action. Each representation $\irr$
has the form $\irr=  \irr_1 \otimes \irr_2$, with $\irr_j$ an even unitary
irreducible representation of~SL$_2(\R)$. Table 1 of \cite{BMP1} lists
the possible isomorphism classes for each~$\irr_j$. For each~$\irr$
we define a spectral parameter $\n_\irr =
\left(\n_{\irr,1},\n_{\irr,2}\right)$, with $\n_{\irr,j}$ as
in the last column of the corresponding  table. The Casimir operators $C_j$
act on each coordinate for $1\le j \le 2$. The
eigenvalue~$\ld_\irr\in \R^2$ is given by $\ld_{\irr,j} =
\frac{1}{4}-\nu_{\irr,j}^2$. We note that if $\irr_j$ lies in the
complementary series, $\ld_{\irr,j} \in (0,\frac 14)$ and if $\irr_j$
is isomorphic to a discrete series representation $D^{\pm}_b$, $b \in
2\Z$, $b\ge 2$, then $\ld_{\irr,j}=\frac b2 (1- \frac b2)\in \Z_{
\le 0}$.

The constant functions give rise to $\irr = \bf{1} \isdef \otimes_j 1$.
It occurs with multiplicity one. If $V_\irr$ does not consist of the
constant functions, then $\irr_j\neq 1$ for all~$j$.

For each $\Pi$ the subspace $V_{\Pi, q}$ of weight $q$ has dimension at most $1.$ We choose bases $\mathcal{M}_q$ by taking an orthonormal system $\{\psi_{\Pi,q}\}$ with  $\psi_{\Pi,q} \in V_{\Pi, q}.$

\subsubsection{Fourier coefficient normalization for the trace formula}\label{sec:normy}
The Fourier coefficients of order $r$ are essentially a property of $\Pi$ and not of an individual automorphic form in $V_{\Pi}.$ To introduce the normalized Fourier coefficients needed in the trace formula let  $$d_{r}(q,\nu):=
\prod_{j=1}^2 d_{r}(q_j,\nu_j)$$ with 
$$d_{r}(q_j,\nu_j)=\frac{(-1)^{q_j/2}(2\pi{|r_j|})^{-1/2}}{\Gamma(\frac{1}{2}+\nu_j+\frac{q_j}{2}\mbox{sign}(r_j))}.$$ Applying the Maass operators discussed in \cite{BMP1} we can write the Fourier coefficients in \eqref{eq:fouo} for an automorphic form $\psi_{\Pi,q}$ as $$ a_r(\psi_{\Pi,q})=c_r(\Pi) d_{r}(q,\nu).$$ This determines $c_r(\Pi)$ independently of $q.$ The same procedure can be applied to the Fourier coefficients of the Eisenstein series getting  $$ a_r(E_q(P,\nu,i\mu))=D_r(\nu,i\mu) d_{r}(q,\nu+i\mu).$$

%Let us choose a complete orthonormal basis $\{f_l\}_{l \geq 0}$ of the space of $K$-finite vectors in  the discrete spectrum, call it $L_d^2\left( \Gamma \backslash G,\theta \kappa\right),$ of $L^2\left( \Gamma \backslash G,\theta \kappa\right)$ such that each $f_l$ generates one of the irreducible $V$ under the action of $G.$ The orthogonal complement of $L_d^2\left( \Gamma \backslash G,\theta \kappa\right)$ in $L^2\left( \Gamma \backslash G,\theta \kappa\right),$ call it $L_c^2\left( \Gamma \backslash G,\theta \kappa\right),$ is described by integrals of Eisenstein series.  

%As $G\slash K=SL_2(\C)\slash SU_2(\C) \times SL_2(\C)\slash SU_2(\C),$ each $f_l \in L_d^2\left( \Gamma \backslash G,\theta \kappa\right)$ is an eigenvector of the Laplace operators $L_1,L_2.$ We standardly normalize the Laplace operator to be $L_jf_l=\lambda_{l,j} f_l,$ with $\lambda_{l,j}=1-\mu_{l,j}^2, \mu_{l,j} \in \left(i[0,\infty) \cup (0,1]\right).$ Forms with $\mu_{l,j} \in (0,1]$  for any $j$ are said to have {\it exceptional spectral parameter.}

\subsection{Kuznetsov trace formula over $\K$}\label{sec:tree}

 Let  $V =V_1 \times V_2$ be a test function  in $C_0^{\infty}(\R^{+})^2.$  The transforms associated to the archimedean parameter $\nu_{\Pi}$ are defined as \begin{equation}
 h(V_j,\nu_{\Pi_j})
 = \left\{ \begin{array}{ll}
         i^{k}\int_0^\infty V_j(x)J_{\nu_{\Pi_j}-1}(x)x^{-1}dx & \text{if } \nu_{\Pi_j} \in  2 \Z;\\
         \int_0^\infty
V_j(x)B_{2{\nu_{\Pi_j}}}(x)x^{-1}dx & \text{if } \nu_{\Pi_j}
\in i\R.\end{array} \right.  \end{equation}  Here, $B_{2it}(x) = (
\text{2 sin}(\pi it))^{-1}(J_{-2it}(x) - J_{2it}(x)),$ where
$J_\mu(x)$ is the standard $J$-Bessel function of index $\mu.$ Let the
Kloosterman sum be defined as \begin{equation} S(r,s,c):=\sum_{x \in
(\mathcal{O_{\mathbf{K}}}/ c \mathcal{O_{\mathbf{K}}})^{*}}
e(\T(\frac{r \overline{x}+s x}{c})),\end{equation} where
$x\overline{x} \equiv 1 (c).$ 
    
    The Kuznetsov trace formula then is \begin{equation}\sum_{\Pi \neq \mathbf{1}}
h(V,\nu_{\Pi})c_{\mu}(\Pi)\overline{c_{\nu}(\Pi)}+
\{CSC^{\K}_{\mu,\nu}\} = \end{equation} $$=
 \sum_{\substack{c \in \ri \\ c\neq 0}}^\infty
\frac{1}{\mathbb{N}(c)} S( \nu,\mu,c) V( \sqrt{\mu
\nu}
 /c),$$ where $\mu, \nu \in \mathcal{O_{\mathbf{K}}}.$  
The term $CSC^{\K}_{\mu,\nu}$ denotes the continuous spectrum contribution which depends on the parameters $\mu, \nu.$ It is quite lengthy to explain in detail, but we elaborate on it greatly in Section \ref{sec:cont}. Other descriptions and applications
of the Kuznetsov trace formula for number fields include
 \cite{BMP1}, \cite{BMP2}, \cite{KL}, and \cite{V}.
 
 \subsubsection{Bessel transform and its convolution}\label{sec:bccv}
 
 We now follow section 9 of \cite{H} closely. Define
\begin{equation} V_1*V_2(z):=  \int_{-\infty}^{\infty}
\int_{-\infty}^{\infty}
  e \left((\frac{z}{4\pi})(
 \frac{x}{y}+\frac{y}{x})\right) e \left( (\frac{1}{4 \pi z})\frac{1}{xy}\right)\times \end{equation} $$V(\frac{4\pi
}{x})  W(\frac{4\pi}{ y})\frac{dx}{x} \frac{dy}{y}.
$$
Theorem $3.1$ of \cite{H} states 
\begin{theorem}  For all $V, W \in C_0^{\infty}(\R^{+})$ , $h(V*W,t)=C_{t} h(V,t)h(W,t),$ where $C_{t}=2\pi$ for $t$ an even integer, and $C_{t}=\pi$ for $t$ purely imaginary.
\end{theorem}

 \subsection{Automorphic forms and associated representations over $\Q$}\label{sec:asecd}
 
 In this section we shall consider the discrete subgroup of $$ \Gamma_0(D)=\{\begin{pmatrix}a&b\\ c&d\end{pmatrix}| a,b,c,d \in \Z, c \equiv 0(D)\}.$$  Let $\theta$ be a Dirichlet character modulo $D.$ We consider $\theta$ also as a character of $\Gamma_0(D)$ by $$\theta(\begin{pmatrix}a&b\\ c&d\end{pmatrix})=\theta(d).$$ Let $\chi_D$ be the Dirichlet character associated to the quadratic field $\K$

Completely analogous to Section \ref{sec:asec} we can define $L^2\left( \Gamma_0(D) \backslash SL_2(\R),\theta \right)^+$ as the closure of $\sum_{q \in 2\Z} L^2_q\left( \Gamma_0(D) \backslash SL_2(\R),\theta \right)$ where $L^2_q\left( \Gamma_0(D) \backslash SL_2(\R),\theta \right)$ is the space of square-integrable functions on $SL_2(\R)$ which satisfy 

\begin{equation*}
f(\gamma gk) = \theta(\gamma) f(g)\delta_q(k) \qquad \ec{for all } \gamma \in \Gamma, k \in K.
\end{equation*} 
There is a similar action of $SL_2(\R)$ on $L^2_d \left( \Gamma_0(D) \backslash SL_2(\R) ,\theta\right)$  which decomposes this space into irreducible unitary representations with finite multiplicities. Specifically, $L^2_d(\Gamma_0(D) \backslash SL_2(\R), \theta)^+$ is the closure of $\sum_{\pi} V_{\pi}.$ Here $\pi$ runs through an orthogonal family of closed irreducible subspaces for the $SL_2(\R)$-action. Each representation $\pi.$
is an even unitary irreducible representation of $SL_2(\R).$ Denote the representations that are associated to $L^2_d(\Gamma_0(D) \backslash SL_2(\R), \chi_D)^+$ by $\sum_{\pi_D} V_{\pi_D}.$
 We have obvious normalizations of the Fourier coefficients for automorphic forms in representations $\pi_D.$ There is a completely analogous extension of definitions for the Eisenstein series over $\Q.$

%For irreducible automorphic representations over $\Q$ we have associated Hecke eigenforms defined by $\phi_{t,D}$ to be a Maass form with eigenvalue $\frac{1}{4}+t^2$ of level $D$ with nebentypus the quadratic character $\theta_D$ associated to the field $\mathbf{K}.$ Similarly, $\phi_{k,D}$ denotes a holomorphic form of weight $k$ with similar level and nebentypus. As in \cite{H}, for both forms we denote the $n$-th Fourier coefficient  by $a_{n}(\phi_{*,D}).$

\subsubsection{Maass constructed theta forms}\label{sec:theaa}

There are certain special automorphic forms that come from Hecke characters over a quadratic field. 
Let  $\mu_k=\frac{k\pi}{2\log \epsilon_0},k\neq 0 \in \Z$ where  $\epsilon_0$ is the fundamental unit of our field $\K.$ Define $\omega_{\mu_k}(x):=|\frac{x}{x'}|^{i\mu}$ and for $n \in \N$ let  $\Psi_{\mu_k}(n)=\sqrt{2\pi}n^{1/2}\Gamma(\frac{1}{2}+\mu_k)\sum_{\substack{q \in \ri, \N(q)=n}} \omega_{\mu_k}(q).$

%The $\mu$ sum parametrizes the theta forms. Remember from Section \ref{sec:cont}, $\mu=\frac{k\pi}{2\log \epsilon_0},k\neq 0 \in \Z.$ Each $k \neq 0$ corresponds to a cusp form constructed  from Hecke characters over the quadratic field $\K$. 

Then $$f_{\mu_k,q}(g):=\sum_{r=1}^\infty \frac{\Psi_{\mu_k}(r)}{r^{1/2}} d_{r}(q,\mu_k)W^{r,i\mu_k}(g)$$ is in $L^2_d \left( \Gamma \backslash SL_2(\R) ,\chi_D \right).$  Here $W^{r,\mu}(g)=\chi_r(n)\delta_q(k) W_{q/2,\mu}(4\pi y|r|)$ and  $d_{r}(q,\mu)$ is defined analogously to above. 

They also have associated irreducible unitary representations which we denote by $\pi_{D,\mu_k}$ for each $k \in \Z.$
%\sum_{\mathbb{N}(q)=w}\omega_{\mu}(q),$ with the corresponding form being $$\theta_{\omega_{\mu}}(z)=\sum_{n=1}^\infty \psi_{\mu}(n)\sqrt{y}K_{\frac{i\pi k}{2\log \epsilon_0}}(2\pi n y)e(nx).$$ Further $\theta_{\omega_{\mu}}  \in \Gamma_{0}(D,\chi_d).$
%For a more explicit explanation of these forms see \cite{Bu}. 
These are cuspidal representations that base change to Eisenstein series in $\K.$ So it is to be expected that if the Asai $L$-function is detecting the cuspidal automorphic forms over the quadratic field $\mathbf{K}$ that are base changes, they could not come from theta forms over $\Q.$ These are also the forms which give the poles of the symmetric square $L$-function in \cite{V}.

\subsubsection{Kuznetsov trace formula for representations $\pi_D$ over $\Q$}\label{sec:tree1}

  The Kuznetsov trace formula over $\Q$ for representations $\pi_D$ is for a test function $V \in C^\infty_0(\R^{+}),$ \begin{equation}\sum_{\pi_D \neq \mathbf{1}}
h(V,\nu_{\pi_D})c_{m}(\pi_D)\overline{c_{n}(\pi)}+
\{CSC^{\Q}_{m,n}\} = \end{equation} $$
 \sum_{\substack{c \in \N\\ c \equiv 0(D)}}
\frac{1}{c} S_{D}(m,n,c) V( 4\pi \sqrt{mn}/c),$$ where $m,n \in \N.$   
 The transforms associated to the archimedean parameter $\nu_{\Pi}$ are\begin{equation}
 h(V,\nu_{\pi_D})
 = \left\{ \begin{array}{ll}
         i^{k}\int_0^\infty V(x)J_{\nu_{\pi_D}-1}(x)x^{-1}dx & \text{if } \nu_{\pi_D} \in  2 \Z;\\
         \int_0^\infty
V(x)B_{2{\nu_{\pi_D}}}(x)x^{-1}dx & \text{if } \nu_{\pi_D}
\in i\R.\end{array} \right.  \end{equation}  Here, $B_{2it}(x) = (
\text{2 sin}(\pi it))^{-1}(J_{-2it}(x) - J_{2it}(x)),$ where
$J_\mu(x)$ is the standard $J$-Bessel function of index $\mu.$ The
Kloosterman sum \begin{equation} S_{D}(m,n,c):=\sum_{x (c)^{*}}
\chi_D(x)e(\frac{mx+n\overline{x}}{c}),\end{equation} where
$x\overline{x} \equiv 1 (c).$ The term $CSC^{\Q}_{m,n}$ denotes the continuous spectrum contribution is explained in Section \ref{sec:cont}. 

\section{Explicit main result}\label{sec:haq}

 With all the notation defined, we can finally state the theorem we prove.
 
  Our main result is: 

\begin{theorem}\label{theo} Let $V=V_1 \times V_2 \in C_0^\infty(\R^{+})^2$ and $g \in C_0^\infty(\R^{+})$ such that $\int_0^\infty g(x)dx=1.$ Let $h(V,\nu)$ be the Bessel transform of $V$ defined as in Section \ref{sec:tree}. Let $V_1*V_2$ denote the convolution of the Bessel transform from Section \ref{sec:bccv}.

Let $\Pi$ denote the automorphic representations defined over $\ri$ as in Section \ref{sec:asec} with spectral parameter $\nu_{\Pi},$ and $\pi_D$ the automorphic representations defined over $\Q$ with spectral parameter $\nu_{\pi_D}$ as in Section \ref{sec:asecd}. 

For any positive integer $M \geq 0,$ and  $l \in O_{\mathbf{K}}$ with $(l,D)=1,$

\begin{enumerate}
\item{\{Cuspidal contribution\}}

\begin{multline*}\frac{1}{X} \sum_{\mm} \sum_n g(\mm^2 n/X)\sum_{\Pi \neq \mathbf{1}}
h(V,\nu_{\Pi})c_{n}(\Pi)\overline{c_{l}(\Pi)} =   \\ 2\pi  \bigg(\sum_{\substack{r\in \N \\  r^2 | \N(l)}}   \sum_{\pi_{D} \neq \pi_{D,\mu_k}} h(V_1 *V_2,\nu_{\pi_{D}}){c_{\frac{ll'}{r^2}}(\pi_{D})}{\overline{c_1(\pi_{D})}}\bigg)+O(X^{-M}).\end{multline*}

\item{\{$CSC^{\K}:$=Continuous spectrum contribution for $\K$ \}}

%Let $$\mu_k=\frac{k\pi}{log\epsilon_0},$ where $\epsilon_0$ is the fundamental unit of $\K,$
% and $\tau_{it}(n)=\sum_{ab=n}\theta_D(a)(\frac{a}{b})^{it}, \psi_{\mu_k}(y):=\sum_{\mathbb{N}(q)=y}\omega_{\mu_k}(q),$ where $\omega_{\mu_k}(x)=|\frac{x}{x'}|^{i\mu_k},$ then 
\begin{multline*} \frac{1}{X} \sum_{\mm} \sum_{n} g(\mm^2 n/X) CSC^{\K}_{n,l}=\\ 2\pi  \bigg(\sum_{\substack{r\in \N \\  r^2 | \N(l)}}   \sum_{\pi_{D,\mu_k}} h(V_1*V_2,\nu_{\pi_{D,\mu}}){c_{\frac{ll'}{r^2}}(\pi_{D,\mu})}{\overline{c_1(\pi_{D,\mu})}}+\\  \sum_{\substack{r\in \N \\  r^2 | \N(l)}}\frac{1}{4\pi} \int_{-\infty}^{\infty} h(V_1*V_2,t)D_{\frac{ll'}{r^2}} (it,0) \overline{D_{1} (it,0)}dt\bigg)+ O(X^{-M}). \end{multline*}

\end{enumerate}
\end{theorem}

\section{Residue of the Asai $L$-function}\label{sec:exq}

%We answer the question of what is the residue of the Asai $L$-function for a lifted holomorphic Hilbert modular form $F$ of parallel weight $k$. The other spectral cases are similar. As stated earlier, the Asai $L$-function of $F$ has a pole at $s=1$ if the form is a base change $F=BC_{\Q(\sqrt{D})/\Q}(f)$ for a form $f$ of level $D$ and central character $\theta_D$. Following \cite{A},  

In this section we assume the theory of quadratic base change to show the residue of Theorem \ref{theo} is correct.

%Let $f(g)\sim F(w_1,w_2),$ $$\langle F, F\rangle_{SL_2(\Okr)} = \int_{SL_2(\Okr)\setminus  \mathbb{H}\times \mathbb{H} }|\N(y)|^k |F(w_1,w_2)|^2\prod_{i=1}^2 \frac{dx_idy_i}{y_i^2},$$ where $w_j=x_j+y_j i, j=1,2.$ Similarly, for $h \in S_0(\Gamma_0(D),\chi_d),$ we have the standard inner product $$\langle f,f\rangle_{\Gamma_0(D)}=\int_{\Gamma_0(D)\setminus \mathbb{H}} |f(z)|^2y^{-2}dxdy.$$ Then the residue of the Asai $L$-function is $$Res_{s=1} L(s,f,Asai)=\frac{(4\pi)^k}{\Gamma(k)}\frac{\langle F, F\rangle_{SL_2(\Okr)}}{\zeta(2) \langle f,f\rangle_{\Gamma_0(D)}}.$$

%Including that we take an orthonormal basis of $F$ and the Fourier coefficient normalization of \eqref{eq:dnorm}, we get \begin{equation}\label{eq:wasai}  \frac{1}{\langle F, F\rangle_{SL_2(\Okr)}}\text{Res}_{s=1} \sum_{n=1}^\infty  c_n(F)
%n^{-s}=\frac{12}{\pi}\frac{1}{\langle f,f\rangle_{\Gamma_0(D)}}.
%\end{equation}
 Let $\pi_D$ be an irreducible unitary representation in $L^2_0(\Gamma_0(D)\slash SL_2(\R),\chi_D)$ with $$h(g)=\sum_{r \in \ri} a_r(\pi_D)W_0^{r,\nu}(g) \in \pi_{D}$$ of weight zero. Assume $\pi_D$ is not equal to any representation $\pi_{D,\mu_k}$ from Section \ref{sec:theaa}.
We have $$h_{\rho}(g)=\sum_{r \in \ri} \chi_D(r)a_r(\pi_D)W_0^{r,\nu}(g)\in \pi_D \otimes \chi_D$$ also of weight zero. Assuming the theory of quadratic base change, the automorphic representation $\Pi=BC_{\K/\Q}(\pi_D)$ has a vector $f\in \Pi \subset L^2_0(SL_2(\ri)\slash SL_2(\R)^2)$ of weight zero.  The Asai decomposition for $f$ is 
% \int_{SL_2(\Okr) \slash SL_2(\R)^2} f(g) E(g,s)dg
 
 \begin{multline}\label{eq:wdw}   \int_{SL_2(\Okr) \slash SL_2(\R)^2} f(g) E(g,s)dg=\\ \int_{\Gamma_0(D)\slash SL_2(\R)} h(g) \overline{h_{\rho}(g)} E_{D}(g,s) dg \times \int_{\Gamma_0(D)\slash SL_2(\R)} |h(g)|^2 E_{D}(g,s) dg\end{multline} where $ E_{D}(g,s) \in L^2_c(\Gamma_0(D)\slash SL_2(\R),\chi_D)$ and is of weight zero.
 
 Let $$L(s,\Pi,Asai):=\sum_{n=1}^\infty \frac{c_n(\Pi)}{n^s}$$ where again $c_n(\Pi)$ are the trace formula normalized Fourier coefficients associated to the weight zero automorphic form $f \in \Pi.$ Transferring the classical language of Asai \cite{A}, \eqref{eq:wdw} equals $$\frac{\zeta(2s)G(s)L(s,\Pi,Asai)L(s,\chi_D)}{\zeta(s)} \int_{\Gamma_0(D)\slash SL_2(\R)} |h(g)|^2 E_{D}(g,s) dg$$ where $G(s)$ is a certain product of Gamma factors following a standard Rankin-Selberg unfolding of the integral $$\int_{\Gamma_0(D)\slash SL_2(\R)} h(g) \overline{h_{\rho}(g)} E_{D}(g,s) dg.$$
 
 As the left hand side of \eqref{eq:wdw} has a simple pole at $s=1$ with residue $2\pi L(1,\chi_D)||f||_2,$ the right hand side has a pole there as well with the same residue. The term $$ \frac{L(s,\chi_D) \int_{\Gamma_0(D)\slash SL_2(\R)} |h(g)|^2 E_{D}(g,s) dg }{\zeta(s)}$$ is analytic at $s=1$ and equals $$2\pi L(1,\chi_D)||f||_2\int_{\Gamma_0(D)\slash SL_2(\R)} |h(g)|^2 dg=2\pi L(1,\chi_D)  ||h||_{\Gamma_0(D)}$$ again by a standard Rankin-Selberg unfolding. Therefore, it must be the case that $\zeta(2s)G(s)L(s,\Pi,Asai)L(s,\chi_D)$ has a simple pole at $s=1.$ This implies, as $\zeta(2s), L(s,\chi_D),$ and $G(s)$ are analytic at $s=1,$ that $L(s,\Pi,Asai)$ has a simple pole. In other words, $$\mbox{Res}_{s=1} \zeta(2s)L(s,\Pi,Asai)=\frac{2\pi ||f||_2}{G(1)||h||_{\Gamma_0(D)}}=\frac{2\pi ||f||_2}{||h||_{\Gamma_0(D)}}.$$  Here we used as in standard Rankin-Selberg unfolding that $G(1)=1.$
% \Lambda(s,f \otimes \overline{f}=\Lambda(s,f,Asai) \times \Lambda(s,f \otimes \chi, Asai)=[\zeta

 %Using standard unfolding the Asai L-function can be written as $ \sum_{n=1}^\infty  a_n(\Pi)n^{-s}.$ However, we study the normalized Asai L-function $\sum_{n=1}^\infty  c_n(\Pi)n^{-s}$  coming from the normalized Fourier coefficients built into the Kuznetsov trace formula over $\K.$ Following Section \ref{sec:normy}, $$ \sum_{n=1}^\infty  c_n(\Pi)n^{-s}=\frac{2\pi }{\Gamma(\frac{1}{2}+\nu_{\Pi})\Gamma(\frac{1}{2}-\nu_{\Pi})}\sum_{n=1}^\infty  a_n(\Pi)n^{-s}.$$ 
 %As $f\in \Pi$ is a base change from $h \in \pi_{D},$ using the Fourier normalizations  built into the trace formula over $\Q$ we have $$2 \cos(\pi \nu)\sum_{n=1}^\infty  a_n(\Pi)n^{-s}=2 \cos(\pi \nu)\zeta(s)\sum_{n=1}^\infty  a_{n^2}(\pi_{D}) n^{-s}= \zeta(s)  \sum_{n=1}^\infty  c_{n^2}(\pi_{D}) n^{-s}.$$
From Section \ref{sec:asec} we assume for $f\in \Pi,$  $||f||_2=1.$  Taking into consideration we have an extra average over $\mm$ in the Theorem removing the $\zeta(2s)$ and the orthonormalization of $f$ we have  \begin{equation}\label{eq:wasai11}  \text{Res}_{s=1} \sum_{n=1}^\infty  c_n(\Pi)n^{-s}=\frac{2\pi ||f||_2 }{||h||_{\Gamma_0(D)}}=\frac{2\pi}{||h||_{\Gamma_0(D)}}.\end{equation} This is exactly what Theorem \ref{theo} is giving.
 %Note that this term $2\pi$ is $A(\pi)$ in \eqref{eq:BE4}. So the dependence on $\pi$ is only through the level $D,$ and one can see it agrees with the residue in the main theorem \ref{theo}.

\section{Number-theoretic lemmas}\label{sec:ntl}
In order to make some manipulations on the geometric side of the trace formula, we prove some number-theoretic lemmas. Specifically when we open the Kloosterman sum on the geometric side, we need to get rid of the term $e(\T(\frac{\overline{ x}l}{\delta c}))$ to exploit cancellation. We use these lemmas to facilitate doing that in the Section \ref{sec:howto}.  

\begin{definition}\label{xdef}
Let $X(c,n)$ denote the set $$\{x \in (O_{\mathbf{K}}\slash (c))^{*}: \frac{\T(\delta'cx)-n}{\N(\delta c)} \equiv 0 \pmod \Z\}.$$

%$$
%\delta'  c' x+\delta  c x' = n-m\mathbf{N}(\delta  c).
%$$
%Here we say that $x$ is equivalent to $y$ if $x\equiv
%y \pmod{ c}.$ Let $X(c,n)$ be a
%set of representatives for the classes in
%$\overline{X}(c, n)$.
\end{definition}

We now want to look at the set of solutions of the above definition for $n=0.$ 
%Let us first state a theorem from \cite{FT}.

%\begin{theorem}\label{ftt} [[FT], V.1.16]
%Let $\tau$ denote the non-trivial automorphism of $\K \backslash \Q,$ and let $\mathtt{a}$ denote an $\ri$-ideal with the property $\mathtt{a}^{\tau}=\mathtt{a},$ then $\mathtt{a}=r \mathtt{q},$ where $r$ is a positive rational and $\mathtt{q}$ is a square-free $\ri$-ideal which is divisible only by ramified primes of $\K \backslash \Q.$

%\end{theorem} 

\begin{prop}\label{n00}
Let $c \in \ri.$ Then $x \in  (O_{\mathbf{K}}\slash (c))^{*}$ is in $X(c,0)$ if and only if there exists a unit $\eta \in \ri$ such that $c'=\eta c$ and $x' \equiv \eta x (\delta' c').$

\end{prop}

\begin{proof}
Note in the case $n=0,$ Definition \ref{xdef} implies for $x \in X(c,0),$ $\delta'c' x+ \delta c x' = \N(\delta c)m$ for some $m \in \Z.$ Taking this equation modulo $c$ implies \begin{equation}\label{eq:dq} \delta' c' x \equiv 0(\delta c),\end{equation} which also implies $c' \equiv 0(c)$ as $(x,c)=1.$ Without loss of generality, we also get $c \equiv 0(c').$ Therefore, $c'=\eta c$ with $\eta$ a unit in $\ri.$ Plugging $c'=\eta c$ into \eqref{eq:dq} immediately implies $x' \equiv \eta x (\delta' c').$

Assume there exists a unit $\eta \in \ri$ such that $c'=\eta c$ and $x' \equiv \eta x (\delta' c').$ The latter congruence condition implies $x' - \eta x=\delta' c' t,$ for some $ t \in \ri.$  Multiplying by $\delta c$ and using $c'=\eta c$ gives $$\delta c x' +\delta' \eta c x =\delta c x' +\delta' c' x =\N(\delta c) t.$$ As the middle equation is in $\Z$ we have $t\in \Z.$ Therefore, $x \in X(c,0).$

\end{proof}

\begin{lemma}\label{nDa}
 $X(c,n)$ is empty unless $D|n.$
\end{lemma}
\begin{proof}
Let $c:=\alpha+\beta \sqrt{D}, x:=w+z\sqrt{D},$ then
\begin{equation} n=(-\sqrt{D})(\alpha-\beta
\sqrt{D})(w+z\sqrt{D})+(\sqrt{D})(\alpha+\beta
\sqrt{D})(w-z\sqrt{D})+\end{equation} $$Dm(\alpha^2-D
\beta^{2}).$$ This equals \begin{equation*} -2D(\alpha z -\beta w)
+ Dm(\alpha^2-2 \beta ^{2})\equiv 0(D).\end{equation*}

\end{proof}

\begin{prop}\label{dche}
Suppose $c \in \ri, (c,c')=d.$ Let $x \in X(c,n),$ then $d'=d.$

\end{prop}

\begin{proof}
We note Definition \ref{xdef} is equivalent to there existing an $x(c)^{*}$ such that $$\delta'c' x+ \delta c x' =n +m\N(\delta c), m \in \Z.$$ Suppose $d' \neq d,$ or $d=s \sqrt{D}, s \in \Z.$ By Lemma \ref{nDa}  $D|n$ but clearly also $d|n$ as $(c,c')=d.$ So $\sqrt{D}|n$ but as $n \in \Z, D^2|n.$ Writing $c=\sqrt{D}v,v \in \ri$ and dividing the above equality by $D^2$ we have  $$\frac{v' x+ v x'}{D} =\frac{n}{D^2} +m\N(v) \in \Z.$$ This implies $\T(v'x) \equiv 0(D),$ or writing $v=a+b\sqrt{D}, x=r+s\sqrt{D},$ $$2ar \equiv 0(D).$$
Suppose $(a,D)>1$ then $(v,v')>1,$ contradicting $(c,c')=d.$ If $(r,D)>1,$ then $(x,c)>1,$ again a contradiction. Hence, it must be the case that $d'=d.$
\end{proof}

%\begin{prop}\label{lem:n00}
%  $X(c,0)$ is empty unless $c=\gamma a$ or $\gamma b \sqrt{D}$ with $a,b \in \N,$ and $\gamma=(-1)^e \eta^l$ where $\eta$ is a fundamental unit of the field $\mathbf{K}$ with $e \in \{0,1\}, l \in \Z.$ 

%Take a $c$ as above, then for $x \in X(c,0)$ the following congruence conditions hold:
%\begin{enumerate}

%\item[if $c=\gamma a$] 
% we have $$\eta'^l \overline{x'}\equiv \eta^l \overline{x}(\delta a).$$ Also, there exists $k,h \in \Z$ with $2h\equiv 0(a)$ such that $$\eta^l\overline{x}\equiv  k+h\sqrt{D}(\delta a).$$ 

%\item[If $c=\gamma b \sqrt{D}$]

%then $$\eta'^l\overline{x'}\equiv -\eta^{l} \overline{x}(D b).$$  There exists  $k,h \in \Z$ with $2k\equiv 0(Db)$ such that $$\eta^l \overline{x}\equiv k+h\sqrt{D} (Db)$$  \end{enumerate}\end{prop}

%\begin{proof}
%If there exists $x \in X(c,0)$ we have $c\equiv 0(c')$ and $c'\equiv 0(c).$ By Theorem \ref{ftt} we conclude one of the following is true: $$c=(-1)^e \eta^l  \text{ or } c=(-1)^e \eta^lb\sqrt{D}$$  for some $a,b,l \in \Z, e \in \{0,1\}.$ Without loss of generality take $c=\gamma a,$ then we have $$\delta a(\eta^{'l} x-\eta^l x')\equiv 0(D a^2),$$ or $$\eta^{'l} x\equiv \eta^l x'(\delta a).$$ This implies  $\eta^{l}\overline{x}\equiv \eta^{'l} \overline{x'}(\delta a).$ The identical calculation for $c=\gamma b \sqrt{D},$ gives  $\eta^{'l}\overline{x'}=-\eta^l \overline{x}(Db).$ 
%The last statements of each case are clear if we set $\eta^l x=k+h\sqrt{D}$ for $k,h \in \Z.$

%\end{proof}

Now we look at the case $n \neq 0$ from Definition \ref{xdef}.

\begin{prop}\label{inverse} Let $x \in X(c,n).$ Let  $\xbar \in O_{\mathbf{K}} $ be an inverse of $x$ modulo $c.$ Then there exists an $r \in \ri$ such that
$rr'\equiv 1\pmod{n}$ and
\begin{equation}\label{a}
\xbar = \frac{\delta  c r+\delta'  c'}{n}.\end{equation}
The $r$ is uniquely determined modulo $\frac{n}{\delta}$ by the
equivalence class of  $x$, and the map from
$X(c,n)$ to the set  $r$ modulo $\frac{n}{\delta}$ is
injective.
\end{prop}

\begin{proof} Set
$$
r = \frac{n\xbar-\delta'c'}{\delta  c}.
$$

As $x \in X(c,n)$ we have the equality $$\delta'  c' x+\delta  c x' = n-m\mathbf{N}(\delta  c)$$ for some $m \in \Z.$

Note that $r$ is an integer in the field $\mathbf{K}$ because
$$
n\xbar-\delta' c'= (\delta'  c'x+\delta  cx'+m\mathbf{N}(\delta c))\xbar - \delta'  c' = \delta' c'(x\xbar -
1)+\delta  c x'\xbar + m\mathbf{N}(\delta c)\xbar \equiv  0 \pmod{\delta c}
$$

It is clear that the equivalence class of $r$ is determined by the equivalence class of
$\xbar$.  Indeed, if we replace $\xbar$ by $\ybar = \xbar+\mu
c$,  $r$ is replaced by
\begin{equation}\label{eq:sqq}
s =  r+\mu \frac{n}{\sqrt{D}}
\end{equation}

If  $x$ and $y$  in $X(c,n)$ are both
associated to $r$, then $\xbar=\ybar.$ Therefore $x\equiv y\pmod{c}.$ 
 Finally,
\begin{align*}
rr' & =  \left(\frac{n\xbar -
\delta'  c'}{\delta  c}\right)\left(\frac{n\xbar' - \delta  c}{\delta'  c'}\right)  = 1 +
\frac{n^2\xbar\xbar'-n\xbar \delta c-n\ybar \delta'  c'}{\mathbf{N}(\delta  c)} = 1 +
n\frac{n\xbar\xbar'-\xbar \delta  c-\xbar' \delta'  c'}{\mathbf{N}(\delta  c)}. \end{align*}
 But
 $$
 n\xbar\xbar' = (\delta' cx+\delta  cx+m\mathbf{N}(\delta c))\xbar\xbar' = \delta'  c'x\xbar\xbar' + \delta  c\xbar x'\xbar'+m\mathbf{N}(\delta c)\xbar \xbar'
 $$
 so we have
 \begin{align*}
rr' &=  1 + n\frac{\delta'  c' x\xbar\xbar' +  \delta  c\xbar x'\xbar'-\xbar \delta  c-\xbar' \delta'  c'+m\mathbf{N}(\delta c)\xbar \xbar'}{\mathbf{N}(\delta  c)}\\
&=  1 + n\left[\frac{\delta'  c'(x\xbar -1)\xbar' + \delta  c(x'\xbar' - 1)\xbar}{\mathbf{N}(\delta  c)}+m\xbar \xbar'\right]\\
&=1+n\left[\frac{\xbar' q}{\delta}+\frac{\xbar q'}{\delta'}+m\xbar \xbar'\right], q \in \mathbf{O_K}.
\end{align*}
The expression in brackets is an integer of the field, so   $rr'\equiv
1\pmod{n}$. It is also easy to check that from \eqref{eq:sqq}, $$(r+\mu\frac{n}{\sqrt{D}})(r'-\mu'\frac{n}{\sqrt{D}})\equiv rr' \equiv 1 (n).$$
\end{proof}

 \begin{definition} Let $c$ be an integer in $O_{\mathbf{K}}.$ Set  $d = (c,c')$. Assume that $d|n$.
 Let $Y(c,n)$ be the set of  classes $r\in(O_{\mathbf{K}}/(\frac{n}{\delta}))^*$ with $rr'\equiv1(n),$ and such that
\begin{enumerate}
\item[(a)] $(\delta  c/d)r+(\delta'  c'/d) \equiv 0\pmod{\frac{n}{d}}$
\item[(b)] $(\delta  c/d)r+(\delta'  c'/d) \not\equiv 0 \pmod{\frac{n}{k}}$
if  $k|d$ and  $k<d$.
\end{enumerate}
 \end{definition}

It is easy to check this definition is well-defined on classes $r\in(O_{\mathbf{K}}/(\frac{n}{\delta}))^*$.

 \begin{prop}\label{bijection} The map $i:x \to r$ defines a bijection between $X(c,n)$
 and $Y(c,n).$
 \end{prop}
 \begin{proof} Let $x \in X(c,n)$. We show that the associated $r$ belongs to $Y(c,n).$ We have
$$
\xbar = \frac{ \delta  cr+\delta'  c'}{n} = \frac{(\delta  c/d)r+(\delta'  c'/d)}{(n/d)}.
$$
Therefore, $\frac{\delta  cr}{d}+\frac{\delta' \ c'}{d} \equiv 0\pmod{\frac{n}{d}}$ and (a) is
satisfied.  Suppose that $m$ is a proper divisor of  $d$
 and let $k=d/m$. We claim that  $\frac{\delta  c}dr+\frac{\delta'  c'}d\not\equiv 0\pmod{\frac{n}{k}}.$ If this were not the case,
we would have
$$
\xbar = \frac{(\delta'  c'/d)+(\delta  c/d)r}{(n/d)} =
m\frac{(\delta'  c'/d)+(\delta  c/d)r}{(n/k)}.
$$
This would imply that $m$ divides $\xbar$, which contradicts the
fact that $\xbar$ is a unit modulo $c$. Also by the previous proposition, $rr'\equiv1(n).$ Therefore (b) is
satisfied and $r\in Y(c,n)$. Furthermore, the map $i$ is
injective on $X(c,n)$  by Proposition \ref{inverse}. Next,
assume that $Y(c,n)$ is non-empty.
 Let $r\in O_{\mathbf{K}}$ be relatively prime to $n$ and assume that $r\pmod{n}$ belongs to $Y(c,n)$.
Set
\begin{equation}\label{eq:xi}
\xi = \frac{( c/d)r-(\ c'/d)}{n/\delta d} = \frac{ cr -   c'}{n/\delta}.
\end{equation}
Then $\xi$ is relatively prime to $d$ because
$(  c/d)r-(  c' /d)\not\equiv 0\pmod{n/k\delta}$ for all proper
divisors $k$ of $d$. On the other hand, if $q$ is a common factor
of both $\xi$ and  $  c/d$, then  $q| c' /d$. But $( c/d, c'/d)=1$
so $q$ is a unit. This proves that $\xi$ is prime to both $d$ and $ c/d$,
and hence is a unit modulo $c$. Now choose $x\in O_{\mathbf{K}}$ such that
$x\xi\equiv 1\pmod{c}$ and set $\xbar = \xi$. Then
$$
x\xbar = 1 + \mu c
$$
for some $\mu\in O_{\mathbf{K}}$. As well, the property $rr' \equiv 1(n)$ implies, by multiplying \eqref{eq:xi} by $r',$ that $r'  \equiv x\overline{x'}(\delta c).$

We claim that there exists a $m \in \Z $ such that $$
\delta' c' x +  \delta  c x' = n -m\mathbf{N}(\delta  c).$$ This certainly shows $x \in X(c,n).$

We notice first $$\frac{n-\delta' c' x}{\delta  c} \in O_{\mathbf{K}}. $$ Indeed, by \eqref{eq:xi} we have $\delta' c'=\overline{x}n - \delta  cr,$ so $$ \frac{n-\delta' c' x}{\delta  c}=\frac{n-(\overline{x}n - \delta  cr)x}{\delta  c}=\frac{n(1-x\overline{x}) +\delta  c rx}{\delta  c}=rx+\frac{n\mu}{\delta} \in O_{\mathbf{K}}. $$

Now, by the above argument, there exists some $y \in O_{\mathbf{K}}$ and $m \in \Z$ such that 
 $$y +m\delta' c'=\frac{n-\delta' c' x}{\delta  c}.$$ But this implies \begin{equation}\label{eq:gag}\delta  c y  + \delta' c' x=\delta'  c' y'  + \delta c x'=n-m\mathbf{N}(\delta c) \in \Z.\end{equation} We will show $x' \equiv y(\delta' c').$ Using $r'  \equiv x\overline{x'}(\delta c)$  multiply \eqref{eq:gag} by $\overline{x'}$ to get $$\delta  c y\overline{x'}  + \delta' c' r'=\delta'  c' y'\overline{x'}  + \delta c\equiv n\overline{x'}\text{ } (\mathbf{N}(\delta c)).$$ With \eqref{eq:xi}, this equality reduces to $$\delta  c y\overline{x'}  + \delta' c' r'=\delta'  c' y'\overline{x'}  + \delta c \equiv \delta' c' r' + \delta c\text{ } (\mathbf{N}(\delta c)),$$ or $\delta  c y\overline{x'}\equiv  \delta c\text{ }  (\mathbf{N}(\delta c)).$ This implies $x' \equiv y(\delta' c').$

%Take the above equation modulo $\delta' c'$ to get $$ \delta c y \equiv n (\delta' c').$$ as well modulo $\delta c$ to get $$ \delta' c' x \equiv n (\delta c).$$ Taking the Galois conjugate of the first equation and equating the two implies $x \equiv y'(\frac{\delta c}{d}).$

%So $$\delta  c y  + \delta' c' x=\delta'  c' y'  + \delta c x'.$$ It is an easy check then that $y=x'.$

%Now, by the above argument, there exists some $m \in O_{\mathbf{K}}$ such that $$x' +m\delta' c'=\frac{n-\delta' c' x}{\delta  c}.$$  But this implies $$ \delta  c x'  + \delta' c' x=n-m\mathbf{N}(\delta c),$$ which implies $m \in \Z.$ 

Now if we take $$s=r+\gamma\frac{n}{\delta},$$ its clear \eqref{eq:xi} changes $\xi \to \xi +\gamma c,$ and the rest of the argument follows analogously. Therefore equivalence classes map to equivalence classes. 
This proves the surjectivity and hence the bijection.
\end{proof}

\section{Taking Geometric side of Trace formula}\label{sec:howto}
In this section we rewrite the left hand side of Theorem \ref{theo} and discuss what manipulations are taken in the following sections.
Using the Kuznetsov formula the left hand side of Theorem \ref{theo} equals,
\begin{multline} (L):= \frac{1}{X} \sum_{\mm,n\in \Z}
g(\mm^2 n/X) \bigg(\sum_{\Pi \neq \mathbf{1}}
h(V,\nu_{\Pi})c_{n}(\Pi)\overline{c_{l}(\Pi)}+
\{CSC^{\K}_{n,l}\}\bigg) =\\ \frac{1}{X} \sum_{\mm,n\in \Z}
g(\frac{\mm^2 n}{X})  \sum_{c \in \ri, c \neq 0}
\frac{1}{\mathbb{N}(c)} S(n , l,c)V_1(\frac{4\pi
\sqrt{n l}}{c})V_2(\frac{4\pi
\sqrt{n l'}}{c'}).\end{multline}

We now break up the Kloosterman sums and gather all the $n$-terms.
We can do this because the $c$ and $n$ sum are finite due to the
support of $g$ and $V.$ We have

\begin{equation} \frac{1}{X}  \sum_{\mm}  \sum_{c \in \ri, c \neq 0}
\frac{1}{\mathbb{N}(c)}\sum_{x  (c)^{*}}
  e(\frac{\overline{ x}l}{\delta c}+\frac{\overline{ x'}l'}{\delta' c'})\end{equation} $$\left\{
  \sum_{n\in \Z}
 e(n(\frac{ x}{\delta c}+\frac{ x'}{\delta' c'}))g(\frac{\mm^2 n}{X})
 V_1(\frac{4\pi
\sqrt{nl}}{c})  V_2(\frac{4\pi
\sqrt{nl'}}{c'})\right\},
$$
where $\overline{x}$ is the multiplicative inverse of $x(c).$

Since the term in brackets is smooth, we can and do apply Poisson to the $n$-sum to get 
\begin{equation}\label{eq:LLL00}  \frac{1}{X} \sum_{\mm}  \sum_{c \in \ri, c \neq 0} \frac{1}{\mathbb{N}(c)}\sum_{x  (c)^{*}}
  e(\frac{\overline{ x}l}{\delta c}+\frac{\overline{ x'}l'}{\delta' c'})\end{equation}
  $$\left\{
 \sum_m
 \int_{-\infty}^{\infty} e(t(\frac{x \delta' c' + x'  c \delta-\mathbf{N}(\delta c)m}{\mathbf{N}(\delta c )}) )g(\frac{\mm^2t}{X})
 V_1(\frac{4\pi
\sqrt{tl}}{c})  V_2(\frac{4\pi
\sqrt{tl'}}{c'})dt\right\}.
$$

We now make a change of variables $t \to Xt$ to get \begin{equation}\label{eq:LLL}  \sum_{\mm}  \sum_{c \in \ri, c \neq 0}
\frac{1}{\mathbb{N}(c)}\sum_{x  (c)^{*}}
  e(\frac{\overline{ x}l}{\delta c}+\frac{\overline{ x'}l'}{\delta' c'})\end{equation}
  $$\left\{
 \sum_m
 \int_{-\infty}^{\infty} e(Xt(\frac{x \delta' c' + x'  c \delta-\mathbf{N}(\delta c)m}{\mathbf{N}(\delta c )}) )g(\mm^2t)
 V_1(\frac{4\pi
\sqrt{Xtl}}{c})  V_2(\frac{4\pi
\sqrt{Xtl'}}{c'})dt\right\}.
$$

As we have fixed $l,$ let 
$$   I_{\mm}(n,c,X): = \int_{-\infty}^{\infty}
e(\frac{Xtn}{\mathbf{N}(\delta c)})g(\mm^2 t)
 V_1(\frac{4\pi
\sqrt{Xtl}}{c})  V_2(\frac{4\pi \sqrt{Xtl'}}{c'})dt.
$$

Then $(L)$ is equal to 
\[  \sum_{\mm}
 \sum_{c \in \ri, c \neq 0}
\frac{1}{\mathbb{N}(c)}\sum_{x  (c)^{*}}
e(\frac{\overline{x}l}{\delta c}+\frac{\overline{x'}l'}{\delta' c'}) \sum_{m \in
\mathbf{Z}}
 I_{\mm}(x \delta' c' + x'  c \delta-\mathbf{N}(c \delta)m ,c,X).
\]

 Let $X'(c,n)$ be the set of solutions $(x,m)$ of
the equation
$$
 \delta' c'x +   \delta c x' -\mathbf{N}(\delta c)m = n,
$$
where $x$ range over a fixed set of representatives of
$(O_{\mathbf{K}}/cO_{\mathbf{K}})^*$ and $m \in \Z.$ 

We are free to interchange sums here as the compact support of the test functions $g,V$ imply $j^4\N(c) \sim X,$ so we write $(L)$ as $$
 \sum_{\mm} \sum_{n\in\mathbf{Z}}   \sum_{c \in \ri, c \neq 0}
\frac{1}{\mathbb{N}(c)}\sum_{(x,m)\in X'(c,n)}
 e(\frac{\overline{x}l}{\delta c}+\frac{\overline{x'}l'}{\delta' c'})  I_{\mm}(n,c,X).
$$

Note by a standard integration by parts argument in the $t$-integral inside of $I_{\mm}(n,c,X)$ that the only non-negligible $n$ are those satisfying $n\ll X^{\epsilon}$  for all $\epsilon >0.$ This same estimate is seen in \cite{H}.

%\begin{proof}
%Given $x \in X(c,n)$ by definition there exists an $m$ such that  $\delta' c'x +   \delta c x' -\mathbf{N}(\delta c)m = n,$ so define a map by $x \to (x,m).$ If we take another representative $y\equiv x \pmod c,$ or $y=x+\mu c, \mu \in \ri_{\K},$ then

 We claim there is a bijection between $X'(c,n)$ and $X(c,n).$ By Definition \ref{xdef}, for $x \in X(c,n)$ there exists $m \in \Z$ such that $\frac{\T(\delta c x')-n}{\N(\delta c)} =m.$ We define a map $x \to (x,m).$ Recall $x$ is an equivalence class modulo $c.$ If we change the representative of this class to $y= x +\mu c,$ then we have the map $ y \to (y, m-\T(\delta' \mu)$ so our map is well defined on equivalence classes. This map is surjective by construction. If $(x,m)$ and $(x,l)$ are in $X'(c,n),$ then by an easy check $m=l$ and so the map is injective.

%Likewise, assume $(x,m) \in X'(c,n)$ and that $y=x+\mu c.$ If $(y,b) \in X'(c,n)$ for some $b \in \Z,$ then $(x,b+\T(\delta' \mu)) \in X'(c,n),$ which implies by the previous argument $b=m-[\delta' \mu+\delta \mu].$ Therefore, we have a bijection between the set
%$(x,m)\in X'(c,n)$ and the set of equivalence classes  $x$ in $X(c,n)$ from Definition \ref{xdef}. \end{proof}

Thus we may replace the sum
over $X'(c,n)$ with a sum over $X(c,n)$:
\[
 (L) = \sum_{\mm} \sum_{n\in\mathbf{Z}}   \sum_{c \in \ri, c \neq 0}
\frac{1}{\mathbb{N}(c)}\sum_{x\in X(c,n)}
 e(\frac{\overline{x}l}{\delta c}+\frac{\overline{x'}l'}{\delta' c'}) I_{\mm}(n,c,X).
\]
Finally, let
\begin{equation} A_{n,X}:=  \sum_{\mm}  \sum_{c \in \ri, c \neq 0}
\frac{1}{\mathbb{N}(c)}\sum_{x\in X(c,n)}
 e(\frac{\overline{x}l}{\delta c}+\frac{\overline{x'}l'}{\delta' c'}) I_{\mm}(n,c,X);
\end{equation}
then by an interchange of sums we can write
 \begin{equation} \label{eq:an} (L) =  \sum_{n \in \mathbf{Z}} A_{n,X}.\end{equation}
Here again the interchange of sums is legitimate by the support of the functions $g, V.$

We now see that \eqref{eq:an} equals

\begin{equation} \sum_{D|n \in \Z} A_{n,X}\end{equation} by Lemma \ref{nDa}.

Now for $n \neq 0,$ we can use the bijection of Proposition
\ref{bijection} to rewrite $A_{n,X}$ as a sum over  $r\in
Y(c, n)$:

\[
A_{n,X}= \sum_{\mm} \sum_{\substack{r\in O_{\mathbf{K}}/(\frac{n}{\delta})\\rr'\equiv1(n)}} e(\frac{ r l+ r' l'}{n})
 \sum_{\substack{ c \in \ri, c \neq 0 \\ r\in
Y(c,n)}} \frac{1}{\mathbb{N}(c)}
 e(\frac{-1}{n}(\frac{l c'}{ c }+\frac{l' c}{ c' })) I_{\mm}(n,c,X)
 \]

In summary, we have used the number theoretic lemmas of the previous section to remove the term $e(\frac{-1}{n}(\frac{l c'}{ c }+\frac{l' c}{ c' })$ from the Kloosterman sum over the quadratic field $\K$ and replace it with a more manageable sum. We say a more manageable sum in that for a fixed $r \mod (\frac{n}{\delta}),$ the $c$-sum is essentially now a smooth sum with certain congruence conditions. In other words, we swapped out a complicated Kloosterman sum in the $c$-variable with no arithmetic conditions that is very difficult to estimate with a smooth sum with more complicated congruence conditions. We explain these complicated congruence conditions in Section \ref{sec:casen}.

 \begin{definition}
 Let $X_n(r)$ be the set $c$ such that $r\in Y(c,n)$.
 \end{definition}

\begin{definition}\label{hnn}
Let
$$H_{n,\mm}(x,y):=\frac{1}{j^2xy}e(\frac{-1}{n}(\frac{xl'}{y}+\frac{yl}{x}))e(\frac{n}{xy\mathbf{N}(\delta )}) V_1(\frac{4\pi \sqrt{l}}{x})  V_2(\frac{4\pi \sqrt{l'}}{y}).$$

\end{definition}

The main result of the calculations on the geometric side of our trace formula can be broken down into the
cases: $n=0,$ and $n\neq 0.$

We let \[ \delta(l,l') = \left\{ \begin{array}{ll}
         1 & \mbox{if $l=l'$};\\
        0 & \mbox{if $l\neq l'$}.\end{array} \right. \] 

In Section \ref{sec:case0} we show that for any integer $M \geq 0,$ \begin{equation}\label{a0c}
A_{0,X}=  \delta(l,l') \int_0^\infty \int_0^{\infty} V_1(x)V_2(x)\frac{dx}{x} + O(X^{-M}).
\end{equation}

In Section \ref{sec:casen}  for fixed $n\neq 0,$ $r\in(O_{\mathbf{K}}/(\frac{n}{\delta}))^*$ and any integer $M \geq 0,$   we wil show that
\begin{equation}\label{maincalc}
 \sum_{\mm} \sum_{c\in
X_n(r)} \frac{1}{\mathbb{N}(c)}
 e(\frac{-1}{n}(\frac{l c'}{ c }+\frac{l' c}{ c' })) I_{\mm}(n,c,X) =\frac{1}{n}  \int_{0}^{\infty}  \int_{0}^{\infty} H_{n,\mm}(x,y)dxdy +O_((nX)^{-M}) .\end{equation} In particular the main term is independent of $r$ except for in the exponential sum $e(\T (\frac{rl}{n})).$ In Section \ref{sec:zag} for $(l,D)=1,$ we show that the $r$-sum in $A_{n,X}$ can be replaced by a sum of Kloosterman sums twisted by the character $\chi_{D}$ attached to the quadratic field $\mathbf{K}.$
Specifically, using Lemma \ref{nDa} we  can write $n=Da, a \in \Z$ and using the exponential sum identity in \cite{Z}, $$ 
 \sum_{\substack{r\in O_{\mathbf{K}}/(\frac{Da}{\delta})\\rr'\equiv1(Da)}} e(\frac{ r l+ r' l'}{Da})= \sum_{\substack{r\in \N \\ r^2 |\N(l) \\ r|a}} r S_{D}(\frac{ll'}{r^2},1,\frac{Da}{r}).$$ 

Combining Sections \ref{sec:casen} and \ref{sec:zag}  we get $$A_{Da,X}=
\frac{1}{Da} \sum_{\substack{r\in \N \\ r |l \\ r|a}} r S_{D}(\frac{ll'}{r^2},1,\frac{Da}{r})\int_{0}^{\infty}
\int_{0}^{\infty} H_{Da,1}(x,y)dxdy +O((aX)^{-M}).$$

Lastly, in Sections \ref{sec:case0}, \ref{sec:casen}, and \ref{sec:zag} we will show
\begin{multline}\label{eq:big} (L)=  \delta(l,l')\int_0^\infty \int_0^{\infty} V_1(x)V_2(y)\frac{dxdy}{xy} +  \sum_{a=1}^\infty \frac{1}{Da}
\sum_{\substack{r\in \N \\ r^2|\N(l) \\ r|a}} r S_{D}(\frac{ll'}{r^2},1,\frac{Da}{r})\times \\ \int_{0}^{\infty}
\int_{0}^{\infty} H_{Da,1}(x,y)dxdy +O(X^{-M}).\end{multline}

If $l=l',$ then this equation above is almost the Kuznetsov trace formula over $\Q$ for representations $\pi_D$ in Section \ref{sec:tree1}. In Section \ref{sec:putti} we make this connection.

\section{Computing $A_{n,X}.$}\label{sec:qann}

We separate the cases of $A_{0,X},$ and $A_{n,X},n \neq 0.$ We compute $A_{0,X}$ first.
\subsection{Evaluating $A_{0,X}$}
\label{sec:case0}

%Let us first state a theorem.

%\begin{theorem}\label{ftt} [[FT], V.1.16]
%Let $\tau$ denote the non-trivial automorphism of $\K \backslash \Q,$ and let $\mathtt{a}$ denote an $\ri$-ideal with the property $\mathtt{a}^{\tau}=\mathtt{a},$ then $\mathtt{a}=r \mathtt{q},$ where $r$ is a positive rational and $\mathtt{q}$ is a square-free $\ri$-ideal which is divisible only by ramified primes of $\K \backslash \Q.$

%\end{theorem} 

\indent 
For this section only for $a \in \Z$ we use the notation $y\mod (a)_{\Z}$ for rational integral residue classes modulo $a.$

We recall from Proposition \ref{n00} that $X(c, 0), c \neq 0$ is empty unless there exists a unit $\eta \in \ri$ such that $c'=\eta c$ and an $x (c)$ such that $x' \equiv \eta x (\delta'c')$ or $\overline{x'} \equiv \eta' \overline{x} (\delta'c').$ Using a linear independence of characters argument, such a $c$ must equal $\delta^i a, a \in \Z, i \in \{0,1\}.$ This also implies $\eta=(-1)^i, i \in \{0,1\}.$

  % By Theorem \ref{ftt}  the ideal $\mathtt{c}$ generated by $c$ satisfies $\mathtt{c}^{\tau}=\mathtt{c},$ and therefore $c= aD_1$ for $a \in \N,$ and $D_1$ some generator of the ramified ideal $\mathcal{D}_1 |\mathcal{D}.$ Clearly $c'=\eta c$ implies $D_1'=\eta D_1.$ We can choose a $\epsilon \in \ri$ such that $\eta=\frac{\epsilon'}{\epsilon}$ with $(\epsilon, \epsilon')=1$ by Hilbert's Theorem $90.$ 

Using Proposition \ref{n00}, the exponential sum inside $A_{0,X}$ for $c= \delta^i a$ is \begin{equation*}
\sum_{x \in X(\delta^ia,0)} e(\frac{\overline{x}l}{\delta \delta^ia}+\frac{\overline{x'}l'}{\delta'  \delta'^ia})= \sum_{\substack{x( \delta^ia)^{*}\\x' \equiv (-1)^i x (\delta'\delta^ia)}}e(\frac{\overline{x}l}{\delta  \delta^i a}+\frac{\overline{x'}l'}{\delta'  \delta'^i a}).
\end{equation*}
For $i=1$ the set of $x\mod ( \delta a)$ such that $x' \equiv - x (\delta' \delta a)$ can be represented by $x=\delta y, y \in \Z, y \mod  (a)_{\Z}.$ However, these representatives (and any other choice of representatives) are never coprime to $\delta a,$ and therefore the exponential sum is zero. We consider only the case $i=0.$

Now we assume $i=0.$ We can rewrite the above last line as \begin{equation}\label{eq:ble} \sum_{\substack{x( a)^{*}\\ x' \equiv  x (\delta a)}}e(\frac{\overline{x}l}{\delta  a}-\frac{\overline{x}l'}{ \delta a } )=\sum_{\substack{x( 
a)^{*}\\ x' \equiv  x (\delta a)}}e(\frac{\overline{x}}{  a}( \frac{l-l'}{\delta})).
\end{equation}

We know $z_l:=  \frac{l-l'}{\delta} \in \Z$ and we can choose as representatives of the set $x\mod( a)^{*}, x' \equiv  x (\delta a)$ to be $y \in \Z, y\mod (a)^{*}_{\Z}.$ This sum is a standard Ramanujan sum and equals

\begin{equation*}
 \left\{ \begin{array}{ll}
       \mu\left(\frac{a}{(z_l,a)}\right)\frac{\phi(a)}{\phi(\frac{a}{(z_l,a)})} & \text{if } l \neq l',  \medskip \\
       \phi(a) & \text{if }l =l', \end{array} \right. \end{equation*}

Using \eqref{eq:ble} we can rewrite $A_{0,X}$ if $l \neq l'$ as

\begin{multline} \label{eq:c0} A_{0,X}=\sum_{\mm}   \sum_{\eta \in \ri^{*}}  \sum_{\substack{c\neq 0 \in \ri\\ c'=\eta c}} \frac{1}{\N(c)}
\sum_{\substack{x( c)^{*}\\x' \equiv \eta x (\delta'c')}}e(\frac{\overline{x}l}{\delta  c}+\frac{\overline{x'}l'}{\delta'  c'})
\times \\ \int_{-\infty}^{\infty}  g(\mm^2 t)
 V_1(\frac{4\pi
\sqrt{Xtl}}{c})  V_2(\frac{4\pi
\sqrt{Xtl'}}{c'})dt =\\
 \sum_{\mm}   \sum_{a =1}^\infty \frac{1}{a^2}
  \mu\left(\frac{a}{(z_l,a)}\right)\frac{\phi(a)}{\phi(\frac{a}{(z_l,a)})}
 \int_{-\infty}^{\infty}  g(\mm^2 t) V_1(\frac{4\pi\sqrt{Xtl}}{a})  V_2(\frac{4\pi\sqrt{Xtl'}}{a})dt \end{multline}

While if $l'=l,$ $$A_{0,X}=\sum_{\mm}   \sum_{a =1}^\infty \frac{\phi(a)}{a^2}\int_{-\infty}^{\infty}  g(\mm^2 t) V_1(\frac{4\pi\sqrt{Xtl}}{a})  V_2(\frac{4\pi\sqrt{Xtl}}{a})dt. $$

Let \[ \delta(l,l') = \left\{ \begin{array}{ll}
         1 & \mbox{if $l=l'$};\\
        0 & \mbox{if $l\neq l'$}.\end{array} \right. \] 

\begin{prop}\label{m000}
For any integer $M \geq 0,$
\begin{equation}\label{eq:c1}
 A_{0,X} =
 \delta(l,l') \int_0^\infty \int_0^{\infty} V_1(x)V_2(x)\frac{dx}{x} +O(X^{-M}).
\end{equation}

\end{prop}
\begin{proof}
 
Consider the case $l'\neq l.$ Let $$H_{\mm}(x):=\frac{1}{x^2}\int_0^{\infty} g(\mm^2 t)V_1(\frac{4\pi\sqrt{lt}}{x})V_2(\frac{4\pi\sqrt{l't}}{x})dt,$$ then notice $$H_{\mm}(x)=H_{1}(\mm x).$$ For economy, we use $H(x)$ for $H_{1}(x).$
 
We then have \begin{multline}\label{eq:c08}  \sum_{\mm}   \sum_{a =1}^\infty \frac{1}{a^2}
  \mu\left(\frac{a}{(z_l,a)}\right)\frac{\phi(a)}{\phi(\frac{a}{(z_l,a)})}
H(\frac{ja}{\sqrt{X}}). \end{multline}
 
  Let $\hat{H}(s):= \int_0^\infty H(x) x^{s-1}dx.$ Then by the smoothness of $H,$ which is inherited from $V_1,V_2,$ and Mellin inversion, the sum over  $a$ in \eqref{eq:c0} equals \begin{equation}\label{eq:hmt} \sum_{\mm}   \frac{1}{2\pi i} \int_{(\sigma)}\hat{H}(s)  \left(\frac{\sqrt{X}}{\mm }\right) ^{s} L(s) ds,\end{equation}  for $\sigma>0$ sufficiently large.
Here $$L(s):=\sum_{a=1}^\infty \frac{f_{ a} (z_l)}{a^{s}},$$where $f_n(y)=\sum_{x(n)^{*}} e\left(\frac{xy}{ n} \right)$ is the classical Ramanujan sum. 

 We have \begin{equation} \label{eq:lfn} L(s)=\frac{\sigma_{s-1}(z_l)}{z_l^{s-1}\zeta(s)}\end{equation} from \cite{IK}. Here $\sigma_x(y)=\sum_{d|y} d^x, d,y \in \ri.$

 If $l'=l,$ completely analogous to the previous case, the Dirichlet series in question is $$L(s)=\sum_{a=1}^\infty \frac{\phi(a)}{a^{s}}=\frac{\zeta(s-1)}{\zeta(s)}.$$ 

So we have after recollecting terms that \eqref{eq:hmt} equals

\begin{equation}\label{eq:ss}
 \frac{1}{X2\pi i} \int_{(\sigma)}\hat{H}(s) \left(\sqrt{X}\right) ^{s} L(s)\zeta(s) ds.
\end{equation}

In the case of $l'=l,$ the poles of $L(s)=\frac{\zeta(s-1)}{\zeta(s)},$ outside of $s=2,$ are at the zeroes of $\zeta(s).$ However these poles coming from the zeroes are exactly removed by our extra summation over $\mm \in \Z.$ So the only pole from \eqref{eq:ss} is at $s=2.$ Now using contour integration we shift $\Re(\sigma) \to -2M, M>0,$ and get \begin{equation} \label{eq:mt} \hat{H}(2) + 
 \frac{1}{X}\left(\frac{1}{2\pi i}\right) \int_{(-2M)}\hat{H}(s)  \left(\sqrt{X}\right) ^{s} \zeta(s-1) ds.\end{equation}

So using trivial bounds on the integral, $$
\hat{H}(2) +O(X^{-M}).$$

Similar analysis can be done for $l' \neq l,$ here however we get no pole from the $L$-function \eqref{eq:lfn} and the term here is $O (X^{-M}).$  

%For the $b$-sum we only note here $$\sum_{b=1} \frac{\phi(Db)}{(Db)^{2s}}=\frac{D}{D^{2s}}\sum_{b=1}\frac{\phi(b)}{b^{2s}}.$$

%The same argument  in \eqref{eq:c0}  gives $\frac{X}{2D}\hat{H}(1) + O(X^{-M}).$ 

Therefore \eqref{eq:c0} equals $$ \hat{H}(2) +O(X^{-M})$$ if $l'=l,$ and $O(X^{-M})$ else.
We note in the former case
$$\hat{H}(2)=\int_0^{\infty}\int_0^{\infty}\int_0^{\infty} g(t)V_1(\frac{\sqrt{tl}}{x})V_2(\frac{\sqrt{tl}}{x})dt\frac{dx}{x}.$$ 
A change of variables $x \to \frac{\sqrt{tl}}{ x}$ and the use of $\int_0^\infty g(x)dx=1,$ gives Proposition \ref{m000}.
\end{proof}

\subsection{Evaluating $A_{n,X}, n \neq 0$}
\label{sec:casen}

 Fix $r,$ then by Proposition \ref{bijection}, $X_n(r)$ is the
set of  $c$ such that, setting $d=(c,c')$, we have
 \begin{enumerate}
 \item  $\frac{ c}{d}, \frac{c'}{d}$ are both prime to $\frac{n}{d\delta}.$

 \item $\frac{\delta  cr+\delta'c'}{d}\equiv 0\pmod{\frac{n}d}$
 \item
 $\frac{\delta  cr+\delta'c'}{d}\not\equiv 0\pmod{\frac{n}{k}}$ if $k$ is a proper divisor of $d$.
 \end{enumerate}

 Now for each divisor $d$ of $n$,  let $X_n(r,d)$ be the set of pairs $c$ in $X_n(r)$ such that $(c,c')=d$.

 We would like to prove that there is a constant $R(n,d)$ such that
 \begin{equation}\label{eq:maincalc2}
 \sum_{\mm}\sum_{c\in X_n(r,d)} \frac{1}{\mathbb{N}(c)}
 e(\frac{-1}{n}(\frac{l c'}{ c }+\frac{l' c}{ c' })) I_{\mm}(n,c,X) = R(n,d)\frac1n\int_{0}^{\infty}
\int_{0}^{\infty} H_n(x,y)dxdy +O((nX)^{-M}),
\end{equation}
 and $$\sum_{d|n}R(n,d)=1.$$ 
  \begin{remark}We want the main term to be independent of $r$ as in the geometric side of the trace formula where the archimedean part does not depend on the $r$-sum inside of the Kloosterman sum.\end{remark}

 Let us describe $X_n(r,d)$ explicitly. If $c\in X_n(r,d)$, then there exists a $\lambda \in \mathcal{O_{\mathbf{K}}}$ such that
  \begin{equation}\label{eq:congr}
  \frac{\delta'c'}d = -\frac{\delta  c}d\,r + \lambda\frac{n}d.
  \end{equation}
Also recall from Proposition \ref{dche} that we only consider the case $d'=d$ for the bijection in Proposition \ref{bijection}.

\begin{definition}Let $a,b \in \mathcal{O_{\mathbf{K}}}$  with $c=(a,b).$ Define $(a,b)_{\Z}$ to be the largest rational integer dividing $c.$ In other words, we have $(a,b)_{\Z}|c$ with $(a,b)_{\Z}=(a,b)_{\Z}'$ and if $d=d'$ divides $c$ then $d|(a,b)_{\Z}.$
\end{definition}

  \begin{lemma}\label{lans} Fix  $c$ such that $d'=d,$ $d|c,$ and $r$ such that $r r' \equiv 1 (n).$  Let $\lambda \in \mathcal{O_{\mathbf{K}}}$ be defined by (\ref{eq:congr}). Then $c\in X_n(r,d)$ if and only if $(\lambda,c)_{\Z}=1$ and $(\frac{c}{d}, \frac{n}{d})_{\Z}=1.$  
  \end{lemma}

  \begin{proof}

Let us assume $(\lambda,c)_{\Z}=1$ and $(\frac{c}{d},\frac{n}{d})_{\Z}=1.$
If $(\frac{c}{d}, \frac{c'}{d})=p \neq1,$ then $p' = p$ as otherwise $(\frac{c}{d}, \frac{c'}{d})>p.$  But with \eqref{eq:congr},  $(\frac{c}{d}, \frac{n}{d})_{\Z}=1,$ and $(\lambda,\frac{c}{d})_{\Z}=1,$ there is a contradiction.

If $(\frac{c}{d},\frac{n}{d})=p, p' \neq p,$ then $p | c',$ or $p' |c.$ This implies $p'p |(\frac{c}{d},\frac{n}{d})$ but this contradicts $(\frac{c}{d}, \frac{n}{d})_{\Z}=1.$ Clearly if $p'=p$ there is a contradiction. And if $p$ divides $c'/d$ and
$n/d$, then \eqref{eq:congr} gives $p|r(c/d)$. But $(r,n)=1$, so
this implies that $p$ divides $c/d.$ As $p| (\frac{c}{d}, \frac{c'}{d})$ by a previous argument $p'=p$ which contradicts $(\frac{c}{d},\frac{n}{d})_{\Z}=1.$ This fulfills $(1)$ of the definition of $X_n(r,d).$

If $d=1$, the requirements are met. If $d\neq1$, we must also require that
\begin{equation}\label{eq:repeat}
 \frac{\delta  cr+\delta' c'}d\not\equiv 0 \pmod{p\frac{n}d}
 \end{equation}
 for all $p|d$. Suppose we take $p|(d,\lambda) , p' \neq p.$ Then writing $\lambda=pq$ we have $$
  \frac{\delta  dcr+\delta'dc'}{n} = \lambda=pq.$$ Taking the conjugate of this equation and multiplying by 
  $r'$ and using $rr' =1+ln, l \in \Z$ we get the equality $$pq=p'q'r'-\delta'dc'l.$$ As $p|d$ and 
  $(p,rp')=1$ we have $p| \lambda'$ or $p' | \lambda.$ This implies as $d=d'$ that  $pp'|d
  $ and hence $pp'|(d,\lambda).$ This contradicts that $(dc,\lambda)_{\Z}=1.$

% But
% $$
%  \frac{\delta  cr+\delta'c'}d = \lambda\frac{n}{d}.
% $$
%Therefore \eqref{eq:repeat} holds if and only if $\lambda\not\equiv
%0\pmod{p}$ for all $p|d$. In other words, $\lambda$ must be
%relatively prime to $d$ in the ring of integers of $\K.$

%However, suppose $p \in \ri $ such that $p \neq p', $

% But  as $d=(c,c')$ it suffices to check the gcd of $\lambda$ and $d$ in the rational integers as $d'=d.$ 

Now the other direction. If $c\in X_n(r,d)$ then by definition $(\frac{c}{d},\frac{n}{d})=(\frac{c}{d},\frac{c'}{d})=1.$ So certainly $(\frac{c}{d},\frac{n}{d})_{\Z}=1.$
If $(\lambda, \frac{c}{d})_{\Z}=p>1$ then this immediately contradicts $(\frac{c}{d},\frac{c'}{d})=1.$   Thus we must only show $(\lambda,d)_{\Z}=1.$ By definition of $X_n(r,d)$  \begin{equation}\label{eq:repeat}
 \frac{\delta  cr+\delta' c'}d\not\equiv 0 \pmod{p\frac{n}d}
 \end{equation}
 for all $p|d.$ This implies $(\lambda, d)=1.$ Therefore, we have $(\lambda, d)_{\Z}=1.$
\end{proof}

\begin{lemma}\label{splitcd}
 Let $\lambda \in \mathcal{O_{\mathbf{K}}}$ be defined by \eqref{eq:congr} such that $(\lambda,c)_{\Z}=1$ and $(\frac{c}{d}, \frac{n}{d})_{\Z}=1.$ Then $(d,\frac{c}{d})_{\Z}=1.$
\end{lemma}
\begin{proof}
Suppose $(d,\frac{c}{d})_{\Z}=p>1.$ By \eqref{eq:congr} we have $$\frac{\delta c r}{d}+\frac{\delta' c' }{d}=\frac{n\lambda}{d}.$$ By conjugation, we also have $(d,\frac{c'}{d})_{\Z}=p>1$ which via \eqref{eq:congr} is a contradiction with $(\frac{c}{d}, \frac{n}{d})_{\Z}=1$ and $(\lambda,c)_{\Z}=1.$
\end{proof}

For convenience, we replace the  $c$ with $dc.$ Then using Lemma \ref{lans} the left-hand side of
  \eqref{eq:maincalc2}  equal to
 \begin{equation}\label{eq:maincalc3}
 \frac{1}{X} \sum_{\mm} \sum_{\substack{c \in \ri, c \neq 0\\ (c,\frac{n}{d})=1\\    cr-c' = \lambda\frac{n}{d\delta}\\(\lambda,dc)_{\Z}=1}} \frac{e(\frac{-1}{n}(\frac{cl'}{c'}+\frac{c'l}{c}))}{\N(dc)}  I_{\mm}(n, \frac{d c}{\sqrt{X}},1).
\end{equation}

%In order to get an asymptotic for the $c$-sum, we need to perform Poisson summation. The following lemma simplifies the $c$-sum, a sum over integers of the field $\K$ into two sums over the rational integers. We then can perform Poisson summation over $\Z$ twice as we do in \cite{H}.

%\begin{lemma}
%Let $n \in \Z$ and $r r' \equiv 1 (n).$ Suppose for $f|d, f \in \Z ,$ and $d \in \ri, d|n,$ there exists $c \in \ri, (c,c')=1$ such that \begin{equation}\label{eq:zeq} cr-c' \equiv 0  (\frac{nf}{d\delta})\end{equation}. Then $c \equiv a  (\frac{nf}{d\delta}) , a \in \Z$ or $a \in \sqrt{D}\cdot \Z.$

%\end{lemma}

%\begin{proof}
%Suppose that such a $c$ exists satisfying the above equality with $c' \neq \pm c. $ We show there exists $y \in \ri, y' \neq \pm y$ such that $ycr-(yc)' \equiv 0  (\frac{nf}{d\delta}).$ Suppose there does not exist such a $y,$ then there exists a $t \in \ri, t \neq t',$ \begin{equation} \label{eq:zeq1} tcr-(tc)' \equiv \alpha  (\frac{nf}{d\delta})\end{equation} with $\alpha \not \equiv 0 (\frac{nf}{d\delta}).$ Multiplying \eqref{eq:zeq} by $\beta \in \ri ,\beta \not\equiv  0 (\frac{nf}{d\delta})$ and subtracting from \eqref{eq:zeq1} gives  $$ [tc-\beta c]r-[t'c'-\beta c'] \equiv \alpha  (\frac{nf}{d\delta}).$$
%Choose $\beta=t'$ then the previous equivalence implies $$[t-t']cr \equiv \alpha  (\frac{nf}{d\delta}).$$

%\end{proof}

\begin{prop}\label{limitcalc}
 Recalling $H_{n,j}(x,y)$ from Definition \ref{hnn}, for any integer $M \geq 0$
\begin{enumerate}
\item \begin{multline}\label{eq:bes} \frac{1}{X} \sum_{\mm} \sum_{\substack{c \in \ri, c \neq 0\\ (c,\frac{n}{d})=1\\    cr-c' = \lambda\frac{n}{d\delta}\\(\lambda,dc)_{\Z}=1}} \frac{e(\frac{-1}{n}(\frac{cl'}{c'}+\frac{c'l}{c}))}{\N(dc)}  I_{\mm}(n, \frac{d c}{\sqrt{X}},1)=\\ \frac{1}{ n }R(n,d) \int_{0}^{\infty}  \int_{0}^{\infty} H_{n,1}(x, y)dxdy +O((nX)^{-M}).\end{multline}

\item $\sum_{d|n} R(n,d)=1$
\end{enumerate}

\end{prop}

Combining the two parts of the above proposition we get

\begin{cor}\label{cor12}
For $\epsilon >0,$ $$A_{n,X}=\frac{1}{n}\sum_{d|n} R(n,d) \int_{0}^{\infty}  \int_{0}^{\infty} H_{n,1}(x, y)dxdy +O((nX)^{-M})=$$ $$ \frac{1}{n}\int_{0}^{\infty}  \int_{0}^{\infty} H_{n,1}(x, y)dxdy +O((nX)^{-M}).$$ 
\end{cor}

\begin{proof}[Proof of Corollary \ref{cor12}]
Immediate after Proposition \ref{limitcalc}.
\end{proof}

\begin{proof}

  Let us fix an $\mm$ for the moment. As $(c,d)_{\Z}=1$ by Lemma \ref{splitcd} the condition $(\lambda,dc)_{\Z}=1$ implies $(\lambda,c)_{\Z}=1$ and $(\lambda,d)_{\Z}=1.$ We remove the condition $(\lambda,d)_{\Z}=1$ with Mobius inversion. This gives 
$$\frac{1}{X} \sum_{\substack{c \in \ri, c \neq 0\\ (c,\frac{n}{d})=1\\ cr-c' = \lambda\frac{n}{d\delta}\\  (\lambda, c)_{\Z}=1}} \sum_{\substack{w|d\\ w|\lambda \\  w \in \N}} \mu(w)  \frac{e(\frac{-1}{n}(\frac{cl'}{c'}+\frac{c'l}{c}))}{\N(dc)}  I_{\mm}(n, \frac{d c}{\sqrt{X}},1).$$

%Write $r=x+y\sqrt{D}.$ Let us assume $y \not\equiv 0 (\frac{nw}{d}).$ Inverting sums and using Lemma \ref{las}, this is equal to 

%$$\frac{1}{X}\sum_{\substack{w|d, w \in \N}} \mu(w)   \sum_{\substack{c=a+b\sqrt{D} \in \ri-\{0\} \\ (a,\frac{nw}{d})_{\Z}=1\\ (a, b)_{\Z}=1\\ b \equiv -a\frac{(x-1)}{k}\overline{\frac{Dy}{k}} (\frac{nw}{\delta kd})}}  \frac{e(\frac{-1}{n}(\frac{cl'}{c'}+\frac{c'l}{c}))}{\N(dc)}  I_{\mm}(n, \frac{d c}{\sqrt{X}},1).$$

%Let $E':=kcr-kc'.$

Let $$F_{n,\mm}(x):=\frac{e(\frac{-1}{n}(\frac{xl'}{x'}+\frac{x'l}{x})}{xx'}  I_{\mm}(n,  x,1).$$

Using the compact support of $V_1,V_2, g,$ it is straightforward to check that \begin{equation}\label{eq:parr}\frac{\partial F_{n,\mm}(x)}{\partial x}, \frac{\partial F_{n,\mm}(x)}{\partial x'} \ll 1.\end{equation}

Also, similar to [\cite{H}, $(8.14)$] by an integration by parts in the $t$-integral $Q$-times gives \begin{equation}\label{eq:esw} F_{n,\mm}(x) \ll \frac{1}{n^Q}. \end{equation}

 To study the left hand side of \eqref{eq:bes}, it is now equivalent to investigate the sum $$ \frac{1}{X} \sum_{\substack{w|d\\  \\  w \in \N}} \mu(w) \sum_{\substack{c\in \ri \\ (c,\frac{n}{d})_{\Z}=1\\ cr-c' = w\lambda\frac{n}{d\delta}\\  (w\lambda, c)_{\Z}=1}} F_{n,\mm}(\frac{dc}{\sqrt{X}}).$$
 
% Rewritten this is $$ \sum_{\substack{w|d  \\  w \in \N}} \mu(w) \sum_{\substack{c, \lambda \in \ri \\ (c,\frac{n}{d})_{\Z}=1 \\ cr-c' = w\lambda\frac{n}{d\delta} \\ (w\lambda, c)_{\Z}=1}} F_{n,\mm}(\frac{dc}{\sqrt{X}}).$$

We aim to perform Poisson summation in $c.$  First, however, we have to remove the arithmetic condition $(\lambda, c)_{\Z}=1$ via Mobius inversion:
$$\frac{1}{X} \sum_{\substack{w|d  \\  w \in \N}} \mu(w) \sum_{\substack{c \in \ri \\ (c,\frac{nw}{d})_{\Z}=1\\ cr-c' = w\lambda\frac{n}{d\delta}}}  \sum_{\substack{ l |c \\ l |\lambda\\ l \in \N}} \mu(l) F_{n,\mm}(\frac{dc}{\sqrt{X}}).$$ Here we also combined the two conditions $(c,\frac{nw}{d})_{\Z}=1$ and $(c,w)_{\Z}=1$ into $ (c,\frac{nw}{d})_{\Z}=1.$

Inverting sums this equals 
$$ \frac{1}{X}\sum_{\substack{w|d  \\  w \in \N}} \mu(w)  \sum_{\substack{ l=1\\  (l, \frac{nw}{d})=1}}^\infty \mu(l) \sum_{\substack{c \in \ri \\ (c,\frac{nw}{d})_{\Z}=1\\ cr-c' =\frac{nw}{d\delta}\lambda}}   F_{n,\mm}(\frac{dlc}{\sqrt{X}}).$$

%This gives $$\sum_{\substack{w|d, w \in \N}} \mu(w)   \sum_{\substack{(a,b) \in \Z^2\\ b \equiv -a\frac{(x-1)}{k}\overline{\frac{Dy}{k}} (\frac{nw}{\delta kd})}} \sum_{l|(a,b)_{\Z}} \mu(k) \sum_{t|(a,\frac{nw}{d})_{\Z}} \mu(t)  F_{n,\mm}(\frac{d a}{\sqrt{X}},\frac{d b}{\sqrt{X}}). $$

%Inverting sums this equals $$  \sum_{\substack{w|d, w \in \N}} \mu(w) \sum_{\substack{l=1\\ (l,\frac{nw}{d})=1}}^\infty \mu(l) \sum_{t|\frac{nw}{d}} \mu(t) \sum_{\substack{(a,b) \in \Z^2\\ b \equiv -ta\frac{(x-1)}{k}\overline{\frac{Dy}{k}} (\frac{nw}{\delta kd})}} F_{n,\mm}(\frac{d tla}{\sqrt{X}},\frac{d lb}{\sqrt{X}}). $$

From the definition of $ I_{\mm}(n,y,X)$ it is clear by a change of variables in the $t$-integral 
$$F_{n,\mm}(x)=F_{n,1}(\mm x).$$ So including the $\mm$-sum from our starting Proposition \ref{limitcalc} we have \begin{equation}\label{eq:wwi}\frac{1}{X} \sum_{\mm=1}^\infty \sum_{\substack{w|d, w \in 
\N}} \mu(w)   \sum_{\substack{ l=1\\  (l, \frac{nw}{d})=1}}^\infty \mu(l) \sum_{\substack{c \in \ri \\ (c,\frac{nw}{d})_{\Z}=1\\ cr-c'  \equiv 0 (\frac{nw}{d\delta}) }}   F_{n,1}(\frac{\mm dlc}{\sqrt{X}}). \end{equation}

Our aim now is to remove the $\mm, l$-sums. Notice $V_1, V_2,g$ inside of $F_{n,1}$ are of compact support therefore the $\mm,c,l$-sums are of compact support. Let $h(y):=F_{n,1}(\frac{ydc
 }{\sqrt{X}}).$ Interchanging the $\mm, l$ with the $c$ sums we study $$\sum_{\mm}  \sum_{\substack{l=1 \\ (l,\frac{nw}{d})=1}}^\infty \mu(l) h(lj).$$ Via Mellin inversion this equals $$\frac{1}{2\pi i} \int_{\Re(s)=4} \hat{h}(s) \prod_{p|\frac{nw}{d}} \frac{(1-\frac{1}{p^{s}})^{-1} \zeta(s)}{\zeta(s)}ds=\prod_{p|\frac{nw}{d}} \sum_{l=0}^\infty h(p^l).$$

Going to back to $F_{n,1}$ we are left focusing on \begin{equation}\label{eq:wwi2}  \frac{1}{X}\sum_{\substack{w|d, w \in \N}} \mu(w) 
 \sum_{\substack{c \in \ri \\ (c,\frac{nw}{d})_{\Z}=1\\ cr-c'  \equiv 0 (\frac{nw}{d\delta})}} \prod_{p|\frac{nw}{d}}   \sum_{l=0}^\infty   F_{n,1}(\frac{ dp^l c}{\sqrt{X}}). \end{equation}

\begin{lemma}\label{hl}
There exists $s \in \ri$ such that $(s,n)=1$ and $r \equiv s \overline{s'} \mod (n).$

\end{lemma}

\begin{proof}
This is a variant of Hilbert's Theorem $90.$ Let $s=c+rc'$ for some $c \in \ri , c \neq 0.$ Then $r s' \equiv s \mod (n).$ For $d \not\equiv c \mod (n)$ and $m=d+rd',$ by a linear independence of characters argument, $s \not\equiv m \mod (n).$ Therefore, there must exist a $c \mod (n), (c+rc', n)=1,$ else this contradicts the elementary result that $\phi(n) \gg \sqrt{n}.$ 

\end{proof}

By Lemma \ref{hl} we can write the previous sum as \begin{equation}\label{eq:wwi3}  \frac{1}{X}\sum_{\substack{w|d, w \in \N}} \mu(w) 
 \sum_{\substack{c \in \ri \\ (c,\frac{nw}{d})_{\Z}=1\\ c\overline{c'}s\overline{s'}  \equiv 1 (\frac{nw}{d\delta})}} \prod_{p|\frac{nw}{d}}   \sum_{l=0}^\infty   F_{n,1}(\frac{ dp^l c}{\sqrt{X}}). \end{equation}

We use the subgroup of Dirichlet characters $\psi$ modulo $(\frac{nw}{d})$ satisfying $$\psi \cdot \overline{\psi'} \equiv 1 \mod (\frac{nw}{d})$$ to detect the congruence condition $c\overline{c'}s\overline{s'}  \equiv 1 (\frac{nw}{d\delta}).$ These characters are in bijective correspondence with characters $\Psi \mod (n), \Psi(\N(\cdot)) \equiv 1 \mod (\frac{nw}{d}).$ There are $\phi(\frac{nw}{d})_{\Z}:=\#\{x \mod \frac{nw}{d} | x\in \Z, (x,\frac{nw}{d})_{\Z}=1\}$ of these characters as they come from characters of $(\Z \slash \frac{nw}{d} \Z)^{*}$ composed with the norm map. After interchanging the $l$-(associated to the geometric series) and $c$-sums this gives \begin{equation}\label{eq:wwi4} \frac{1}{X} \sum_{\substack{w|d, w \in \N}} \mu(w) 
 \prod_{p|\frac{nw}{d}}   \sum_{l=0}^\infty  \sum_{\substack{c \in \ri \\ (c,\frac{nw}{d})_{\Z}=1}}    \frac{1}{\phi(\frac{nw}{d})_{\Z}} \sum_{\substack{\psi \mod (\frac{nw}{d})\\ \psi \cdot \overline{\psi'} \equiv 1 \mod (\frac{nw}{d})}} \psi(cs \overline{(cs)'}) F_{n,1}(\frac{ dp^l c}{\sqrt{X}}). \end{equation}

The main contribution comes from the principal character $\psi_0 \mod (\frac{nw}{ d}).$ We label the term with the principal character as $T^{M}(n,d)$ and the rest of the characters as  $T^{R}(n,d).$ Therefore, \eqref{eq:wwi4} equals $T^{M}(n,d) +T^{R}(n,d).$

\subsubsection{Analysis of $T^{M}(n,d)$}

This term, where we now write $\psi_0 \equiv 1 \mod (\frac{nw}{d})$ is equal to 
\begin{equation}\label{eq:bbn} \frac{1}{X} \sum_{\substack{w|d, w \in \N}}  \frac{\mu(w)}{\phi(\frac{nw}{d})_{\Z}}
 \prod_{p|\frac{nw}{d}}   \sum_{l=0}^\infty  \sum_{\substack{c \in \ri \\ (c,\frac{nw}{d})_{\Z}=1}}     F_{n,1}(\frac{ dp^l c}{\sqrt{X}}). \end{equation}
 
 Write $c=a+\sqrt{D}b$ with $a,b \in \Z.$ The condition $(c,\frac{nw}{d})_{\Z}=1$ can be written as sum of sums with conditions $(a,\frac{nw}{d})_{\Z}=1,(b,\frac{nw}{d})_{\Z}=1 $ minus $(ab,\frac{nw}{d})_{\Z}=1.$ More specifically,  \begin{equation}\label{eq:sts}  \sum_{\substack{c \in \ri \\ (c,\frac{nw}{d})_{\Z}=1}}     F_{n,1}(\frac{ dp^l c}{\sqrt{X}})= \sum_{\substack{c=a+\sqrt{D} b\in \ri \\ (a,\frac{nw}{d})_{\Z}=1}}     F_{n,1}(\frac{ dp^l c}{\sqrt{X}}) + \sum_{\substack{c=a+\sqrt{D} b\in \ri \\ (b,\frac{nw}{d})_{\Z}=1}}     F_{n,1}(\frac{ dp^l c}{\sqrt{X}}) - \sum_{\substack{c=a+\sqrt{D} b\in \ri \\ (ab,\frac{nw}{d})_{\Z}=1}}     F_{n,1}(\frac{ dp^l c}{\sqrt{X}}).\end{equation}

 From each of the three above sums, we can remove the condition $(\cdot,\frac{nw}{d})_{\Z}=1$ again by Mobius inversion. We consider only the first sum on the RHS of \eqref{eq:sts} with the other two sums being analogous. We have  \begin{equation}\label{eq:azz}  \frac{1}{X} \sum_{\substack{w|d, w \in \N}}  \frac{\mu(w)}{\phi(\frac{nw}{d})_{\Z}}
 \prod_{p|\frac{nw}{d}}   \sum_{l=0}^\infty \sum_{t|\frac{nw}{d}, t \in \N} \mu(t) \sum_{c \in \ri}     F_{n,1}(\frac{ dp^l (ta+\sqrt{D}b)}{\sqrt{X}}). \end{equation}

%Interchanging the $l$ and $c$-sums, we apply Poisson summation to the $c$-sum. We get after a standard change of variables the $\lambda$-sum equal to \begin{multline} 
%\sum_{\lambda \in \ri} F_{n,1}(\frac{ dp^l c}{\sqrt{X}},\frac{ dp^l(cr-w\lambda\frac{n}{d\delta})}{\sqrt{X}})= \\ \frac{\delta \sqrt{X}}{p^l w n} \sum_{m_1 \in \ri} e(\T(\frac{cr}{\frac{nw}{\delta d}})) \int_{-\infty}^\infty F_{n,1}(\frac{d kt
%\mm a}{\sqrt{X}},t_1)e(\frac{-m_1\sqrt{X} t_1}{\frac{p^k n}{d \delta}})dt_1. \end{multline} 

Poisson summation  and a change of variables gives $$\sum_{c=a+\sqrt{D}b \in \ri}     F_{n,1}(\frac{ dp^l (ta+\sqrt{D}b)}{\sqrt{X}})=$$  $$\frac{X}{ d^2p^{2l}t} \sum_{m_1 \in \ri} \int_{-\infty}^{\infty}  F_{n,1}(t_1+\sqrt{D}t_2)e(-\T(\frac{\sqrt{X}m_1(t_1+\sqrt{D}tt_2)}{tp^ld}))dt_1dt_2.$$ We apply integration by parts to the $t_1$-integral. Recall from Section \ref{sec:howto} $n \ll X^{\epsilon}$ for any  $\epsilon >0.$ Using the estimate \eqref{eq:parr} and applying integration by parts $Q$-times the integral for $m_1 \neq 0$ is bounded by $$ \int_{-\infty}^{\infty} \int_{-\infty}^\infty F_{n,1}(t_1+\sqrt{D}t_2)e(-\T(\frac{\sqrt{X}m_1(t_1+\sqrt{D}t_2)}{p^ld}))dt_1dt_2 \ll (\frac{n^{l+1}}{m_1 \sqrt{X}})^{Q} \ll X^{-Q/2+\epsilon}m_1^{-Q}.$$ The power of $X$ in this way can be made smaller than any $-M, M>0.$

So we consider only the $m_1=0$ term for now on. The terms with non-zero $m_1$ using the estimate from integration by parts is $O(X^{-M})$ for any $M>0.$ This $m_1=0$ term is
 \begin{equation}\label{eq:wwi3} \frac{1}{d^2} \sum_{\substack{w|d, w \in \N}}  \frac{\mu(w)}{\phi(\frac{nw}{d})_{\Z}}  
\prod_{p|\frac{nw}{d}} \sum_{l=0}^\infty \frac{1}{p^{2l}}  \sum_{t|\frac{nw}{d}, t \in \N} \frac{\mu(t)}{t} \int_{-\infty}^\infty \int_{-\infty}^\infty F_{n,1}(t_1+\sqrt{D}t_2)dt_1dt_2= \end{equation}

$$ \frac{1}{d^2} \sum_{\substack{w|d, w \in \N}} \frac{\mu(w)}{\phi(\frac{nw}{d})_{\Z}}  
\prod_{p|\frac{nw}{d}} (1-\frac{1}{p^2})^{-1}  \sum_{t|\frac{nw}{d}, t \in \N} \frac{\mu(t)}{t} \int_{-\infty}^\infty \int_{-\infty}^\infty F_{n,1}(t_1+\sqrt{D}t_2)dt_1dt_2.$$

%By a completely analogous argument to the $k_1$-sum, we interchange the $a$- and $k$-sums  and perform Poisson summation on the $a$-sum giving  \eqref{eq:wwi3} equal to \begin{equation*}\label{eq:wwi4} \frac{\delta X}{dn} \sum_{\substack{w|d, w \in \N}} \frac{\mu(w)}{w}  \sum_{t|\frac{nw}{d}} \frac{\mu(t)}{t} 
%\prod_{p|\frac{nw}{d}}  \sum_{k=0}^\infty \frac{1}{p^{2k}} \sum_{m_2 \in \Z} \int_{-\infty}^\infty \int_{-\infty}^\infty F_{n,1}(t_2,t_1)e(\frac{-m_2 t_2}{dtp^k})dt_1. \end{equation*}

%Again $dtp^k \ll n^{k+2} \ll X^{\epsilon}$ and assuming $m_2 \neq 0,$ another integration by parts argument ($Q$-applications) gives the integral is bounded by $$ \int_{-\infty}^\infty F_{n,1}(t_2,t_1)e(\frac{-m_2 \sqrt{X} t_2}{dtp^k})dt_1 \ll (\frac{n^{3}}{m_2\sqrt{X}})^Q\ll X^{-Q/2+\epsilon}.$$

%From Proposition \ref{limitcalc} we are now left with \begin{equation*}\label{eq:wwi5} \frac{\delta }{dn} \sum_{\substack{w|d, w \in \N}} \frac{\mu(w)}{w}  \sum_{t|\frac{nw}{d}} \frac{\mu(t)}{t} 
%\prod_{p|\frac{nw}{d}} \sum_{k=0}^\infty \frac{1}{p^{2k}} \int_{-\infty}^\infty \int_{-\infty}^\infty F_{n,1}(t_2,t_1)dt_2dt_1+O(X^{-M}), \end{equation*} upon taking $M>\frac{Q+1}{2}>0.$

By using the estimate \eqref{eq:esw}, the error term can be improved to $O((nX)^{-M}).$

We now simplify the archimedean integral. By definition, the integral \begin{multline}\int_{-\infty}^\infty \int_{-\infty}^\infty F_{n,1}(t_1+\sqrt{D}t_2)dt_2dt_1= \int_{-\infty}^\infty \int_{-\infty}^\infty\frac{e(\frac{-1}{n}(\frac{(t_1+\sqrt{D}t_2)l'}{(t_1-\sqrt{D}t_2)}+\frac{(t_1-\sqrt{D}t_2)l}{(t_1+\sqrt{D}t_2)})}{\N(t_1+\sqrt{D}t_2)} \times \\  \int_{-\infty}^{\infty}
e(\frac{tn}{\mathbf{N}( \delta(t_1+t_2\sqrt{D}))})g( t)
 V_1(\frac{4\pi \sqrt{tl}}{t_1+t_2\sqrt{D}})  V_2(\frac{4\pi \sqrt{tl'}}{t_1-t_2\sqrt{D}})dt dt_2 dt_1.\end{multline}

With a change of variables $t_1+t_2\sqrt{D} \to x, t_1-t_2\sqrt{D} \to y$ this equals \begin{equation}  \int_{-\infty}^\infty \int_{-\infty}^\infty e(\frac{-1}{n}(\frac{xl'}{y}+\frac{yl}{x})) \int_{-\infty}^{\infty} e(\frac{tn}{-Dxy})g( t)
 V_1(\frac{4\pi \sqrt{tl}}{x})  V_2(\frac{4\pi \sqrt{tl'}}{y})dt \frac{dx dy}{xy}.\end{equation}
 
 We can separate the $t$-integral by the change of variables $x \to \sqrt{t}x, y \to \sqrt{t}y$ to get 
\begin{equation}\label{eq:sqs} \left[\int_{-\infty}^\infty g(t) dt\right] \int_{-\infty}^\infty \int_{-\infty}^\infty e(\frac{-1}{n}(\frac{xl'}{y}+\frac{yl}{x}))  e(\frac{n}{-Dxy})g( t)
 V_1(\frac{4\pi \sqrt{l}}{x})  V_2(\frac{4\pi \sqrt{l'}}{y})\frac{dx dy}{xy}.\end{equation}

Using the fact that we chose $g$ such that $\int_{-\infty}^\infty g(t) dt=1,$ \eqref{eq:sqs} equals $$ \int_{-\infty}^\infty \int_{-\infty}^\infty H_{n,1}(x,y)dx dy.$$

Collecting our above analysis together, \eqref{eq:azz} equals  \begin{equation}\label{eq:wwi6}  \frac{1}{d^2}\sum_{\substack{w|d, w \in \N}}  \frac{\mu(w)}{\phi(\frac{nw}{d})_{\Z}}    \sum_{t|\frac{nw}{d}} \frac{\mu(t)}{t} 
\sum_{p |\frac{nw}{d}} \frac{1}{(1-\frac{1}{p^2})} \int_{-\infty}^\infty \int_{-\infty}^\infty H_{n,1}(x,y)dx dy+O((nX)^{-M}),\end{equation} for $M>0.$

The other two sums by a completely analogous argument give the respective terms: \begin{multline*}  \frac{1}{d^2}\sum_{\substack{w|d, w \in \N}}  \frac{\mu(w)}{\phi(\frac{nw}{d})_{\Z}}    \sum_{s|\frac{nw}{d}} \frac{\mu(s)}{s} 
\sum_{p |\frac{nw}{d}} \frac{1}{(1-\frac{1}{p^2})} \int_{-\infty}^\infty \int_{-\infty}^\infty H_{n,1}(x,y)dx dy-\\
 \frac{1}{d^2} \sum_{\substack{w|d, w \in \N}}  \frac{\mu(w)}{\phi(\frac{nw}{d})_{\Z}}    \sum_{z|\frac{nw}{d}} \frac{\mu(z)}{z^2} 
\sum_{p |\frac{nw}{d}} \frac{1}{(1-\frac{1}{p^2})} \int_{-\infty}^\infty \int_{-\infty}^\infty H_{n,1}(x,y)dx dy
+O((nX)^{-M}).\end{multline*}

In summary, we have shown combining the three sums from \eqref{eq:azz} that \eqref{eq:bbn} equals \begin{equation}\label{eq:wwi7}   \frac{1}{d^2}\sum_{\substack{w|d, w \in \N}}  \frac{\mu(w)}{\phi(\frac{nw}{d})_{\Z}}    \sum_{p |\frac{nw}{d}} \frac{1}{(1-\frac{1}{p^2})} \prod_{p|\frac{nw}{d}} \left[2(1-\frac{1}{p}) -(1-\frac{1}{p^2})\right]
 \int_{-\infty}^\infty \int_{-\infty}^\infty H_{n,1}(x,y)dx dy+O((nX)^{-M}),\end{equation}

We have the simplification: \begin{multline*} R(n,d):= \frac{1}{d^2}\sum_{\substack{w|d, w \in \N}}  \frac{\mu(w)}{\phi(\frac{nw}{d})_{\Z}}    \sum_{p |\frac{nw}{d}} \frac{1}{(1-\frac{1}{p^2})} \prod_{p|\frac{nw}{d}} \left[2(1-\frac{1}{p}) -(1-\frac{1}{p^2})\right]=\\  \frac{1}{d^2} \sum_{\substack{w|d, w \in \N}}  \frac{\mu(w)}{\phi(\frac{nw}{d})_{\Z}}    \sum_{p |\frac{nw}{d}} \frac{1}{(1-\frac{1}{p^2})} \prod_{p|\frac{nw}{d}} (1-\frac{1}{p})^2=\\ \frac{1}{nd} \sum_{\substack{w|d, w \in \N}} \frac{\mu(w)}{w}  \sum_{p |\frac{nw}{d}} \frac{1}{(1+\frac{1}{p})}.\end{multline*}
The last equality using that $\phi(n)=n \prod_{p|n} (1-\frac{1}{p}).$  We have the even further simplification of $R(n,d)$ in the following lemma.
\begin{lemma}
We have $\sum_{d|n} R(n,d)=\frac{1}{n}.$
\end{lemma}
\begin{proof}
$R(n,d)$ is multiplicative and so it remains to evaluate it at a prime power $n=p^l$ where $d=p^i, i=0,...,l.$ It is sufficient to do this computation for $i=0, 1 \leq i \leq l-1,$ and $i=l.$  For $i=0,$ the sum has only one term
 $$R(p^l,1)=\frac{1}{p^l(1+\frac{1}{p})}.$$ For  $1\leq i \leq l-1,$ 
$$R(p^l,p^i)=\frac{1}{p^{l+i}} \left(\frac{1-\frac{1}{p}}{1+\frac{1}{p}}\right) .$$

For $i=l,$  
$$R(p^l,p^l)=\frac{1}{p^{2l}}\left(1-\frac{1}{p(1+\frac{1}{p})}\right)=\frac{1}{p^{2l}(1+\frac{1}{p})}.$$

Now $$\sum_{i=1}^{l-1} \frac{1}{p^i} = \frac{1-\frac{1}{p^l}}{1-\frac{1}{p}} -1 =  \frac{\frac{1}{p}-\frac{1}{p^l}}{1-\frac{1}{p}}.$$ 

So \begin{eqnarray*}\sum_{i=0}^{l} R(p^l,p^i) &=& \left[\frac{1}{p^l(1+\frac{1}{p})} \right]+ \left[\frac{1}{p^l}\left(\frac{1-\frac{1}{p}}{1+\frac{1}{p}}\right) \frac{\frac{1}{p}-\frac{1}{p^l}}{1-\frac{1}{p}}\right] + \left[\frac{1}{p^{2l}(1+\frac{1}{p})}\right] \\
&=& \frac{1}{p^{2l}(1+\frac{1}{p})}\left[p^l+p^{l-1}-1 +1\right] \\
 &=&\frac{p^l+p^{l-1}}{p^{2l}(1+\frac{1}{p})}=\frac{1}{p^l}. \end{eqnarray*}

\end{proof}

Using the above lemma with \eqref{eq:wwi6} we are left with \begin{equation}\label{eq:ffr}\frac{1}{ n } \int_{-\infty}^{\infty}  \int_{-\infty}^{\infty} H_{n,1}(x, y)dxdy+O((nX)^{-M}).\end{equation} But remember $H_{n,1}(x,y)$ is defined in terms of the test functions $V_1,V_2 \in C_0^{\infty}(\R^{+}),$ so in the main term integration is limited to $$\int_{0}^{\infty}  \int_{0}^{\infty} H_{n,1}(x, y)dxdy.$$

%The same computation can be done for $d=-d'$ as is done in getting \eqref{eq:ffr}(case of $d=d'$).
%The difference is that the size of the character group is $\phi(\frac{nf}{s})$ where $d=s\sqrt{D}$ by Lemma \ref{Dcard} and gives a term analogous to \eqref{eq:qber} \begin{equation}\label{eq:qber2}\frac{1}{-\sqrt{D}Ds^2  } \sum_{a|s\sqrt{D}} \frac{\mu(a)}{\phi(\frac{na}{ s})}
%\prod_{\substack{ p |\frac{na}{ s\sqrt{D}}}} (\frac{1-\frac{1}{p}}{1+\frac{1}{p}})\int_{-\infty}^{\infty}  \int_{-\infty}^{\infty} H_n(x, y)dxdy.\end{equation}

% to get $$\frac{1}{ D^{3/2}n } \int_{-\infty}^{\infty}  \int_{-\infty}^{\infty} H_n(x, y)dxdy.$$

\subsubsection{Analysis of $T^{R}(n,d)$}
Here recall the remainder term of \eqref{eq:wwi4} is for a fixed $\psi \not\equiv 1 (\frac{nw}{d})_{\Z}, \psi \cdot \overline{\psi'} \equiv 1 \mod (\frac{nw}{d}),$

\begin{equation}\label{eq:www1} \frac{1}{X} \sum_{\substack{w|d, w \in \N}} \mu(w) 
 \prod_{p|\frac{nw}{d}}  \frac{1}{\phi(\frac{nw}{d})_{\Z}}  \sum_{l=0}^\infty  \sum_{\substack{c \in \ri \\ (c,\frac{nw}{d})_{\Z}=1}}    \psi(cs \overline{(cs)'}) F_{n,1}(\frac{ dp^l c}{\sqrt{X}}). \end{equation}

Break the $c$-sum into arithmetic progressions modulo $\frac{nw}{d}$ to get the inner sum equaling

\begin{equation*}\label{eq:www2} \frac{1}{X} \sum_{\substack{w|d, w \in \N}} \mu(w) 
 \prod_{p|\frac{nw}{d}}  \frac{1}{\phi(\frac{nw}{d})_{\Z}}  \sum_{l=0}^\infty  \sum_{q \mod (\frac{nw}{d})} \psi(qs \overline{(qs)'})\sum_{\substack{c \in \ri \\ c\equiv q(\frac{nw}{d})}}     F_{n,1}(\frac{ dp^l c}{\sqrt{X}}). \end{equation*}

We now apply Poisson summation and change of variables to get \begin{multline*} \sum_{k \in \ri}     F_{n,1}(\frac{ dp^l (q+\frac{nw}{d}k)}{\sqrt{X}})=\\ \frac{X}{ (nw)^2p^{2l}} \sum_{m_1 \in \ri} e(\T(\frac{qm_1}{\frac{nw}{d}}))  \int_{-\infty}^{\infty}  F_{n,1}(t_1+\sqrt{D}t_2)e(-\T(\frac{\sqrt{X}m_1(t_1+\sqrt{D}t_2)}{p^lnw}))dt_1dt_2.\end{multline*} By the exact same integration by parts analysis we did for $T^{M}(n,d)$ we get for the integral associated to $m_1 \neq 0$ the bound  $O((m_1X)^{-M})$ for $M>0.$ Therefore, only the $m_1=0$ is non-negligible. However, for $m_1=0$ the $q$-sum from \eqref{eq:www1} is now a complete multiplicative character sum for a nontrivial character and is equal to zero. Summing over all of the nontrivial Dirichlet characters $\psi \mod (\frac{nw}{d})$ then is $O(X^{-M}), M>0.$ Likewise, by an analogous argument to the trivial character, we can get the estimate $O((nX)^{-M}), M>0.$

\end{proof}

\section{From exponential sums over a quadratic field to Kloosterman sums}\label{sec:zag}

In order to classify base change, one needs a bridge between a trace formula over the quadratic field $\mathbf{K}$ to a trace formula over $\Q.$ This formula was discovered in (\cite{Z}, Proposition 24). 

We need a few definitions first. Let $D=D_1D_2,$  then define $\psi(D_2)=(\frac{D_1}{D_2})\sqrt{D_2},$ where $(\frac{a}{b})$ is the quadratic character modulo $b.$ Next define
$$H_{b}(n,m)=\sum_{\substack{D=D_1D_2\\D_2|n \\ (b,D_2)=1}} \frac{\psi(D_2)}{D_2} H^{D_1}_{bD_1}(\frac{n}{D_2},m),$$ where $$H^{D_1}_c(n,m)=\frac{1}{c}(\frac{c}{D_2})S_{D_1}(n\overline{D_2},m,c),$$ and $$S_{D_1}(n\overline{D_2},m,c)=\sum_{d(c)^{*}} (\frac{d}{D_1})e(\frac{n\overline{D_2 d}+md}{c}).$$

\begin{prop}\label{zagier}Let $D$ be prime, then writing $n=Da,  a\in \N$ 
$$\sum_{\substack{r\in O_{\mathbf{K}}/(\frac{n}{\delta})\\rr'\equiv1(n)}} e(\frac{ r l+ r' l'}{n})=a\sum_{\substack{r\in \N \\ r |l \\ r|a}}  H_{a/r}(-\frac{ll'}{r^2},-1). $$
\end{prop} 

\begin{cor}\label{zagcor}
If $(l,D)=1$ then Proposition \ref{zagier} implies $$ 
 \sum_{\substack{r\in O_{\mathbf{K}}/(\frac{Da}{\delta})\\rr'\equiv1(Da)}} e(\frac{ r l+ r' l'}{Da})= \sum_{\substack{r\in \N \\ r^2 |\N(l) \\ r|a}}  r S_{D}(\frac{ll'}{r^2},1,\frac{Da}{r}).$$ 
\end{cor}

\section{The continuous spectrum}\label{sec:cssc}
\label{sec:cont}
We now need to compute the continuous spectrum contribution on the spectral side of the trace formula over the quadratic field $\K.$ Unfortunately, the normalization for the Fourier coefficients for the Eisenstein series is formidable.

Let $\mu_k=\frac{\pi k }{2 \log {\epsilon_0}}, k \in \Z.$ Define $$\omega_{\mu_k}(x):=|\frac{x}{x'}|^{i\mu_k},k \in \Z,$$ where  $\epsilon_0$ is the fundamental unit of our field $\K.$ From \cite{BMP2},
\begin{equation}\label{eq:contfourier}D_r(it,i\mu)= 2\pi \sqrt{\mathbb{N}(r)}\Gamma(1/2+i(t+\mu)) \Gamma(1/2+i(t+\mu'))  \times  \end{equation} $$\left[\frac{ \pi^{i(t+\mu)}}{\N(r)^{1/2+it}\omega_{\mu_k}(r)\Gamma(1/2+i(t+\mu)) \Gamma(1/2+i(t+\mu'))}\sum_{(c)} \frac{\omega_{\mu_k}^2(c) S(r,0,c)}{\mathbb{N}(c)^{1+2it}}\right],$$ where $S(r,0,c)=\sum_{x(c)^{*}}e(xr/c).$  Here also $\mu$ is in the lattice $Tr_{\mathbf{K}/\Q}(x)=0,$ such that $|\epsilon|^{i\mu}=1,$ for all $\epsilon \in \ri^{*}.$ Clearly $\mu_k, k \in \Z$ is in this lattice.

%Using the functional equation of the gamma function \begin{equation}\label{eq:gamm}\Gamma(1/2+it)\Gamma(1/2-it)=\frac{\pi}{\cosh{\pi t}^{-1}\end{equation}, 

Simplifying we have $$D_r(it,i\mu_k)= \frac{2\pi \pi^{i(t+\mu_k)}}{ N(r)^{it}\omega_{\mu_k}^2(r)}\sum_{(c)} \frac{\omega_{\mu_k}^2(c) S(r,0,c)}{\mathbb{N}(c)^{1+2it}}.$$
%Here also $\mu$ is in the lattice $Tr_{\mathbf{K}/\Q}(x)=0,$ such that $|\epsilon|^{i\mu}=1,$ for all $\epsilon \in \ri^{*}.$ In our case, $\mu_k=\frac{\pi k }{2 \log {\epsilon_0}}, k \in \Z.$ 

The continuous spectrum contribution is \begin{equation}\label{eq:cco}  CSC^{\K}_{\rho.\xi}:=\frac{1}{ L(1,\chi_D)} \sum_{k \in \Z} \int_{-\infty}^{\infty}h(V_1,t+\mu_k)h(V_2,t+\mu_k')D_{\rho}(it,i\mu_k)\overline{D_{\xi}(it,i\mu_k)}dt.\end{equation} We note the term $ \frac{1}{L(1,\chi_D)},$ is mentioned in \cite{J}, but buried as a constant in \cite{BMP1}. 

We prove in this section

%\begin{prop}\label{propcsc}
%\begin{multline*}\label{eq:prpp} \frac{1}{X} \sum_{\mm} \sum_{n} g(\mm^2 n/X) CSC_{n,l}=\\ 2\pi  \bigg(\sum_{\substack{r\in \N \\  r^2 | \N(l)}}   \sum_{\pi_{D,\mu_k}} h(V_1*V_2,\nu_{\pi_{D,\mu}}){c_{\frac{ll'}{r^2}}(\pi_{D,\mu})}{\overline{c_1(\pi_{D,\mu})}}+\\  \sum_{\substack{r\in \N \\  r^2 | \N(l)}}\frac{1}{4\pi} \int_{-\infty}^{\infty} h(V_1*V_2,t)D_{\frac{ll'}{r^2}} (it,0) \overline{D_{1} (it,0)}dt\bigg)\bigg(1+o(1)\bigg)+ O(X^{-M}). \end{multline*} for any integer $M \geq 0.$
%\end{prop}

\begin{prop}\label{propcsc}
If $$\sigma_{l,\omega_{\mu_k}}(n):=\sum_{a|n} \omega_{\mu_k}(a)\mathbb{N}(a)^l,$$ then we have \begin{align}\label{eq:prpp}  \sum_{\mm,n \in \N}  g(\mm^2 n/X) CSC^{\K}_{n,\xi}= \frac{1}{ L(1,\chi_D)} \sum_{k \in \Z}\frac{\pi^3 h(V_1,\mu_k)h(V_2,\mu_k) \sigma_{0,\omega_{\mu_k}^2}(\xi)}{ \omega_{\mu_k}(\xi) L(1,\chi_D)  \overline{L(1,\omega_{\mu_k}^2)}} +\\ 4\pi^2 \int_{-\infty}^{\infty} h(V_1,t)h(V_2,t)\frac{\sigma_{-2it,0}(\xi)\mathbb{N}(\xi)^{it}  }{   |L(1+2it,\chi_{D})|^2}dt +O(X^{-M}),\end{align} for any integer $M \geq 0.$
\end{prop}

 We aim first to simplify our normalized coefficients. Fix a character $\omega_{\mu_k}.$  Using Mobius inversion $$ \sum_{(c)} \frac{\omega_{\mu_k}^2(c)S(r,0,c)}{\mathbb{N}(c)^{1+2it}}= \frac{ \sigma_{2it,\omega_{\mu_k}^2}(r)}{ L(1+2it,\omega_{\mu_k}^2)}.$$ Therefore, \begin{equation}\label{eq:contmob}D_r(it,i\mu_k)= \frac{ 2\pi \pi^{i(t+\mu)}}{ N(r)^{it}\omega_{\mu_k}(r)}\frac{ \sigma_{2it,\omega_{\mu_k}^2}(r)}{ L(1+2it,\omega_{\mu_k}^2)}.\end{equation}
 
Recall $G(s)=\int_0^\infty g(x) x^{s-1}dx, G(1)=1.$ We collect the $n$ and $\mm$ terms together from \eqref{eq:cco} and use Mellin Inversion  to get 
\begin{equation}\label{eq:div}
\sum_{\mm} \sum_{n} g(\mm^2 n/X)\frac{\sigma_{2it,\omega_{\mu_k}^2}(n)}{\mathbb{N}(n)^{it}}=\frac{1}{2\pi i}\int_{(\sigma)}G(s)X^{s} \zeta(2s)L_{\omega_{\mu_k},t}(s)ds,  
\end{equation} for $\sigma>0$ sufficiently large. Here $$L_{\omega,t}(s)=\sum_{n=1}^\infty \frac{\sigma_{2it,\omega^2}(n)}{n^{s+2it}},$$  is the Asai $L$-function associated to the Eisenstein series over the quadratic field $\K.$  Using the multiplicativity of the divisor function $L_{\omega,t}(s)$ can be written as an Euler product. Specifically, the Dirichlet series equals for $p$ split, $p=\pa \pb,$ $$\frac{(1-\frac{1}{p^{2s}})}{(1-\frac{\omega^2(\pa)}{p^{s}})(1-\frac{\omega^2(\pb)}{p^{s}})(1-\frac{1}{p^{s+2it}})(1-\frac{1}{p^{s-2it}})}.$$ 
 
 For $p$ inert, $\chi_D(p)=-1$ and $\omega^2(p)=1,$ so we get $$\frac{1}{(1-\frac{1}{p^{s+2it}})(1-\frac{1}{p^{s-2it}})}=\frac{(1-\frac{1}{p^{2s}})}{(1-\frac{1}{p^{s}})(1-\frac{\chi_{D}(p) }{p^{s}})(1-\frac{1}{p^{s+2it}})(1-\frac{1}{p^{s-2it}})}.$$ 
 
 For $p=\pa^2$ ramified, $\chi_D(\pa)=0$ and $\omega^2(\pa)=1$ and so the corresponding piece of the Euler product is   $$\frac{(1+\frac{1}{p^{s}})}{(1-\frac{1}{p^{s+2it}})(1-\frac{1}{p^{s-2it}})}=\frac{(1-\frac{1}{p^{2s}})}{(1-\frac{1}{p^{s}})(1-\frac{\chi_{D}(p) }{p^{s}})(1-\frac{1}{p^{s+2it}})(1-\frac{1}{p^{s-2it}})}.$$
 
 So $$L_{\omega_k,t}(s)=\frac{L(s,\omega_{\mu_k}^2)\zeta(s+2it)\zeta(s-2it)}{\zeta(2s)}.$$
 %So \begin{equation}\label{eq:elc} L(s)=\frac{(1-\frac{1}{p^{2s}})}{(1-\frac{\omega(p)}{p^{s}})(1-\frac{\chi_{D}(p) }{p^{s}})(1-\frac{1}{p^{s+2it}})(1-\frac{1}{p^{s-2it}})}.\end{equation}

There are now 2 cases of analysis for this L-function $L_{\omega_{\mu_k},t}(s):$  $\omega=\omega_{\mu_k}$ having $k\neq 0$ or $k=0.$ 

\subsection{$\bf{k \neq 0}$}

Fix a $k \neq 0,$ then $ L_{\omega_{\mu_k},t}(s)$ only has poles at $s=1\pm 2it.$ We also have a $\zeta(2s)$ coming from the Euler product above which will cancel with the $\zeta(2s)$ in \eqref{eq:div}.   
 
 Thus, shifting the contour $\sigma$ of the integral to $-M$ for $M\geq 2$ will pick up a simple pole at $s=1+2it$ (as well as a pole at $1-2it$) with residue $$  X^{1+2it} \frac{\zeta(1+4it)L(1+2it,\omega_{\mu_k}^2)}{\zeta(2(1+2it))}$$  (an analogous term exists for the pole at $s=1-2it$).

 Shifting the contour we get \begin{multline}\label{eq:dcomp} \frac{1}{L(1,\chi_D)X}\sum_{n,\mm} g(\mm^2 n/X) \int_{-\infty}^{\infty}h(V_1,t+\mu_k)h(V_2,t+\mu_k') D_n(it,i\mu_k)\overline{D_{\xi}(it,i\mu_k)}dt =\\  \frac{4 \pi^2}{L(1,\chi_D)X} \int_{-\infty}^{\infty}h(V_1,t+\mu_k)h(V_2,t+\mu_k') \times \\ \left(\frac{\sigma_{-2it,\omega_{\mu_k}^2}(\xi)\mathbb{N}(\xi)^{it}}{ \omega_{\mu_k}(\xi) |L(1+2it,\omega_{\mu_k}^2)|^2} \right) \times \\
   X^{2it} \zeta(1+4it)L(1+2it,\omega_{\mu_k}^2)dt + O(X^{-M}). \end{multline} 
 
 %We note the term $ \frac{1}{L(1,\theta_D)},$ is mentioned in \cite{J}, but buried as a constant in \cite{BMP1}. 
   
Here $\zeta(1+4it)$ has a pole at $t=0$ and to understand this we use the following lemma. 

\begin{lemma}\label{ve}
Let $H$ be a differentiable function in $L^{1}(\mathbb{R})$ with $b>0$ then
$$PV \int_{-\infty}^{\infty} H(x)e^{2\pi ibx}\frac{dx}{x}=\pi i H(0).$$  
\end{lemma}

\begin{proof}
This is just a statement of the Cauchy principal value theory for a Fourier integral  \cite{BTVV}. 
\end{proof}

Let \begin{equation}F(t)=\frac{4 \pi^2}{L(1,\chi_D)} h(V_1,t+\mu_k)h(V_2,t+\mu_k')\left(\frac{\sigma_{-2it,\omega_{\mu_k}^2}(\xi)\mathbb{N}(\xi)^{it}}{  \omega_{\mu_k}(\xi) |L(1+2it,\omega_{\mu_k}^2)|^2} \right) L(1+2it,\omega_{\mu_k}^2).\end{equation} We can ensure we can use Lemma \ref{ve} on $F(t)$ by using specific decaying properties of $h(V_i,t), i=1.2$ from Proposition \ref{venkm2}. This is analogous to the decay properties needed to prove Lemma 10 of \cite{V}.

%the integral of \eqref{eq:dcomp} can be reduced to $$\int_{-\infty}^{\infty} F(t)e^{it\log X}\frac{dt}{t}.$$

Using $\mu_k'=-\mu_k,$ the residue of $\zeta(1+4it)$ is $\frac{-i}{4}.$ As  $h(V_j,t)$ is a real valued function, we apply the lemma for $b=\log X>0$ to \eqref{eq:dcomp} to obtain

\begin{equation}\label{eq:kcomp}  \frac{ \pi^3}{L(1,\chi_D)} h(V_1,\mu)h(V_2,\mu)\left(\frac{\sigma_{0,\omega_{\mu_k}^2}(\xi)}{ \omega_{\mu_k}(\xi) \overline{ L(1,\omega_{\mu_k}^2)}} \right).
\end{equation}

%\begin{remark}
%We point out in this case we do take the limit as $X \to \infty$ as writing the integral $\int_{-\log X}^{\log X}...$ everywhere would be tedious. The reader should take away from this section that there is a main term (depending on X) and an error term of size $O(X^{-M}),$ which as $X \to \infty$ leaves the term \eqref{eq:kcomp}. 
%\end{remark}

 Therefore, for $k\neq0$  \begin{equation}\label{eq:knzz} \frac{ \pi^3 h(V_1,\mu_k)h(V_2,\mu_k) \sigma_{0,\omega_{\mu_k}^2}(\xi)}{ \omega_{\mu_k}(\xi) L(1,\chi_D)  \overline{L(1,\omega_{\mu_k}^2)}} .\end{equation}
 
Using properties of $\omega_{\mu_k}, k \neq 0$  $$\frac{\sigma_{0,\omega_{\mu_k}^2}(\xi)}{ \omega_{\mu_k}(\xi) \overline{L(1,\omega_{\mu_k}^2)}}=\frac{\sigma_{0,\omega_{\mu_k}^2}(\xi)}{ \omega_{\mu_k}(\xi) L(1,\omega_{\mu_k}^2)},$$ and we have the first part of Proposition \ref{propcsc}.

\subsection{$\bf{k=0}$}

The calculations in this case are analogous to the previous section, however here we have $$L(s)=\frac{\zeta(s+2it)\zeta(s-2it)\zeta(s)L(s,\chi_D)}{\zeta(2s)}.$$ Again the $\zeta(2s)$ cancels with the $\mm$-sum, and the poles of $L(s)$ are at $s=1, 1\pm 2it.$ However, the poles at $s=1\pm 2it,$ cancel each other as they residues $\pm 4i,$ respectively.  
 
 We note $$\frac{1}{|L(1+2it,\omega^2)|^2}=\frac{1}{|\zeta(1+2it)L(1+2it,\chi_D)|^2}$$ found  in \eqref{eq:contmob} after expanding $D_n(it,0)\overline{D_{\xi}(it,0)}.$ So for $k=0,$ the left hand side of Proposition \ref{propcsc}  after a contour shift  to $-M, M\geq 0,$ similar to \eqref{eq:div}, equals

\begin{multline}\label{eq:3res} \frac{4\pi^2}{L(1,\chi_D) X}  \int_{-\infty}^{\infty} h(V_1,t)h(V_2,t) \left(\frac{\sigma_{-2it,0}(\xi)\mathbb{N}(\xi)^{it}}{ |\zeta(1+2it) L(1+2it,\chi_{D})|^2} \right)\times \\ \bigg[ X|\zeta(1+2it)|^2L(1,\chi_{D})G(1) +X^{1+2it} \zeta(1+4it)\zeta(1+2it)L(1+2it,\chi_D)G(1+2it) \\ + X^{1-2it} \zeta(1-4it)\zeta(1-2it)L(1-2it,\chi_D)G(1-2it)\bigg]+O(X^{-M}).
\end{multline}

We deal with the terms containing $X^{1\pm2it}$ first. Take the term with $X^{1+2it},$ the other term will be analogous. This term is after some simplifications, \begin{equation}\label{eq:2ndres}    \frac{4\pi^2}{L(1,\chi_D) }\int_{-\infty}^{\infty} h(V_1,t)h(V_2,t) \left(\frac{\sigma_{-2it,0}(\xi)\mathbb{N}(\xi)^{it} X^{2it} \zeta(1+4it)}{ \zeta(1-2it) L(1-2it,\chi_{D})} dt\right)
\end{equation} Now $\frac{\zeta(1+4it)}{\zeta(1-2it)L(1-2it,\chi_{D})}$ is analytic for $t \in \R,$ and using Proposition \eqref{venkm2}, we can assure that $$F(t):= h(V_1,t)h(V_2,t) \left(\frac{\sigma_{-2it,0}(\xi)\mathbb{N}(\xi)^{it} \zeta(1+4it)}{ \zeta(1-2it) L(1-2it,\chi_{D})} \right)$$ is analytic in the complex variable $t.$ By an application of the Paley-Wiener theorem,  we know $\hat{F}(k)=\int_{-\infty}^{\infty} F(t)e(kt)dt$ has compact support. So taking $k=\log X$ large enough to be outside the support of $\hat{F},$ we have \eqref{eq:2ndres} is equal to zero.

The term in \eqref{eq:3res} coming from the pole at $s=1$ can easily be simplified to 
\begin{equation}\label{eq:1stres}  4\pi^2 \int_{-\infty}^{\infty} h(V_1,t)h(V_2,t) \left(\frac{\sigma_{-2it,0}(\xi)\mathbb{N}(\xi)^{it}  }{   |L(1+2it,\chi_{D})|^2}dt \right).
\end{equation}

Putting together \eqref{eq:knzz}, \eqref{eq:1stres}, and the associated error $O(X^{-M})$ we have completed Proposition \ref{propcsc}.

\section{Putting it all together} \label{sec:putti}

Incorporating Proposition \ref{m000}, Corollary \ref{cor12}, and Corollary \ref{zagcor}, we have 
\begin{equation}
(L)=\frac{1}{X} \sum_{\mm} \sum_n g(\mm^2n/X)\big(\sum_{\Pi \neq \mathbf{1}}
h(V,\nu_{\Pi})c_{n}(\Pi)\overline{c_{l}(\Pi)}+
CSC^{\K}_{n,l}\big)=
\end{equation}
 
$$  \delta(l,l')\int_0^\infty \int_0^{\infty} V_1(x)V_2(x)\frac{dx}{x} +$$

$$ +  2\sum_{a=1}^\infty \frac{1}{Da}
\sum_{\substack{r\in \N \\ r^2 |\N(l) \\ r|a}} r S_{D}(\frac{l l'}{r^2},1,\frac{Da}{r})\int_{0}^{\infty}
\int_{0}^{\infty} H_{Da,1}(x,y)dxdy+O(X^{-M}).$$ The $2$ in front of the Kloosterman sum is counting $\{Da, -Da\}$ from the Poisson summation. We also mention here that we used Proposition \ref{limitcalc} (writing $n=Da$ by Lemma \ref{nDa}) to get $$\sum_{a=1}^\infty O((Da)^{-M}X^{-M})=O(X^{-M}),$$ for $M$ large. Our aim now is to show $(L)$ is a geometric side of a Kuznetsov trace formiula for representations $\pi_D.$

%We now follow section 9 of \cite{H} closely. 
%Define
%\begin{equation} V_1*V_2(z):=  \int_{-\infty}^{\infty}
%\int_{-\infty}^{\infty}
%  \exp \left(z\frac{i}{2}(
% \frac{x}{y}+\frac{y}{x})\right) \exp \left( (\frac{1}{z})\frac{8\pi^2i}{xy}\right)\times \end{equation} $$V(\frac{4\pi
%}{x})  W(\frac{4\pi}{ y})\frac{dx}{x} \frac{dy}{y}.
%$$
 Recalling Definition \ref{hnn} and the dependence on $l,$ as well as the definition of the convolution operation for the Bessel transform in Section \ref{sec:bccv}, we have the equality 
 \begin{multline}\label{eq:ST}
\int_{0}^{\infty}\int_{0}^{\infty} H_{Da,1}(x,y)dxdy =\\ \int_{-\infty}^{\infty}e(\frac{-1}{Da}(\frac{xl'}{y}+\frac{yl}{x}))e(\frac{Da}{xyD}) V_1(\frac{4\pi\sqrt{l}}{x})  V_2(\frac{4\pi\sqrt{l'}}{y})\frac{dx}{x} \frac{dy}{y}=\\
\int_{-\infty}^{\infty}e \left(-\frac{\sqrt{ll'}}{Da}(\frac{x}{y}+\frac{y}{x})\right) e \left( \frac{Da}{\sqrt{ll'}}\frac{1}{xy}\right)\times V(\frac{4\pi}{x})  W(\frac{4\pi}{ y})\frac{dx}{x} \frac{dy}{y}=V_1*V_2(\frac{4\pi \sqrt{ll'}}{Da}).\end{multline}

%D \left(V_1*V_2(\frac{\sqrt{l l'}}{Da})\right).

%The $D$ factor comes from a change of variables $t \to D t$ in the definition of $I_m(n,x,1)$ inside of Definition \ref{hnn}.

%Now Theorem $3.1$ of \cite{H} states 
%\begin{theorem} \label{main theorem} For all $V, W \in C_0^{\infty}(\R^{+})$ , $h(V*W,t)=C_{t} h(V,t)h(W,t),$ where $C_{t}=2\pi$ for $t$ an even integer, and $C_{t}=\pi$ for $t$ purely imaginary.
%\end{theorem}

 %Using Theorem \ref{main theorem} from Section \ref{sec:bccv} and the Sears-Titchmarsh inversion formula for $V_1*V_2$ \cite{H},
%we have \begin{equation}\label{gg}  V_1*V_2(z) = 4\pi \left(\int_0^{\infty} h(V_1*V_2,t) \text{tanh}(\pi t) B_{2it}(z) tdt +
%\sum_{2k>0,k\in \N} (k-1)J_{k-1}(z)h(V_1*V_2,k)\right),\end{equation}

From \cite{H} we need the following Proposition.
\begin{prop}\label{pr5} \begin{align}
  \int_0^{\infty} V_1(x)V_2(x)
\frac{dx}{x} =  2\left(\int_{-\infty}^\infty
h(V_1*V_2,t)\tanh(\pi t) t dt  +  \sum_{2k>0,k\in \N}
(k-1)h(V_1*V_2,k) \right).
\end{align}

\end{prop}

 $(L)$ can be then rewritten as 
\begin{multline}\label{eq:biggie}
%\begin{split} 
(L)= 2\pi\bigg[\delta(l,l')\left(\int_{-\infty}^\infty
h(V_1*V_2,t)\tanh(\pi t) t dt   +  \sum_{2k>0,k\in \N}(k-1)h(V_1*V_2,k)\right) +\\    \sum_{a=1}^\infty \frac{1}{Da}
\sum_{\substack{r\in \N \\ r^2 | \N(l) \\ r|a}} r S_{D}(\frac{l l'}{r^2},1,\frac{Da}{r})(V_1*V_2)\big(\frac{4\pi \sqrt{l l'}}{Da}\big) \bigg].
%\end{split} 
\end{multline}

First we need to simplify the Kloosterman sums in \eqref{eq:biggie}. We invert the $r$- and $s$- sums giving 
\begin{multline}
\sum_{a=1}^\infty \frac{1}{Da}
\sum_{\substack{r\in \N \\ r^2 | \N(l) \\ r|a}} r S_{D}(\frac{l l'}{r^2},1,\frac{Da}{r}) (V_1*V_2)\big(\frac{4\pi \sqrt{l l'}}{Da}\big)= \\ 
\sum_{\substack{r\in \N \\ r^2 | \N(l)} } \sum_{a=1}^\infty \frac{1}{Da} S_{D}(\frac{l l'}{r^2},1,Da)(V_1*V_2)\big(\frac{4\pi \sqrt{l l'}}{Da}\big)=\\ 
\sum_{\substack{r\in \N \\  r^2 | \N(l)} } \sum_{a=1}^\infty \frac{1}{Da} S_{D}(\frac{l l'}{r^2},1,Da)(V_1*V_2)\big(\frac{4\pi \sqrt{l l'}}{Da}\big).
\end{multline}

%We rewrite the holomorphic side of the trace formula as well in this fashion in \eqref{eq:biggie}. Using the spectral sides of the Kuznetsov and Petersson trace formula we get $(L)$ equals,  

%\begin{multline}\label{eq:cscsq}
%2\pi \sum_{\substack{r\in \N \\  r | l} }  \bigg[  \delta(l,l')\int_{-\infty}^\infty M
%(t)\tanh(\pi t) t dt + \sum_{a=1}^\infty \frac{1}{Da} S_{D}(\frac{l l'}{r^2},1,Da)\times \\  \int_0^{\infty} M(t) \tanh
%(\pi t) B_{2it}(\frac{4\pi\sqrt{\frac{ll'}{r^2}}}{Da}) tdt  + \\   \delta(l,l') \sum_{2k>0,k\in \N}(k-1)M(k) +  \sum_{a=1}^\infty \frac{1}{Da} S_{D}(\frac{l l'}{r^2},1,Da)\sum_{2k>0,k\in \N} (k-1)J_{k-1}(\frac{4\pi\sqrt{\frac{ll'}{r^2}}}{Da})M(k) \bigg] = \\   =2\pi \sum_{\substack{r\in \N \\  r | l} }  \bigg( \sum_{\pi_{D} } M(t_{\pi})
%{c_{\frac{ll'}{r^2}}(\pi_{D})}{\overline{c_1(\pi_{D})}}  +\\ \frac{1}{4
%\pi}\int_{-\infty}^\infty M(t)D_{\frac{ll'}{r^2}}(1/2+it,0)\overline{D_1(it,0)} dt +
%\sum_{\pi_{k,D} }
%M(k_{\pi}){c_{\frac{ll'}{r^2}}(\pi_{k,D})}{\overline{c_1(\pi_{k,D})}}\bigg).
%\end{multline}

Applying the Kuznetsov trace formula over $\Q$ for automorphic representations $\pi_{D},$ we have 
\begin{multline}\label{eq:xzx} 2\pi \delta(l,l')\sum_{\substack{r\in \N \\  r^2 | \N(l)} }\bigg[  \big(\int_{-\infty}^\infty
h(V_1*V_2,t)\tanh(\pi t) t dt  +  \sum_{2k>0,k\in \N}
(k-1)h(V_1*V_2,k)\big)+\\  \sum_{a=1}^\infty \frac{1}{Da} S_{D}(\frac{l l'}
{r^2},1,Da)(V_1*V_2)\big(\frac{4\pi \sqrt{l l'}}{Da}\big)\bigg]=\\ 2\pi \sum_{\substack{r\in \N \\   r^2 | \N(l)} } 
\big( \sum_{\pi_{D}} h(V_1 *V_2,\nu_{\pi_{D}})c_{\frac{ll'}{r^2}}(\pi_{D})
\overline{c_1(\pi_{D})} + CSC^{\Q}_{\frac{ll'}{r^2},1}\big).\end{multline}

We showed in Section \ref{sec:cont}, 

\begin{multline}\label{eq:cscs} \frac{1}{X} \sum_{\mm}\sum_{n} g(\mm^2 n/X) CSC^{\K}_{n,l}=4\pi^3 \sum_{k \in \Z}\frac{ h(V_1,\mu_k)h(V_2,\mu_k) \sigma_{0,\omega_{\mu_k}^2}(\xi)}{ \omega_{\mu_k}(\xi) L(1,\chi_D) L(1,\omega_{\mu_k}^2)} +\\ 4\pi^2 \int_{-\infty}^{\infty} h(V_1,t)h(V_2,t)\frac{\sigma_{-2it,0}(\xi)\mathbb{N}(\xi)^{it}  }{   |L(1+2it,\chi_{D})|^2}dt +O(X^{-M}),\end{multline} for $M >0.$ In the next section we show that the $k$-sum is associated to the automorphic representations $\pi_{D,\mu_k},k \in \Z, k \neq 0$ over $\Q$ while the integral from \eqref{eq:cscs} is equal to the continuous spectral term from \eqref{eq:xzx}. 

\subsection{Comparing Coefficients from $\K$ to $\Q$}

\indent In this section we compare the normalized Fourier coefficients of the continuous spectrum from  \eqref{eq:cscs} to those of analogous coefficients over $\Q.$ Further we show the terms $\sigma_{0,\omega_{\mu_k}^2}(\xi)$ from the $k$-sum are associated to Fourier coefficients of $\pi_{D,\mu_k}.$
%Further we show the $\mu$-sum in \eqref{eq:cscs} is associated with a sub sum ( "Theta functions") of the cuspidal spectrum in \eqref{eq:cscsq}.

Recall $\psi_{\mu_k}(k):=\sum_{\mathbb{N}(q)=k}\omega_{\mu_k}(q)$ from Section \ref{sec:theaa} are associated to representations $\pi_{D,\mu_k}.$

 \begin{lemma}\label{compa}

\hspace{.05cm}

 \begin{enumerate}

 \item{ $\mathbb{N}(l)^{it}\sigma_{-2it,0}(l)= \sum_{\substack{r\in \N \\  r^2 | \N(l)} }\tau_{it}(\frac{ll'}{r^2}),$}
 
\item{ $\omega_{\mu}(l)^{-1}\sigma_{0,\omega^2_{\mu}}(l)=\sum_{\substack{r\in \N \\  r^2 | \N(l)}} \psi_{\mu}(\frac{ll'}{r^2}).$} 
 \end{enumerate}
 \end{lemma}

\begin{proof}
By multiplicativity, we do this for a prime power. Notice $$\mathbb{N}(p^m)^{2it}\sigma_{-2it,0}(p^m)=\sum_{ab=p^m} \frac{\mathbb{N}(a)}{\mathbb{N}(b)}^{it}.$$ Remember $\tau_{it}(p^m)=\sum_{ab=p^m}\chi_{D}(a)(a/b)^{it},$ 
Assume $p$ inert prime, then we claim \begin{equation}\label{eq:pin} \sum_{\substack{r\in \N \\  r^2 | \N(l)} }\tau_{it}(\frac{ll'}{r^2})=\sum_{j=0}^m \tau_{it}(\frac{p^{2m}}{p^{2j}})=\sum_{ab=p^m} \frac{\mathbb{N}(a)}{\mathbb{N}(b)}^{it}.\end{equation}   Let $X=p^{it},$ then the LHS of \eqref{eq:pin} equals \begin{align*} \frac{1}{X^{2m}} \sum_{j=0}^m X^{2j} \sum_{k=0}^{2(m-j)} (-1)^k X^{2k} &=&   
 \frac{1}{X^{2m}} \sum_{j=0}^m X^{2j}\left(\frac{1 + (-1)^{2(m-j)}X^{2(2(m-j)+1)}}{1+X^2}\right) \\ &=& \frac{1}{X^{2m}(1+X^2)}\left(\frac{1-X^{2(m+1)}}{1-X^2} + \frac{X^{2(2m+1)}(1-X^{-2(m+1)})}{(1-X^{-2})}\right)\\ &=& \frac{1}{X^{2m}(1+X^2)}\left(\frac{1-X^{2(m+1)}}{1-X^2} + \frac{X^{2m+2}(X^{2(m+1)}-1)}{X^2(1-X^{-2})}\right) \\&=& \frac{1-X^{4(m+1)}}{X^{2m}(1-X^4)}.
\end{align*}
The RHS of \eqref{eq:pin} equals $$\sum_{ab=p^m} \frac{\mathbb{N}(a)}{\mathbb{N}(b)}^{it} = \sum_{k=0}^m X^{2(2k-m)}=\frac{1-X^{4(m+1)}}{X^{2m}(1-X^4)},$$ so we are done. 

For $p=p_1 p_2,$ we show for $l=p_1^m,$  \begin{equation}\label{eq:psplit} \sum_{\substack{r\in \N \\  r^2 | \N(l)} }\tau_{it}(\frac{ll'}{r^2})=\sum_{j=0}^m \tau_{it}(p^{m})=\sum_{ab=p_1^m} \frac{\mathbb{N}(a)}{\mathbb{N}(b)}^{it}.\end{equation} 
Let $X=p^{it},$ then it is clear both sides of \eqref{eq:psplit} equals $$X^{-2m} \sum_{j=0}^m X^{2j}.$$ For $p=\sqrt{D},$ the identity is as transparent as the split prime case.

The identity $$\omega_{\mu}(l)^{-1}\sigma_{0,\omega^2_{\mu}}(l)=\sum_{\substack{r\in \N \\ r^2 | \N(l)}} \psi_{\mu}(\frac{ll'}{r})$$ is proven in an identical fashion. Let us look at the most interesting case of $p_1^m$ where $p=p_1 p_2 \in \Z, p$ prime. The LHS equals $$   \frac{\sum_{k=0}^m  \omega^2_{\mu_k}(p_1^k)}{\omega_{\mu_k}(p_1^m)}=\sum_{k=0}^m  \omega_{\mu_k}(p_1^{2k-m}).$$ Now the $k=j, m-j$ terms are $$\omega_{\mu_k}(p_1^{-(m-2j)})+\omega_{\mu_k}(p_1^{m-2j})=\psi_{\omega_{\mu_k}}(p^{m-2j})$$ as $\omega_{\mu_k}(l)\omega_{\mu_k}(l')=1$ for any $l \in \ri.$ Summing over $j$ this equals $$\sum_{j=0}^m \psi_{\omega_{\mu_k}}(p^{m-2j})
=\sum_{j=0}^m \psi_{\omega_{\mu_k}}(\frac{p^{m}}{p^{2j}})=\sum_{\substack{r \in \N \\ r^2|\N(p_1^m)}} \psi_{\omega_{\mu_k}}(\frac{\N(p_1^m)}{r^2}).$$ This proves the identity in the split case and the inert and ramified cases are analogous.
\end{proof}

\subsection{Proof of Theorem \ref{theo}}\label{sec:thee}

\begin{prop}\{Proof of Theorem \ref{theo} (2)\}

For $l \in \ri,$ \begin{multline}\label{eq:almost}
 \frac{1}{X} \sum_{\mm} \sum_n g(\mm^2 n/X)CSC^{\K}_{n,l}=\\  2\pi\sum_{\substack{r\in \N \\  r^2 | \N(l)}} \bigg( \sum_{\pi_{D,\mu_k}}   h(V_1,\mu_k)h(V_2,\mu_k) c_{\frac{ll'}{r^2}}(\pi_{D,\mu_k}) \overline{c_{1}(\pi_{D,\mu_k})}+  CSC^{\Q}_{\frac{ll'}{r^2},1}\bigg)+O(X^{-M}) \end{multline} for any positive integer $M.$

\end{prop}

\begin{proof}
Using Lemma \ref{compa}, \eqref{eq:cscs} equals \begin{equation}
4\pi^3  \sum_{\substack{r\in \N \\  r^2 | \N(l)}} \sum_{k \in \Z, k \neq 0}  \frac{ h(V_1,\mu_k)h(V_2,\mu_k) \psi_{\mu_k}(\frac{ll'}{r^2})\overline{\psi_{\mu_k}(1)}}{L(1,\chi_D) L(1,\omega^2_{\mu_k})}+
\end{equation}
$$4\pi^2  \sum_{\substack{r\in \N \\  r^2 | \N(l)}} \int_{-\infty}^{\infty} \frac{h(V_1,t)h(V_2,t)\tau_{it}(\frac{ll'}{r^2})\overline{\tau_{it}(1)}}{ |L(1-2it,\chi_{D})|^2}+O(X^{-M}).$$

Let us look at the first sum. Incorporating the normalizations needed for the trace formula over $\Q$ and using the duplication formula of the Gamma functions, 
\begin{multline}\label{eq:tqq}4\pi^3  \sum_{\substack{r\in \N \\  r^2 | \N(l)}} \sum_{k \in \Z, k \neq 0}  \frac{ h(V_1,\mu_k)h(V_2,\mu_k) \psi_{\mu_k}(\frac{ll'}{r^2})\overline{\psi_{\mu_k}(1)}}{L(1,\chi_D) L(1,\omega^2_{\mu_k})}=\\2\pi ^2 \sum_{k \in \Z, k \neq 0}   h(V_1,\mu_k)h(V_2,\mu_k) \frac{ \Psi_{\mu_k}(\frac{ll'}{r^2})\overline{\Psi_{\mu_k}(1)}}{|\Gamma(\frac{1}{2}+i\mu_k)|^2L(1,\chi_D) L(1,\omega^2_{\mu_k})}=\\ 2\pi  \sum_{k \in \Z, k \neq 0}   h(V_1,\mu_k)h(V_2,\mu_k)   \Psi_{\mu_k}(\frac{ll'}{r^2})\overline{\Psi_{\mu_k}(1)} \frac{\cosh(\pi \mu_k)}{L(1,\chi_D) L(1,\omega^2_{\mu_k})}.\end{multline} Recall that the normalized Fourier coefficients $\Psi_{\mu_k}(r)$ are defined in Section \ref{sec:theaa}.

 If we want to associate this sum to a spectral sum of the Kuznetsov trace formula over $\Q,$ we must account for the orthonormalization on the spectral side of the trace formula over $\Q.$ For a automorphic representation $\pi_{D, \mu_k}$ the inner product is $$|| \pi_{D,\mu_k}||_2^2= \frac{L(1,sym^2(\psi_{\mu}))}{\cosh(\pi \mu_k)}= \frac{ L(1,\chi_D) L(1,\omega^2_{\mu_k})}{ \cosh(\pi _{\mu_k})}.$$ 

So \eqref{eq:tqq} equals $$2\pi  \sum_{\substack{r\in \N \\  r^2 | \N(l)}} \sum_{k \in \Z, k \neq 0}   h(V_1,\mu_k)h(V_2,\mu_k)  \frac{ \Psi_{\mu_k}(\frac{ll'}{r^2})\overline{\Psi_{\mu_k}(1)} }{||\pi_{D,\mu_k}||_2^2}.$$ This is exactly the spectral sum of the representations $\pi_{D,\mu_k}.$ Specifically this is $$ 2\pi  \sum_{\substack{r\in \N \\  r^2 | \N(l)}} \sum_{\pi_{D,\mu_k}}   h(V_1,\mu_k)h(V_2,\mu_k) c_{\frac{ll'}{r^2}}(\pi_{D,\mu_k}) \overline{c_{1}(\pi_{D,\mu_k})}.$$

Analogous to getting \eqref{eq:contmob}, the normalized Eisenstein series $n$-th Fourier coefficient over $\Q$ is  $$D_n(it,0)=\big( \sqrt{2\pi}n^{1/2}\Gamma(1/2+it) \pi^{it}\big) \frac{ \tau_{it}(n)}{n^{1/2+it} \Gamma(1/2+it) L(1+2it,\chi_{D})}= \frac{ \sqrt{2\pi} \pi^{it}}{ n^{it}}\frac{ \tau_{it}(n)}{ L(1+2it,\chi_{D})}.$$

So we can rewrite \begin{multline}4\pi^2  \sum_{\substack{r\in \N \\  r^2 | \N(l)}} \int_{-\infty}^{\infty} \frac{h(V_1,t)h(V_2,t)\tau_{it}(\frac{ll'}{r^2})\overline{\tau_{it}(1)}}{ |L(1-2it,\chi_{D})|^2}=\\ 2\pi \sum_{\substack{r\in \N \\  r^2 | \N(l)}} \int_{-\infty}^{\infty}h(V_1,t)h(V_2,t)D_{\frac{ll'}{r^2}}(it,0)D_{1}(it,0)dt =2\pi CSC^{\Q}_{\frac{ll'}{r^2},1}.\end{multline}

%Now it can be seen that the continuous spectrum part of \eqref{eq:cscsq} equals the LHS of \eqref{eq:xzx}. This completes Theorem \ref{theo} (2).

Using $h(V_1 *V_2, t)=h(V_1,t)h(V_2,t),$ we have proved for $l \in \ri,$
\begin{multline} \frac{1}{X} \sum_{\mm} \sum_n g(\mm^2 n/X)CSC^{\K}_{n,l}=\\  2\pi\sum_{\substack{r\in \N \\  r^2 | \N(l)}} \bigg( \sum_{\pi_{D,\mu_k}}   h(V_1,\mu_k)h(V_2,\mu_k) c_{\frac{ll'}{r^2}}(\pi_{D,\mu_k}) \overline{c_{1}(\pi_{D,\mu_k})}+  CSC^{\Q}_{\frac{ll'}{r^2},1}\bigg)+O(X^{-M}) \end{multline} for any positive integer $M.$ This completes Theorem \ref{theo} (2).

 \end{proof}

%Subtracting the continuous spectrums from both sides we get an equality of the discrete spectrum plus the $\mu$-sum.

Now using \eqref{eq:almost}, we can immediately imply from \eqref{eq:xzx} that  \begin{multline}\label{eq:cusp}
 \frac{1}{X} \sum_{\mm} \sum_n g(\mm^2 n/X)\sum_{\Pi \neq \mathbf{1}}
h(V,\nu_{\Pi})c_{n}(\Pi)\overline{c_{l}(\Pi)} =\\   2\pi\sum_{\substack{r\in \N \\  r^2 | \N(l)}}  \sum_{\pi_{D} \neq \pi_{D,\mu_k}} h(V_1*V_2,t_{\pi}){c_{\frac{ll'}{r^2}}(\pi_{D})}{\overline{c_1(\pi_{D})}}+O(X^{-M}).
\end{multline} for any integer $M \geq 0.$ This completes Theorem \ref{theo} (1).

\section{Hecke Algebra}\label{sec:hal}

We now get matching of the individual representations. 
We exploit the fact that each irreducible representation $\Pi$ has associated Fourier coefficients that obey Hecke relations: $$c_{n}
(\Pi)c_{m_1}(\Pi)=\sum_{r|(\nn,m_1)}c_{\frac{m_1\nn}{r^2}}(\Pi),$$ one gets 
\begin{equation}\sum_{\Pi \neq \mathbf{1}}
h(V,\nu_{\Pi})c_{\mu}(\Pi)c_{q}(\Pi)c_{\nu}(\Pi)= \sum_{r|(q,\nu)} \sum_{\Pi \neq \mathbf{1}}
h(V,\nu_{\Pi})c_{\mu}(\Pi)c_{\frac{q\nu}{r^2}}(\Pi).\end{equation}

Likewise, one gets \begin{multline}\label{eq:mulhecke}\sum_{\Pi \neq \mathbf{1}}
h(V,\nu_{\Pi})c_{\mu}(\Pi)\left[\prod_{i=1}^Nc_{q_i}(\Pi)\right]c_{\nu}(\Pi)=\\ \sum_{r_1|(q_1,\nu)}  
\sum_{r_2|(\frac{q_1\nu}{r_1^2},q_2)} \sum_{r_3|(\frac{q_1q_2\nu}{(r_1r_2)^2},q_3)}...\sum_{r_{N}|
\left(\frac{\prod_{i=1}^{N}q_i\nu}{\prod_{i}^{N}r_i^2},q_N\right)}  \sum_{\Pi \neq \mathbf{1}}
h(V,\nu_{\Pi})c_{\mu}(\Pi)c_{\frac{\prod_{i=1}^N q_i \nu}{\prod_{i=1}^N r_i^2}}(\Pi).\end{multline}

So one computes by using Theorem \ref{theo} and \eqref{eq:mulhecke} for any $q_i$ such that $(\prod_{i=1}^N q_i,D)=1,$ \begin{multline}\label{eq;theoh}
 \frac{1}{X} \sum_{\mm} \sum_n g(\mm^2 n/X)\sum_{\Pi \neq \mathbf{1}}
h(V,\nu_{\Pi})c_{n}(\Pi)\left[\prod_{i=1}^N c_{q_i}(\Pi)\right]{c_{l}(\Pi)}= \\ \sum_{r_1|(q_1,\nu)}  
\sum_{r_2|(\frac{q_1\nu}{r_1^2},q_2)} \sum_{r_3|(\frac{q_1q_2\nu}{(r_1r_2)^2},q_3)}...\\ \sum_{r_{N}|
\left(\frac{\prod_{i=1}^{N}q_i\nu}{\prod_{i}^{N}r_i^2},q_N\right)} \frac{1}{X} \sum_{\mm} \sum_n g
(\mm^2 n/X)\sum_{\Pi \neq \mathbf{1}}
h(V,\nu_{\Pi})c_{n}(\Pi)c_{\frac{\prod_{i=1}^N q_i \nu}{\prod_{i=1}^N r_i^2}}(\Pi)=\\ 
2\pi  \sum_{\substack{r\in \N \\  r  | \frac{\prod_{i=1}^N q_i \nu}{\prod_{i=1}^N r_i^2}}}   \sum_
{\pi_{D} \neq \pi_{D,\mu_k}} h(V_1*V_2,t_{\pi})c_{\frac{\mathbf{N}\left(\frac{\prod_{i=1}^N q_i \nu}{\prod_{i=1}^N r_i^2}\right)}{r^2}}(\pi_{D}){\overline{c_1(\pi_{D})}} + O(X^{-M})
\end{multline}

In particular, for any polynomial $F(c_{q_1},c_{q_2},...,c_{q_N})$ with complex coeffiicients, the equation \begin{equation}
 \frac{1}{X} \sum_{\mm} \sum_n g(\mm^2 n/X)\sum_{\Pi \neq \mathbf{1}}
h(V,\nu_{\Pi})c_{n}(\Pi)F(c_{q_1},c_{q_2},...,c_{q_N})\overline{c_{l}(\Pi)}
\end{equation} has a corresponding identity analogous to Theorem \ref{theo} over the forms of level $D$ on $\Q.$ Call the term corresponding to $F(c_{q_1},c_{q_2},...,c_{q_N})$ on the side of forms over $\Q,  T\circ F.$

With this freedom of choice of $F,$ if we were in a finite dimensional setting (e.g. the space of forms of a fixed weight $k$ or eigenvalue parameter $t_j$), we can choose such a polynomial $F$ to be zero on all but one single representation $\Pi.$

\subsection{Reduction to a finite dimensional setting}

Now we use Hypothesis \ref{stu} in the following propositions of \cite{V}:
\begin{prop}\label{venkm1}
Let $t_j$ be a discrete subset of $\mathbb{R}$ with $\{j: t_j \leq T\} \ll T^{r}$ for some r. Let, for each j, there be given a function $c_{X}(t_j)$ depending on $X$, so that $c_{X}(t_j) \ll t_{j}^{r'}$ for some $r'$- the implicit constant independent of $X$; similarly, for each $k$ odd, let there be given a function $c_{X}(k)$ depending on $X$ so that $c_{X}(k) \ll k^{r'}.$ Suppose that one has an equality 
\begin{equation}
\lim_{X \to \infty} \big(\sum_{j} c_{X}(t_j)h(V,t_j)+\sum_{k \text{ odd}} c_{X}(k)h(V,k)\big)=0
\end{equation}
for all $(h(V,t_j),h(V,k))$ that correspond via Sears-titchmarsh inversion to $V$. Then $\lim_{X \to \infty} c_{X}(t_j)$ exists for each $t_j$ and equals 0, and similarly the same holds for $\lim_{X \to \infty}c_{X}(k).$ This equality holds for all functions $h$ for which both sides converge.
\end{prop}  

\begin{prop}\label{venkm2}
Given $j_0 \in \mathbb{N}, \epsilon > 0$ and an integer $N >0,$ there is a $V$ of compact support so that $h(V,t_j)=1,$ and for all $j' \neq j_0, h(V,t_{j'}) \ll \epsilon(1+|t_{j'}|)^{-N},$ and for all $k$ odd, $h(V,k) \ll \epsilon k^{-N}.$ 

Given $k_0, \epsilon >0$ and an integer $N>0,$ there is a $V$ of compact support so that $h(V,k_0)=1, h(V,k) \ll \epsilon k ^{-N}$ for $k$ odd $k\neq k_0,$ and $h(V,t) \ll (1+|t|)^{-N}$ for all $\mathbb{R}.$ 
\end{prop}

To apply the above propositions we apply Hypothesis \ref{stu}, which again is an artifact of the functional equation of the Asai L-function. Let \begin{multline}c_{X}(t_j)=\frac{1}{X}  \sum_{\substack{\Pi \\ t_{\Pi}=\{t_j,t_k\}}}h(V_2,t_{k}) \sum_{\mm} \sum_n g(\mm^2 n/X)c_{n}(\Pi)F(c_{q_1},c_{q_2},...,c_{q_N})\overline{c_{l}(\Pi)}-\\ 2\pi (T\circ F) \sum_{\substack{r\in \N \\  r^2 | \N(l)}}  {c_{\frac{ll'}{r^2}}(\phi_{t_j,D})}{\overline{c_1(\phi_{t_j,D})}},\end{multline} for any polynomial $F \in \C[x_1,..,x_N].$

Then we can choose an $F$ to isolate $\Pi$ with archimedean parameter $t_{\Pi}=\{t_{j},t_k\},$ with $j$ fixed and $k$ can vary. The problem is reduced to \begin{multline}\label{eq:toom0}
\frac{1}{X}   \sum_{\substack{\Pi \\ t_{\Pi}=\{t_j,t_k\}}}h(V_2,t_{k}) \sum_{\mm} \sum_n g(\mm^2 n/X)c_{n}(\Pi)\overline{c_{l}(\Pi)}=\\ 2\pi \sum_{\substack{\pi_{t}\\ t=t_j}} (T\circ F) \sum_{\substack{r\in \N \\  r^2 | \N(l)}}  {c_{\frac{ll'}{r^2}}(\pi_{D})}{\overline{c_1(\pi_{D})}} +O(X^{-M}).
\end{multline}
We can apply the above propositions again for the sum over $t_k$ to reduce to the equality 
\begin{equation}\label{eq:toom}
\frac{1}{X}   \sum_{\mm} \sum_n g(\mm^2 n/X)c_{n}(\Pi)\overline{c_{l}(\Pi)}=2\pi \sum_{\substack{\phi_{t,D}\\ t=t_j}} (T\circ F) \sum_{\substack{r\in \N \\  r^2 | \N(l)}}  {c_{\frac{ll'}{r^2}}(\pi_{D})}{\overline{c_1(\pi_{D})}} +O(X^{-M}).
\end{equation}

It is clear if we choose to isolate $t_{\Pi}=\{t_j,t_k\}$ with $j\neq k$ then the sum of forms over $\Q$ is zero. So for now on we assume $t_{\Pi}=\{t_j,t_j\}.$ 

It is easy to check that \begin{equation}\label{eq:tfc}
(T\circ F)=F( \sum_{\substack{r\in \N \\  r | q_1}}c_{\frac{\mathbf{N}(q_1)}{r^2}}(\pi_{D}),  \sum_{\substack{r\in \N \\  r | q_2}}c_{\frac{\mathbf{N}(q_2)}{r^2}}(\pi_{D}),..., \sum_{\substack{r\in \N \\  r | q_N}}c_{\frac{\mathbf{N}(q_N)}{r^2}})(\pi_{D}).
\end{equation}

%Furthermore, also using the Hecke qualities of $c_n(\pi_{D}),$ it is easy to check, though laborious, that $$
%\alpha_l(\pi_D):=\sum_{\substack{r\in \N \\  r^2 | \N(l)}}c_{\frac{ll'}{r^2}}(\pi_D)$$ has the properties of a 
%Hecke eigenvalue for a representation $\Pi$ over $\K/\Q$ with archimedean parameter $t_{\Pi}=\{t_j,t_j\}.$ 

Therefore, with \eqref{eq:tfc} the polynomial $F$ kills all but at most one of the terms $\pi_{D}$ on the RHS of \eqref{eq:toom}. So there is exactly one term on the RHS of \eqref{eq:toom} and so we have  

%If it kills all of them, this implies by the LHS of \eqref{eq:toom} that if $\Pi \neq 0$ there exists an $l$ such that $c_l(\Pi)\neq 0,$ and $$ \frac{1}{X}\sum_{n,m} g(\frac{nm^2}{X})c_{n}(\Pi)=O(X^{-M}),$$ for any $M>0.$ 

\begin{equation}\label{eq:okok} \frac{1}{X}  \sum_{\mm} \sum_n g(\mm^2 n/X)c_{n}(\Pi)\overline{c_
{l}(\Pi)}=2\pi  \sum_{\substack{r\in \N \\  r^2 | \N(l)}}  {c_{\frac{ll'}{r^2}}(\pi_{D})} +O(X^{-M}),
\end{equation} 

Hence if we let $l=1,$ we have \begin{equation}\frac{1}{X}\sum_{n,m} g(\frac{nm^2}{X})c_{n}(\Pi)=2\pi +O(X^{-M}).\end{equation}

But since \eqref{eq:okok} holds for all $l \in \ri,(l,D)=1,$ we must have $$c_l(\Pi)=\sum_{\substack{r\in \N \\  r^2 | \N(l)}}  {c_{\frac{ll'}{r^2}}(\pi_{D})}.$$ This proves Corollary \ref{bc}.

Using a standard Euler product calculation the Fourier coefficient equality also implies the equality of the L-functions $$L(s,\Pi)=L(s,\pi_{D})L(s,\pi_{D} \otimes \chi_D).$$ This is the definition that $\Pi$ is a base change or that $\Pi=BC_{\K/\Q}(\pi_{D}).$ If we had isolated an archimedean parameter $t_{\Pi}=\{t_j,t_k\}$ with $j\neq k,$ then as stated there would be no matching forms over $\Q$ and therefore $\Pi$ would not be a base change. 

By Mellin inversion \begin{multline}\label{eq:acf}
\frac{1}{X}\sum_{n,m} g(\frac{nm^2}{X})c_{n}(\Pi)=\frac{1}{2\pi i}\int_{(\sigma)}G(s)X^{s-1}L(s,\Pi,Asai)ds=\\ G(1)\zeta(2)Res_{s=1}L(s,\Pi,Asai) +O(X^{-M}).
\end{multline}The residue is nonzero precisely when $\Pi$ is a base change from a representation $\pi_D$ over $\Q.$ The last equality comes from shifting the contour to the left to $-M.$ Since $M$ is arbitrary, $L(s,\Pi,Asai)$ has analytic continuation to the entire plane with at most a simple pole at $s=1.$ This concludes Corollary \ref{bc}.

\begin{remark}
Using the equality $$Res_{s=1} L(s,\Pi, Asai)=\int_{Z(\mathbb{A}_{\Q})GL_2(\Q)\setminus GL_2(\mathbb{A}_{\Q})}\phi_{\Pi}(g)dg$$ from \cite{F}, we clearly have $\Pi$ is distinguished if and only if $\Pi$ is a base change which is also equivalent to the Asai L-function having a simple pole at $s=1.$

\end{remark}

\end{document}